\documentclass[12pt]{amsart}

\usepackage{amsthm, amsmath, amsfonts, amscd, amssymb}
\usepackage{mathrsfs} 
\usepackage[breaklinks,colorlinks,citecolor=teal,linkcolor=teal,urlcolor=teal,hyperindex]{hyperref}
\usepackage{color}
\usepackage{tikz}
\usetikzlibrary{arrows} \usetikzlibrary{patterns}
\usepackage[all]{xy}
\usepackage[top=2.5cm,bottom=2.5cm,left=2.6cm,right=2.6cm]{geometry}

\newtheorem{theorem}{Theorem}[section]
\newtheorem{lemma}[theorem]{Lemma}
\newtheorem{prop}[theorem]{Proposition}
\newtheorem{coro}[theorem]{Corollary}

\newtheorem{defn}[theorem]{Definition}

\theoremstyle{definition}

\newtheorem{remark}[theorem]{Remark}

\numberwithin{equation}{section}

\def \fk{\mathfrak}
\def \msf{\mathsf}
\def \mb{\mathbb}
\def \mc{\mathcal}
\def \scr{\mathscr}
\def \inv{^{-1}}
\def \0{\infty}

\def \cplane{\mathbb{C}}
\def \integer{\mathbb{Z}}

\def \ration{\mathbb{Q}}
\def \qq{\quad}

\def \P{\mb P}
\def \rto{\rightarrow}
\def \hrto{\hookrightarrow}
\def \rrto{\rightrightarrows}

\def \vdimc{\mathrm{vdim}_\cplane}

\def \co{\colon\thinspace}
\def \C{{\msf C}}
\def \N{{\msf N}}
\def \X{{\msf X}}
\def \Y{{\msf Y}}
\def \L{{\msf L}}
\def \D{\msf D}
\def \I{{\msf I}}
\def \f{{\msf f}}
\def \T{\scr T}
\def \cN{{\mc N}}
\def \M{{\overline{\scr M}}}

\def \bl{{\Big\langle}}
\def \br{{\Big\rangle}}

\def \aut{{\mathrm{Aut}}}
\def \vir{{\mathrm{vir}}}
\def \age{{\mathrm{age}}}
\def \cont{{\mathrm{Cont}}}
\def \node{{\mathrm{node}}}
\def \orb{{\mathrm{orb}}}
\def \ord{{\mathrm{ord}}}
\def \tV{{\mathrm{V}}}
\def \tE{{\mathrm{E}}}
\def \tL{{\mathrm{L}}}
\def \tH{{\mathrm{H}}}
\def \fkp{{\fk p}}
\def \fkq{{\fk q}}

\allowdisplaybreaks

\begin{document}

\title{Double ramification cycles with orbifold targets}

\author{Bohui Chen}
\address{Department of Mathematics and Yangtz center of Mathematics,
Sichuan University, Chengdu 610064, China}
\email{chenbohui@scu.edu.cn}

\author{Cheng-Yong Du}
\address{School of mathematics and V.\thinspace C. \& V.\thinspace R. Key Lab, Sichuan Normal University, Chengdu 610068, China}
\email{cyd9966@hotmail.com}

\author{Rui Wang}
\address{Department of Mathematics, University of California, Berkeley, CA 94720-3840 USA}
\email{ruiwang@berkeley.edu}


\date{}


\begin{abstract}
In this paper, we consider double ramification cycles with orbifold targets. An explicit formula for double ramification cycles with orbifold targets, which is parallel to and generalizes the one known for the smooth case, is provided. Some applications for orbifold Gromov--Witten theory are also included.

\end{abstract}

\maketitle


\section{Introduction}

Rubber invariant appears in the relative Gromov--Witten theory and it plays an important role in many circumstances. Recently, people observe that it is closely related to the so called double ramification cycles (abbreviated as DR-cycles). DR-cycles for target space $X$ as a point and a smooth manifold were studied and an elegant formula for DR-cycles was obtained by
Janda--Pandharipande--Pixton--Zvonkine (\cite{Janda-Pandharipande-Pixton-Zvonkine2017,Janda-Pandharipande-Pixton-Zvonkine2018}). This is a break-through in this subject and it has many interesting applications, see for example \cite{Janda-Pandharipande-Pixton-Zvonkine2017}, \cite{Janda-Pandharipande-Pixton-Zvonkine2018}, \cite{Farkas-Pandharipande2016}, \cite{Tseng-You2018b}, \cite{Fan-Wu-You2020}, \cite{Fan-Wu-You2019}, etc. Such a formula for DR-cycle with orbifold targets $[pt/G]$, where $G$ is a finite group, was obtained by Tseng--You (\cite{Tseng-You2016b}). In this paper based on the relative Gromov--Witten theory for orbifolds developed by Chen--Li--Sun--Zhao (\cite{Chen-Li-Sun-Zhao2011})
and Abramovich--Fantechi (\cite{Abramovich-Fantechi2016}), we are able to develop a parallel formula for DR-cycles when $X$ is a general orbifold (cf. Theorem \ref{T formula-of-DR}). The proof of the formula relies on the polynomiality of certain twisted Gromov--Witten invariants of roots of line bundles, which was discovered by Pixton (\cite{Janda-Pandharipande-Pixton-Zvonkine2017}). In this paper, we verify such a polynomiality property for the case of orbifold line bundles (cf. Theorem \ref{T polynomiality}). As an application we study the relation between relative orbifold Gromov--Witten invariants and absolute orbifold Gromov--Witten invariants of root constructions, generalizing the results in \cite{Tseng-You2018b}.

We next explain our results explicitly. Consider an orbifold line bundle $\L\rto\D=(D^1\rrto D^0)$ with representation $\rho\co D^1\rto U(1)$. Let $\sqrt[r]\L$ be its $r$-th root, and $(\sqrt[r]\D)_\rho$ be the corresponding $r$-th root gerbe over $\D$, which is a banded $\integer_r$-gerbe over $\D$. Let $\Gamma=(\msf g,\vec g,\beta)$ be a topological data for $\D$ with $\vec g=([g_1],\ldots,[g_n])$ and let the vector $A=(a_1,\ldots,a_n)\in\ration^n$ be $\rho$-admissible for $\Gamma$ (i.e., for $\vec g$) in the sense of \eqref{E Gamma-admissible}, then we have a topological data
\[
\Gamma_{A,r,\rho}=\Upsilon_{r,\rho}(\Gamma,A)
\]
for $(\sqrt[r]\D)_\rho$, see Definition \ref{D lift-Gamma}. Let $\M_\Gamma(\D)$ and $\M_{\Gamma_{A,r,\rho}}((\sqrt[r]\D)_\rho)$ be the moduli space of stable maps to $\D$ and $(\sqrt[r]\D)_\rho$ of topological type $\Gamma$ and $\Gamma_{A,r,\rho}$ respectively. Let $\tilde\M_{\Gamma_{A,r,\rho}}((\sqrt[r]\D)_\rho)$ be the weighted blowup of $\M_{\Gamma_{A,r,\rho}}((\sqrt[r]\D)_\rho)$ along the locus of nodal curves with weight given by the orders of orbifold structures over nodal points. Then we have a universal curve $\pi\co\tilde{\scr C}\rto \tilde \M_{\Gamma_{A,r,\rho}}((\sqrt[r]\D)_\rho)$ equipped with a universal map $\f\co\tilde{\scr C}\rto (\sqrt[r]\D)_\rho$. It induces a $K$-bundle $(\sqrt[r]\L)_{\Gamma_{A,r,\rho}}:=\mc R\pi_*\f^*\sqrt[r]\L$ over $\tilde\M_{\Gamma_{A,r,\rho}}((\sqrt[r]\D)_\rho)$. Let $\tau$ be the composition $\tilde\M_{\Gamma_{A,r,\rho}}((\sqrt[r]\D)_\rho)\xrightarrow{\pi} \M_{\Gamma_{A,r,\rho}}((\sqrt[r]\D)_\rho)\xrightarrow{\epsilon}\M_\Gamma(\D)$.
Our first result concerns the cycles
\[
\tau_*(c_d(-(\sqrt[r]\L)_{\Gamma_{A,r,\rho}}) \cap[\tilde\M_{\Gamma_{A,r,\rho}}((\sqrt[r]\D)_\rho)]^\vir),\qq d\geq 0.
\]
\begin{theorem}[see Theorem \ref{T polynomiality}]
Suppose $\D$ is a quotient orbifold of a smooth quasi-projective scheme by a linear algebraic group. Then for every $\Gamma$, every $\rho$-admissible vector $A\in\ration^n$ for $\Gamma$ and every $d\geq 0$, the cycle class $r^{2d-2\msf g+1}\tau_*(c_d(-(\sqrt[r]\L)_{\Gamma_{A,r,\rho}}) \cap[\tilde\M_{\Gamma_{A,r,\rho}}((\sqrt[r]\D)_\rho)]^\vir)$ is a polynomial in $r$ when $r \gg 1$.
\end{theorem}

We give an explicit formula for the constant term of
\[
r^{2d-2\msf g+1}\tau_*(c_d(-(\sqrt[r]\L)_{\Gamma_{A,r,\rho}}) \cap[\tilde\M_{\Gamma_{A,r,\rho}}((\sqrt[r]\D)_\rho)]^\vir)
\]
in Proposition \ref{P leading-term}.

With this polynomiality we prove a formula for double ramification cycles with orbifold targets. Let $\Y=\P(\L\oplus\mc O_\D)$ be the projectification of $\L$. Let $\D_0$ and $\D_\0$ be its 0-section and $\0$-section respectively. Let $\Gamma=(\msf g,\vec g,\beta,\mu_0,\mu_\0)$ be a topological data for $(\D_0|\Y|\D_\0)$ with $\mu_0$ and $\mu_\0$ denoting the orbifold information and contact orders for relative marked points mapped to $\D_0$ and $\D_\0$ respectively, see \S \ref{SSS D-DR-cycle}. Then we have the moduli space $\M_\Gamma(\D_0|\Y|\D_\0)^\sim$ of orbifold stable maps of topological type $\Gamma$ to the rubber targets of $(\D_0|\Y|\D_\0)$. By forgetting the contact orders in $\mu_0$ for $\D_0$ and $\mu_\0$ for $\D_\0$ we get a topological data $(\msf g,\beta,\vec g\sqcup\vec\mu_0\sqcup\vec\mu_\0)$ for $\D$, which we still denote by $\Gamma=(\msf g,\beta,\vec g\sqcup\vec\mu_0\sqcup\vec\mu_\0)$, see \eqref{E top-data-D}. There is a natural projection
\[
\epsilon_\D\co \M_\Gamma(\D_0|\Y|\D_\0)^\sim\rto\M_\Gamma(\D).
\]
The double ramification cycle (``DR-cycle'') for $(\D,\L)$ of type $\Gamma$ is defined as (cf. Definition \ref{D def-DR})
\[
\msf{DR}_\Gamma(\D,\L):=\epsilon_{\D,*} ([\M_\Gamma(\D_0|\Y|\D_\0)^\sim]^\vir).
\]
This is an orbifold generalization of the DR cycles for smooth targets defined in \cite{Janda-Pandharipande-Pixton-Zvonkine2018}.

Our second result is the computation for the cycle $\msf{DR}_\Gamma(\D,\L)$. From $\vec g$, $\mu_0$ and $\mu_\0$ we get a $\bar\rho$-compatible vector $A_{\bar\rho}$ (cf. \eqref{E A-bar-rho}) for $\Gamma=(\msf g,\beta,\vec g\sqcup\vec\mu_0\sqcup\vec\mu_\0)$. Then we obtain a topological data
\[
\Gamma_\0(r):=\Gamma_{A_{\bar\rho},r,\bar\rho}= \Upsilon_{r,\bar\rho}(\Gamma,A_{\bar\rho})
\]
for $(\sqrt[r]\D_\0)_{\bar\rho}=(\sqrt[r]\D)_{\bar\rho}$, see \eqref{E Gamma-00-r}.

\begin{theorem}[see Theorem \ref{T DR}]
Under the assumption of previous theorem, the DR-cycle $\msf{DR}_\Gamma(\D,\L)$ can be computed by
\[
\msf{DR}_\Gamma(\D,\L)=\left[- r\cdot\tau_*(c_{\msf g}(-(\sqrt[r]{\L^*})_{\Gamma_\0(r)}) \cap[\tilde\M_{\Gamma_\0(r)} ((\sqrt[r]\D)_{\bar\rho})]^\vir)\right]_{r^0},
\]
for $r\gg 1$, where $[\cdot]_{r^0}$ means the constant term of a polynomial in $r$.
\end{theorem}
By using the formula for the constant term of
\[
r^{2d-2\msf g+1}\tau_*(c_d(-(\sqrt[r]\L)_{\Gamma_{A,r,\rho}}) \cap[\tilde\M_{\Gamma_{A,r,\rho}}((\sqrt[r]\D)_\rho)]^\vir)
\]
in Proposition \ref{P leading-term}, we get an explicit formula for $\msf {DR}_\Gamma(\D,\L)$, see Theorem \ref{T formula-of-DR}.

There is also a $\rho$-compatible vector $A_\rho=-A_{\bar\rho}$ for $\Gamma=(\msf g,\beta,\vec g\sqcup\vec\mu_0\sqcup\vec\mu_\0)$, from which we get a topological data $\Gamma_0(r):=\Gamma_{A_\rho,r,\rho}=\Upsilon_{r,\rho}(\Gamma,A_\rho)$ for $(\sqrt[r]{\D_0})_\rho=(\sqrt[r]\D)_\rho$. Then we can also compute the DR-cycle $\msf{DR}_\Gamma(\D,\L)$ via
\[
\msf{DR}_\Gamma(\D,\L)=\left[ r\cdot\tau_*(c_{\msf g}(-(\sqrt[r]{\L})_{\Gamma_0(r)}) \cap[\tilde\M_{\Gamma_0(r)}((\sqrt[r]\D)_{\rho})]^\vir) \right]_{r^0},
\]
see Remark \ref{R Gamma-0-DR}. As a consequence we have an equality
\begin{align*}
&\left[- r\cdot\tau_*(c_{\msf g}(-(\sqrt[r]{\L^*})_{\Gamma_\0(r)}) \cap[\tilde\M_{\Gamma_\0(r)}((\sqrt[r]\D)_{\bar\rho})]^\vir) \right]_{r^0}\\
=\,&\left[ r\cdot\tau_*(c_{\msf g}(-(\sqrt[r]{\L})_{\Gamma_0(r)}) \cap[\tilde\M_{\Gamma_0(r)}((\sqrt[r]\D)_{\rho})]^\vir) \right]_{r^0}
\end{align*}
for $r\gg 1$.

As an application of the polynomiality and the computation of DR-cycles we study the relation between relative orbifold Gromov--Witten invariants of a relative pair $(\X|\D)$ and absolute orbifold Gromov--Witten invariants of $\X_{\D,r}$, the $r$-th root construction of $\X$ along the divisor $\D$. Now let $\Gamma=(\msf g,\vec g,\beta,\mu)$ be a topological data for $(\X|\D)$ with $\mu$ denoting the orbifold information and contact orders for relative marked points mapped to $\D$. Then for each $r\in\integer_{\geq 1}$ we get an induced topological data $\Gamma(r)$ for $\X_{\D,r}$ as the way we get $\Gamma_\0(r)$ and $\Gamma_0(r)$ from $\Gamma$, see \eqref{E gamma-r-in-Sec3}. Our third result is
\begin{theorem}[see Theorem \ref{Thm abs-rel} and Theorem \ref{Thm abs-rel-g=0}]
Under the assumption on $\D$ in previous theorems, when $r\gg 1$, the absolute orbifold Gromov--Witten invariant $\Big\langle\underline\alpha,\underline\mu \Big\rangle^{\X_{\D,r}}_{\Gamma(r)}$ of $\X_{\D,r}$ is a polynomial in $r$, and the constant term satisfies
\[
\left[\Big\langle\underline\alpha,\underline\mu \Big\rangle^{\X_{\D,r}}_{\Gamma(r)}\right]_{r^0}
=\Big\langle\underline\alpha\,\Big|\,\underline\mu \Big\rangle^{\X|\D}_{\Gamma}
\]
where as above $[\cdot]_{r^0}$ means the constant term of a polynomial in $r$, and the term on the right hand side is a relative orbifold Gromov--Witten invariant of $(\X|\D)$.

In particular, when the genus $\msf g=0$, the invariant
$\Big\langle\underline\alpha,\underline\mu \Big\rangle^{\X_{\D,r}}_{\Gamma(r)}$ is constant in $r$ when $r\gg 1$, and
\[
\Big\langle\underline\alpha,\underline\mu \Big\rangle^{\X_{\D,r}}_{\Gamma(r)}= \Big\langle\underline\alpha\,\Big|\,\underline\mu \Big\rangle^{\X|\D}_{\Gamma}.
\]
\end{theorem}
When $(\X|\D)=(X|D)$ is a smooth relative pair, such a result was obtained by Abramovich--Cadman--Wise (\cite{Abramovich-Cadman-Wise2017}) for $\msf g=0$ and by Tseng--You (\cite{Tseng-You2018b}) for $\msf g>0$.

This paper is organized as the following. In \S \ref{S TGW-root-bdls}, we prove the polynomiality of the cycle valued twisted Gromov--Witten invariants for roots of orbifold line bundles. Then following the line of the computation in \cite{Janda-Pandharipande-Pixton-Zvonkine2018}, combining their generalizations to the orbifold case, a formula of DR-cycles with orbifold targets was obtained in \S \ref{S DR-cycle}. Finally in \S\ref{S rel-abs-GWI} by using the polynomiality and computation of DR-cycles with orbifold targets we study the relation between relative orbifold Gromov--Witten invariants and absolute orbifold Gromov--Witten invariants of root constructions.

When this paper is finishing, a paper by Tseng and You (\cite{Tseng-You2020a}) is posed and we note that some results also appear in their paper, such as the polynomiality and the relation between relative orbifold Gromov--Witten invariants and absolute orbifold Gromov--Witten invariants of root constructions mentioned above.

\subsection*{Acknowledgements}
We warmly thank Hsian-Hua Tseng and Fenglong You for valuable comments on the earlier arXiv version; Rui Wang sincerely thanks Alexander Givental for stimulating discussions during the preparation of the work. The first author is supported by the National Natural Science Foundation of China (No. 11890663, No. 11821001, No. 11826102). The second author is supported by the National Natural Science Foundation of China (No. 11501393), the Sichuan Science and Technology Program (No. 2019YJ0509) and a joint research project of Laurent Mathematics Research Center of Sichuan Normal University and V.\thinspace C. \& V.\thinspace R. Key Lab of Sichuan Province.

\section{Twisted Gromov--Witten invariants of roots of orbifold line bundles}\label{S TGW-root-bdls}

In this paper, we study orbifolds via proper \'etale Lie groupoids, which are called orbifold groupoids. There are some nice references on orbifold groupoids.
See for example Adem--Leida--Ruan (\cite{Adem-Leida-Ruan2007}) and Moerdijk--Pronk (\cite{Moerdijk-Pronk1997}).
One can see also \cite[Section 2]{Chen-Du-Hu2019} for a brief introduction of orbifold groupoids and Chen--Ruan cohomology etc.

In this section we study the twisted Gromov--Witten theory of $r$-roots of orbifold line bundles. In \S \ref{SS root-bdls} we explain the roots of orbifold line bundles and root constructions of symplectic orbifolds along divisors. In \S \ref{SS TGW-root-bdls} we prove the polynomiality in $r$ of certain cycle valued twisted Gromov--Witten invariants assuming $r$ is sufficiently large.

\subsection{Roots of orbifold line bundles and root constructions along divisors}
\label{SS root-bdls}

\subsubsection{Orbifold line bundles and its $r$-th root}
Let $\pi\co\L\rto\D=D^1\ltimes D^0$ being a complex line bundle over a compact (almost) K\"ahler orbifold $\D$. (We always use orbifold groupoids to describe orbifolds and the invariants defined are independent of choices of Morita equivalent groupoid representations.) We can choose a groupoid representation so that $D^0$ is a disjoint union of contractible components and $L^0\rto D^0$ is a trivial complex line bundle
\[
L^0=D^0\times \cplane,\qq \L=D^1\ltimes L^0.
\]
The orbifold line bundle $\L\rto \D$ is completely characterized by the representation on $L^0$ for $D^1$ which we denote by
\begin{align}\label{E rep-rho}
\rho\co D^1\rto U(1).
\end{align}
The degree shifting (or age) of arrows in $D^1$ on $\L$ is defined as
\begin{align}\label{E age-g-L}
\age_g(\L)=\frac{\log(\rho(g))}{2\pi\sqrt{-1}}\in[0,1), \qq\mathrm{for}\qq g\in D^1,
\end{align}
where $\log(\cdot)$ is the principal logarithm that takes value in $[0,2\pi\sqrt{-1})$.

For the current groupoid representation, the $S^1$-principle bundle of $\L$ is $\msf P=D^1\ltimes(D^0\times S^1)$, and then $\L$ is the associated bundle as $\L=\msf P\times_{S^1_{(-1,1)}}\cplane$. Here, $S^1_{(-1,1)}$ means the action of $S^1$ on $\msf P\times\cplane$ has weight $(-1,1)$.

\begin{defn}\label{D r-root-L}
For every $r\in\integer_{\geqslant 1}$, the $r$-th root\footnote{There is also a version of root of line bundles over Deligne--Mumford stacks, see for example \cite{Abramovich-Graber-Vistoli2002, Cadman2007}. By view an orbifold groupoid as a representation of a Deligne--Mumford stack, these two definitions coincides with each other.} of the orbifold line bundle $\L\rto\D$ is defined to be the orbifold line bundle
\[
\sqrt[r]{\L}:=\msf P\times_{S^1_{(-r,1)}}\cplane,
\]
where $S^1_{(-r,1)}$ means the action of $S^1$ on $\msf P\times\cplane$ has weight $(-r,1)$.
\end{defn}

The base (and also the zero section) of the line bundle
$\sqrt[r]{\L}$ is denoted by $\sqrt[r]{\L/\D}$ in literatures and is called the $r$-th root gerbe of $\L$ (see for example
\cite{Andreini-Jiang-Tseng2015}). It is a $\integer_r$-gerbe over $\D$. In this paper we denote it by $(\sqrt[r]{\D})_\rho$ to emphasize the role of $\rho$ in its gerbe structure. Detailed construction is given in the following remark.

\begin{remark}\label{R Zr-gerbe}
Consider the exact sequence
\[
\xymatrix{1\ar[r] & \integer_r \ar[r] &U(1) \ar[rr]^-{\phi_r:t\mapsto t^r} &&U(1)\ar[r] &1,}
\]
which together with the representation $\rho: D^1\rto U(1)$ induces the following commutative diagram
\begin{align}\label{E tilde-D1}
\begin{split}\xymatrix{
1\ar[r]
&D^0\times\integer_r \ar[r] \ar[d]_-{pr_2}
& \tilde D^1_{\rho,r}\ar[r]^-{pr_1}\ar[d]_-{pr_2=\tilde\rho}
&D^1\ar[r] \ar[d]^-{\rho}
&1\\
1\ar[r] & \integer_r \ar[r] &U(1)\ar[r]^-{\phi_r} &U(1)\ar[r] &1.}\end{split}
\end{align}
with $\tilde D^1_{\rho,r}$ as the fiber product of $\rho$ and $\phi_r$, i.e. $\tilde D^1_{\rho,r}=\{(g,\xi)\in D^1\times U(1)\mid \rho(g)=\xi^r\}$. Then the $\integer_r$-gerbe $(\sqrt[r]{\D})_\rho=\tilde D^1_{\rho,r}\ltimes D^0$. The natural projection $\pi\co (\sqrt[r]{\D})_\rho\rto\D$ is given by  $pr_1\co \tilde D^1_{\rho,r}\rto D^1$ and $id_{D^0}\co D^0\rto D^0$.

Moreover, the $r$-th root of $\L$ can be written as
\[
\sqrt[r]{\L}=\tilde D^1_{\rho,r}\ltimes L^0
\]
with $\sqrt[r]\L\rto (\sqrt[r]\D)_\rho$ an orbifold line bundle. The representation $\tilde \rho\co \tilde D^1_{\rho,r}\rto U(1)$ for this line bundle is $\tilde \rho=pr_2$.
\end{remark}

Denote by $\T(\D)$ the index set of twisted sectors of the orbifold $\D$. Clearly, since $\T(\D)$ is the set of equivalence classes of conjugate classes and $U(1)$ is commutative, the representation $\rho$ for $\L$ descends to $\rho\co\T(\D)\rto U(1)$. Similarly, the age function descends to $\age_{\cdot}(\L)\co\T(\D)\rto [0,1)\cap\ration$. Moreover we have the following lemma.
\begin{lemma}\label{L T(Dr)}
The index set of twisted sectors of $(\sqrt[r]\D)_\rho$ is the fiber product of $\rho\co\T(\D)\rto U(1)$ and $\phi_r$, i.e.
\[
\T((\sqrt[r]\D)_\rho)=\{([g],\xi)\in\T(\D)\times U(1)\mid \rho([g])=\xi^r\}.
\]
The induced map $\pi\co\T((\sqrt[r]\D)_\rho)\rto\T(\D)$ of the natural map $\pi\co(\sqrt[r]\D)_\rho\rto \D$ on the index sets of twisted sectors is the projection to the first factor.
\end{lemma}

So for a given $[g]\in \T(\D)$, it has exactly $r$ preimages in $\T((\sqrt[r]\D)_\rho)$ under $\pi$. We call them {\em liftings} of $[g]$ in $\T((\sqrt[r]\D)_\rho)$. We could describe them in the following way. Consider a pair $([g],a)\in \T(\D)\times \ration$ and suppose it satisfies
\begin{align}\label{E lift-[g]-a}
\rho([g])=e^{2\pi\sqrt{-1}a},\qq\mathrm{i.e.}\qq \{a\}=\age_{[g]}(\L),
\end{align}
where $\{a\}$ is the fractional part of $a$. Then from the pair $([g],a)$ we get a lifting $([g],e^{2\pi\sqrt{-1}\frac{a}{r}})\in\T((\sqrt[r]\D)_\rho)$ of $[g]$. We use $\Upsilon_{r,\rho}$ to denote this lifting, i.e. we set
\begin{align}\label{E D-Psi}
\Upsilon_{r,\rho}([g],a):=([g],e^{2\pi\sqrt{-1}\frac{a}{r}}).
\end{align}
It is easy to see that
\begin{align}\label{E age-Upsilon}
\age_{\Upsilon_{r,\rho}([g],a)}(\sqrt[r]\L) =\left\{\frac{a}{r}\right\}=\frac{[a]_r}{r},
\end{align}
where $[a]_r$ is the remainder modulo out $r$, that is $a=kr+[a]_r$ with $[a]_r\in[0,r-1)\cap \ration$.
\begin{lemma}\label{L lift-[g]}
The liftings of $[g]\in\T(\D)$ in $\T((\sqrt[r]\D)_\rho)$  are in 1-to-1 correspondence to $\{a=\age_{[g]}(\L)+k\mid k=0,1,\ldots,r-1\}$ via $\Upsilon_{r,\rho}$. And for general $a,b\in\ration$ satisfying \eqref{E lift-[g]-a}, when $[a]_r=[b]_r$, the liftings $\Upsilon_{r,\rho}([g],a)=\Upsilon_{r,\rho}([g],b)$.
\end{lemma}

Next we consider the projectification fiber bundle.

\begin{defn}\label{D P(L+O)}
The projectification of $\L$ is defined to be
\[
\Y:=\P(\L\oplus\mc O_\D),
\]
where $\mc O_\D$ is the trivial line bundle over $\D$, i.e. whose representation $D^1\rto U(1)$ is trivial.

$\Y$ has the 0-section $\P(0\oplus \mc O_\D)$ and $\0$-section $\P(\L\oplus 0)$, which we denote by $\D_0$ and $\D_\0$ respectively. Both of them are isomorphic to $\D$.
\end{defn}

$\D_0$ and $\D_\0$ are divisors of $\Y$, whose normal line bundles are $\L$ and $\L^*$ respectively, where $\L^*$ is the dual line bundle of $\L$.

Let $\T(\L)$ be the index set of twisted sectors of the orbifold line bundle $\L$. It is canonically identified with the set of twisted sectors of $\D$, i.e., $\T(\L)=\T(\D)$. There is the inertia orbifold bundle
\begin{align}\label{E IL-ID}
\I\pi=\bigsqcup_{[g]\in\T(\D)} \pi_{[g]}\co\I\L=\bigsqcup_{[g]\in\T(\D)} \L_{[g]}\rto\I\D=\bigsqcup_{[g]\in\T^\D}\D_{[g]}
\end{align}
constructed from the projections of inertia spaces of $\L$ and $\D$. The following lemma is easy to see.

\begin{lemma}\label{L T(L)}
$\msf I\pi\co \msf I\L\rto\msf I\D$ is also an orbifold vector bundle. For $[g]\in\T(D)$,
\begin{enumerate}
\item if $\rho([g])\neq 1$, then $\L_{[g]}=\D_{[g]}$;
\item if $\rho([g])=1$, then $\L_{[g]}=D^1_{[g]}\ltimes(D^0_{[g]}\times\cplane)$.
    In fact $\L_{[g]}=e_{[g]}^*\L$, where $e_{[g]}\co \D_{[g]}\rto\D$ is the natural evaluation map.
\end{enumerate}
\end{lemma}

For the projectification fiber bundle $\Y=\P(\L\oplus\mc O_\D)$, there is the inertia fiber bundle $\msf I\pi\co \msf I\Y\rto\msf I\D$ constructed similarly.
\begin{lemma}\label{L CR-Y}
For each $[g]\in\T(\D)$, the component $\Y_{[g]}:=\pi_{[g]}\inv(\D_{[g]})$ of the fiber bundle $\I\pi\co\I\Y\rto\I\D$ is determined as follow:
\begin{enumerate}
\item If $\rho([g])\neq 1$, then $\Y_{[g]}=(\D_0)_{[(g)]}\sqcup(\D_\0)_{[g]}$ is a disjoint union of the twisted sector of $\D_0\cong\D$ corresponding to $[g]$ and the twisted sector of $\D_\0\cong \D$ corresponding to $[g]$.

\item If $\rho([g])=1$, then $\Y_{[g]}=\P(\L_{[g]}\oplus\mc O_{\D_{[g]}})$. Moreover, $\Y_{[g]}$ also contains the zero and infinity sections $(\D_0)_{[(g)]}$ and $(\D_\0)_{[g]}$.
\end{enumerate}
\end{lemma}

\subsubsection{Root constructions}

Consider a symplectic orbifold pair $(\X,\D)$ with $\D$ being a divisor of $\X$. By the weight-$(-r)$ blowup for $\X$ along $\D$ (cf. \cite[Section 3]{Chen-Du-Hu2019}), one get the orbifold $\X_{\D,r}$, with exceptional divisor $\D_r$. This corresponds to the $r$-th root construction of Deligne--Mumford stacks (cf. \cite{Abramovich-Graber-Vistoli2008,Cadman2007}). In fact, denote by $\L\rto\D$ the normal line bundle of $\D$ in $\X$ and $\rho\co D^1\rto U(1)$ the corresponding representation for $\L$. Then the exceptional divisor $\D_r$ is just $(\sqrt[r]\D)_\rho$ and the normal bundle of $\D_r$ in $\X_{\D,r}$ is just the $r$-th root $\sqrt[r]\L$ of $\L$ as in Definition \ref{D r-root-L}.

Taking $\X$ as the projectification fiber bundle $\Y$ in Definition \ref{D P(L+O)} and do the $r$-th root construction along its zero divisor $\D_0$ and infinity divisor $\D_\0$, we obtain $\Y_{\D_0,r}$ and $\Y_{\D_\0,r}$ respectively. We next describe the orbifold structures of $\Y_{\D_0,r}$ and $\Y_{\D_\0,r}$.

First $\Y_{\D_\0,r}$ can be written as the weight-$r$ projectification
\[
\Y_{\D_\0,r}=\P_{r,1}(\L\oplus\mc O_\D)=D^1\ltimes \P_{r,1}(L^0\oplus\mc O_{D^0}).
\]
The original $0$-section $\D_0$ of $\Y$ remains unchanged in $\Y_{\D_\0,r}$, and the original $\0$-section $\D_\0$ of $\Y$ becomes $(\sqrt[r]{\D_\0})_{\bar\rho}$, where
\begin{align}\label{E bar-rho}
\bar\rho(\cdot):=(\rho(\cdot))\inv
\end{align}
is the dual representation of $\rho$, i.e. the representation of the dual line bundle $\L^*$ of $\L$. Apply the general construction  Remark \ref{R Zr-gerbe} and Lemma \ref{L T(Dr)} to $\L^*$, we obtain
\begin{align*}
&\tilde D^1_{\bar\rho,r}=\{(g,\xi)\in D^1\times U(1)\mid \bar\rho(g)=\xi^r\},
\end{align*}
and
\begin{align*}
&\T((\sqrt[r]\D)_{\bar\rho})=\{([g],\xi)\in\T(\D)\times U(1)\mid \bar\rho([g])=\xi^r\}.
\end{align*}

For the inertia space of $\Y_{\D_\0,r}$, the following commutative diagram of the natural maps
\[
\xymatrix{\Y_{\D_\0,r}\ar[r]^-\kappa \ar[dr]^-\pi& \Y \ar[d]^-\pi\\ &\D,}
\]
induces the commutative diagram of inertia spaces
\[
\xymatrix{\msf I\Y_{\D_\0,r}\ar[r]^-{\msf I\kappa} \ar[dr]^-{\msf I\pi}& \Y \ar[d]^-{\msf I\pi}\\ &\msf I\D,}
\]
For a $[g]\in\T(\D)$, set $(\Y_{\D_\0,r})_{[g]}:=\msf I\pi\inv(\D_{[g]})$. Then $(\Y_{\D_\0,r})_{[g]}=\msf I\kappa\inv(\Y_{[g]})$, and there is the following lemma.
\begin{lemma}\label{L CR-Y-D00-r}
For a $[g]\in\T(\D)$,
\begin{enumerate}
\item if $\rho([g])=1$, then $(\Y_{\D_\0,r})_{[g]}$ is a disjoint union of $\P_{r,1}(\L_{[g]}\oplus\mc O_{\D_{[g]}})$ (containing $((\sqrt[r]{\D_\0})_{\bar\rho})_{([g],1)}$ as $\0$-section and $(\D_0)_{[g]}$ as $0$-section) and $((\sqrt[r]{\D_\0})_{\bar\rho})_{([g],\xi)}$ with $\xi=e^{2\pi\sqrt{-1}\frac{k}{r}}, 1\leq k\leq r-1$.
\item if $\rho([g])\neq 1$, then $(\Y_{\D_\0,r})_{[g]}$ is a disjoint union of $(\D_0)_{[g]}$ and $((\sqrt[r]{\D_\0})_{\bar\rho})_{([g],\xi)}$ with $\xi^r=\bar\rho([g])$.
\end{enumerate}
\end{lemma}

Similarly $\Y_{\D_0,r}$ can be written as the weight-$r$ projectification
\[
\Y_{\D_0,r}=\P_{r,1}(\L^*\oplus\mc O_{\D})=\Y_{\D_0,r}=D^1\ltimes\P_{r,1}(L^{*,0}\oplus\mc O_{D^0}).
\]
The original $\0$-section $\D_\0$ of $\Y$ remains unchanged in $\Y_{\D_0,r}$, and the original $0$-section $\D_0$ of $\Y$ becomes $(\sqrt[r]{\D_0})_\rho$. Again by applying the general construction Remark \ref{R Zr-gerbe} and Lemma \ref{L T(Dr)} to $\L$, we obtain
\begin{align*}
\tilde D^1_{\rho,r}&=\{(g,\xi)\in D^1\times U(1)\mid \rho(g)=\xi^r\},\\
\T((\sqrt[r]{\D_0})_\rho)&=\{([g],\xi)\in\T(\D)\times U(1)\mid \rho([g])=\xi^r\}.
\end{align*}
We also have the projection between inertia spaces $\msf I\pi\co \msf I\Y_{\D_0,r}\rto\msf I\D$. For a $[g]\in\T(\D)$, set $(\Y_{\D_0,r})_{[g]}:=\msf I\pi\inv(\D_{[g]})$.
\begin{lemma}\label{L T(D0)+CR-Y-D0}
We have
\begin{enumerate}
\item if $\rho([g])=1$, then $(\Y_{\D_0,r})_{[g]}$ is a disjoint union of $\P_{r,1}(\L^*_{[g]}\oplus\mc O_{\D_{[g]}})$ (containing $((\sqrt[r]{\D_0})_\rho)_{([g],1)}$ and $(\D_\0)_{[g]}$) and $((\sqrt[r]{\D_0})_\rho)_{([g],\xi)}$ with $\xi=e^{2\pi\sqrt{-1}\frac{k}{r}}, 1\leq k\leq r-1$.
\item if $\rho([g])\neq 1$, then $(\Y_{\D_0,r})_{[g]}$ is a disjoint union of $(\D_\0)_{[g]}$ and $((\sqrt[r]{\D_0})_\rho)_{([g],\xi)}$ with $\xi^r=\rho([g])$.
\end{enumerate}
\end{lemma}

This lemma also gives a description of the local structure of $\X_{\D,r}$ along its exceptional divisor $\D_r=(\sqrt[r]{\D_0})_\rho$.

\subsection{Twisted Gromov--Witten invariants of root gerbes of orbifold line bundles}\label{SS TGW-root-bdls}

In this subsection we prove the polynomiality of certain cycle valued twisted orbifold Gromov--Witten invariants of root gerbes of orbifold line bundles.

\subsubsection{Setup and main result on polynomiality}
We fix an orbifold line bundle $\L\rto\D$ with representation $\rho:D^1\rto U(1)$. As in previous subsection, denote by $(\sqrt[r]{\D})_\rho$ the $r$-th root gerbe over $\D$ and by $\sqrt[r]\L\rto(\sqrt[r]\D)_\rho$ the $r$-th root of $\L\rto \D$.

A topological data for $\D$ consists of the triple
\[
\Gamma=(\msf g,\beta,\vec g)
\]
with
\begin{itemize}
\item $\msf g\in\integer_{\geq 0}$ the genus,
\item $\beta\in H_2(|\D|;\integer)$ the homological class,
\item $\vec g=([g_1],\ldots, [g_n])\in(\T(\D))^n$ encoding orbifold information of the $n$ marked points.
\end{itemize}
We also denote by $|\Gamma|=(\msf g,\beta,n)$ by forgetting the orbifold data $\vec g$.

We next lift $\Gamma$ to a topological data of $(\sqrt[r]\D)_\rho$ by lifting $\vec g$ to $(\T((\sqrt[r]\D)_\rho))^n$. So we need to choose liftings for $[g_i], 1\leq i\leq n$. By Lemma \ref{L lift-[g]} and \eqref{E lift-[g]-a}, to lift $\vec g$ we need to choose a vector
\[
A=(a_1,\ldots,a_n)\in\ration^n,
\]
such that
\begin{align}\label{E Gamma-admissible}
\rho([g_i])=e^{2\pi\sqrt{-1}a_i}, \qq \mathrm{i.e.}\qq \{a_i\}=\age_{[g_i]}(\L), \qq\mathrm{for}\qq 1\leq i\leq n.
\end{align}
We call such a vector {\em $\rho$-admissible for $\vec g$} (or {\em $\rho$-admissible for $\Gamma$} as $\vec g$ is the orbifold information for marked points in $\Gamma$).

\begin{defn}\label{D lift-Gamma}
Given a $\rho$-admissible vector $A\in\ration^n$ for $\Gamma$. The lifting of $\Gamma$ via $A$ is defined as
\begin{align}\label{E Psi(Gamma,A)}
\Upsilon_{r,\rho}(\Gamma,A):=(\msf g, \beta,(\Upsilon_{r,\rho}([g_i],a_i))_{i=1}^n).
\end{align}
We simplify notation and use $\Gamma_{A,r,\rho}:=\Upsilon_{r,\rho}(\Gamma,A)$ to denote the lifted topological data for $(\sqrt[r]\D)_\rho$ from $\Gamma$ by $A$ when there is no danger of confusion. We also denote $(\Upsilon_{r,\rho}([g_i],a_i))_{i=1}^n$ by $\Upsilon_{r,\rho}(\vec g,A)$.
\end{defn}
For simplicity we denote $\Upsilon_{r,\rho}([g_i],a_i)$ by $([g_i],\xi_i)$ for $1\leq i \leq n$. So by \eqref{E age-Upsilon} we have
\begin{align}\label{E age-gxi-i}
\age_{([g_i],\xi_i)}(\sqrt[r]\L)= \age_{\Upsilon_{r,\rho}([g_i],a_i)}(\sqrt[r]\L)=
\frac{[a_i]_r}{r}.
\end{align}

Now assume that $\Gamma$ is a topological data for $\D$ and $A$ is $\rho$-admissible for $\Gamma$. Let $\Gamma_{A,r,\rho}$ be the lifted topological data for $(\sqrt[r]\D)_\rho$. Consider the moduli space $\M_\Gamma(\D)$ of orbifold stable maps to $\D$ of topological type $\Gamma$ and the moduli space $\M_{\Gamma_{A,r,\rho}}((\sqrt[r]\D)_\rho)$ of orbifold stable maps to $(\sqrt[r]\D)_\rho$ of topological type $\Gamma_{A,r,\rho}$. The natural projection $\pi\co(\sqrt[r]\D)_\rho\rto\D$ induces the natural projection
\[
\epsilon\co\M_{\Gamma_{A,r,\rho}}((\sqrt[r]\D)_\rho) \rto\M_\Gamma(\D).
\]

We next introduce the main theorem in this section. Over the moduli space $\M_{\Gamma_{A,r,\rho}}((\sqrt[r]\D)_\rho)$, there is a universal curve which we denote as
\[
\scr C\rto \M_{\Gamma_{A,r,\rho}}((\sqrt[r]\D)_\rho)
\]
with each fiber the {\em smooth} (nodal) Riemann surfaces. The images of the $n$ sections corresponding to $n$ marked points are denoted by $S_i, i=1,\ldots,n$; the locus of nodal points in $\scr C$ are denoted by $\scr Z_\node$. The locus of nodal curves in $\M_{\Gamma_{A,r,\rho}}((\sqrt[r]\D)_\rho)$ are denoted by $B_\node$. We do weighted blowup to $\M_{\Gamma_{A,r,\rho}}((\sqrt[r]\D)_\rho)$ along $B_\node$ with the weights given by the order $r_\node$ of the orbifold structure at the corresponding nodal points, and do weighted blowup to $\scr C$ along $S_i$ and $\scr Z_\node$ according to the orders of the orbifold structure at marked points and nodal points. By this way, we obtain a new universal curve
\[
\tilde{\scr C}\rto\tilde{\M}_{\Gamma_{A,r,\rho}}((\sqrt[r]\D)_\rho),
\]
which carries a universal map
\[
\xymatrix{\tilde{\scr C}\ar[r]^-\f \ar[d]_-\pi & (\sqrt[r]\D)_\rho\\
\tilde{\M}_{\Gamma_{A,r,\rho}}((\sqrt[r]\D)_\rho). &}
\]
\begin{remark}\label{R universl-curve}
The universal curve $\tilde{\scr C}$ can be taken as follows. First denote by
\[
\Gamma'=(\msf g,(\vec g,[1]),\beta),
\]
i.e., add a marked point with untwisted orbifold structure to $\Gamma=(\msf g,\vec g,\beta)$ as the $(n+1)$-th marked point. Since the added last marked point has trivial orbifold structure, $A'=(A,0)$ is $\rho$-admissible for $\Gamma'$. We use $A'$ to lift $\Gamma'$ to obtain $\Upsilon_{r,\rho}(\Gamma',A')$. Then one can see this topological data is just $\Gamma_{A,r,\rho}'$, i.e. the one obtained from $\Gamma_{A,r,\rho}$ by adding the $(n+1)$-th marked point with untwisted orbifold structure as the way we get $\Gamma'$ from $\Gamma$. The universal curve $\scr C$ over $\M_{\Gamma_{A,r,\rho}}((\sqrt[r]\D)_\rho)$ can be taken as $\M_{\Gamma_{A,r,\rho}'}((\sqrt[r]\D)_\rho)$. The images of the $n$ sections $S_i, 1\leq i\leq n$ and the locus of nodal points $\scr Z_\node$ in $\scr C$ combine to the locus of nodal curves in $\M_{\Gamma_{A,r,\rho}'}((\sqrt[r]\D)_\rho)$. Hence the universal curve $\tilde{\scr C}$ is $\tilde\M_{\Gamma_{A,r,\rho}'}((\sqrt[r]\D)_\rho)$.
\end{remark}

The $K$-theoretic push-forward of the pullback bundle $\f^*\sqrt[r]{\L}$ over $\tilde{\M}_{\Gamma_{A,r,\rho}}((\sqrt[r]\D)_\rho)$ is $\mc R\pi_*\f^*\sqrt[r]{\L}$. For short we denote it by
\[
(\sqrt[r]\L)_{\Gamma_{A,r,\rho}}:=\mc R\pi_*\f^*\sqrt[r]{\L}.
\]
Let $c(-(\sqrt[r]\L)_{\Gamma_{A,r,\rho}})$ be the total Chern class of $-(\sqrt[r]\L)_{\Gamma_{A,r,\rho}}$ and denote by
$\tau=\epsilon\circ\pi$ the composition
\[
\xymatrix{\tilde{\M}_{\Gamma_{A,r,\rho}}((\sqrt[r]\D)_\rho) \ar[r]^-{\pi} & \M_{\Gamma_{A,r,\rho}}((\sqrt[r]\D)_\rho) \ar[r]^-{\epsilon} & \M_{\Gamma}(\D),}
\]
where the first map $\pi$ is the natural blowdown map.

Our main theorem in this section is about the cycle
\begin{align}\label{E cycle}
&\tau_*\left(c(-(\sqrt[r]\L)_{\Gamma_{A,r,\rho}}) \cap[\tilde{\M}_{\Gamma_{A,r,\rho}}((\sqrt[r]\D)_\rho)]^\vir\right) \\&
=\sum_{d\geq 0} \tau_*\left(c_d(-(\sqrt[r]\L)_{\Gamma_{A,r,\rho}}) \cap[\tilde{\M}_{\Gamma_{A,r,\rho}}((\sqrt[r]\D)_\rho)]^\vir\right),
\nonumber
\end{align}
and is given as follows.
\begin{theorem} \label{T polynomiality}
Suppose that $\D$ is a quotient orbifold of a smooth quasi-projective scheme by a linear algebraic group. Then for every $\Gamma$, every $\rho$-admissible vector $A\in\ration^n$ for $\Gamma$ and every $d\geq 0$, the cycle class $r^{2d-2\msf g+1}\tau_*\left(c_d(-(\sqrt[r]\L)_{\Gamma_{A,r,\rho}}) \cap[\tilde\M_{{\Gamma_{A,r,\rho}}}((\sqrt[r]\D)_\rho)]^\vir\right)$ is a polynomial in $r$ when $r\gg 1$.
\end{theorem}

The proof of this theorem is computational. The rest of this section is devoted to the proof of this theorem and an explicit formula for the least order term, i.e. constant term, of $r^{2d-2\msf g+1}\tau_*\left(c_d(-(\sqrt[r]\L)_{\Gamma_{A,r,\rho}}) \cap[\tilde\M_{{\Gamma_{A,r,\rho}}}((\sqrt[r]\D)_\rho)]^\vir\right)$.

The strategy is to calculate the Chern character $ch((\sqrt[r]\L)_{\Gamma_{A,r,\rho}})$ and then use the formula
\begin{align}\label{E c-from-ch}
c(-E^\bullet)=\exp\left( \sum_{d\geq 1} (-1)^d (d-1)! ch_d(E^\bullet)\right)
\end{align}
to obtain the Chern class $c(-(\sqrt[r]\L)_{\Gamma_{A,r,\rho}})$. For example the $d$-th Chern class of $-(\sqrt[r]\L)_{\Gamma_{A,r,\rho}}$ is
\[
c_d(-(\sqrt[r]\L)_{\Gamma_{A,r,\rho}})= -\frac{1}{d!}\left(ch_1((\sqrt[r]\L)_{\Gamma_{A,r,\rho}})\right)^d +\ldots.
\]

The next several subsections are organized as follows:

Before we enter the calculation of $ch((\sqrt[r]\L)_{\Gamma_{A,r,\rho}})$, we first give a description for the strata of several needed moduli spaces, $\M_\Gamma(\D)$, $\M_{\Gamma_{A,r,\rho}}((\sqrt[r]\D)_\rho)$ and $\tilde\M_{\Gamma_{A,r,\rho}}((\sqrt[r]\D)_\rho)$. This part is done in \S\ref{SSS strata}.

In \S \ref{SSS ch}, we use the orbifold Grothendieck--Riemann--Roch formula to calculate the Chern character $ch(-(\sqrt[r]\L)_{\Gamma_{A,r,\rho}})$. It turns out from the calculation that parts of the Chern characters support over the locus of nodal curves which makes it necessary to include the contribution from lower strata described in \S \ref{SSS strata}.

In \S \ref{SSS chernclass}, we plug the contribution from each strata to the expression of the Chern classes and finishes the proof.

\subsubsection{Strata description for related moduli spaces}\label{SSS strata}
For a moduli space of orbifold stable maps, a strata needs two ingredients to describe: A target graph type and a decoration for orbifold data. We start with the simplest moduli space we need in this section, the moduli space $\M_\Gamma(\D)$. Then the same description works for $\M_{\Gamma_{A,r,\rho}}((\sqrt[r]\D)_\rho)$.

We adapt the terminologies in \cite{Janda-Pandharipande-Pixton-Zvonkine2018}. Denote by $G_{\msf g,\beta,n}(\D)=G_{|\Gamma|}(\D)$ the set of stable $\D$-graphs of genus $\msf g$, homology type $\beta$ and with $n$ marked points (legs). A such graph $\fk \Gamma\in G_{|\Gamma|}(\D)$ consists of the following data
\[
\fk\Gamma=(\tV,\tE,\tH,\tL,\msf g\co\tV\rto\integer_{\geq 0},\msf v\co\tH\rto\tV,\iota\co\tH\rto\tH,\beta\co\tV\rto H_2(|\D|;\integer)).
\]
where
\begin{enumerate}
\item $\tV$ is the vertex set, $\msf g\co \tV\rto\integer_{\geq 0}$ is the genus function, and $\beta\co\tV\rto H_2(|\D|;\integer)$ is the homology class function,
\item $\tH$ is the half-edge set with involution $\iota\co\tH\rto\tH$, and $\msf v\co\tH\rto\tV$ is the vertex assignment function,
\item $\tE$ is the edge set, which consists of $2$-cycles of $\iota$ in $\tH$ (self-edges at vertices are permitted),
\item $\tL$ is a subset of $\tH$ which consists of fixed points of $\iota$ and is ordered by $n$ marked points,
\item the pair $(\tV,\tE)$ defines a connected graph satisfying the genus condition
\begin{align}\label{E genus}
\msf g=\sum_{\msf v\in \tV}\msf g(\msf v)+h^1(\fk \Gamma),
\end{align}
where $h^1(\fk\Gamma)$ is the rank of the degree 1 homology group of the connected graph defined by $(\tV,\tE)$, which is
\begin{align}\label{E h1(Gamma)}
h^1(\fk\Gamma)=|\tE|-|\tV|+1,
\end{align}
\item for each vertex $\msf v$, the stability condition holds: if $\beta(\msf v)=0$, then $2\msf g(\msf v)-2+n(\msf v)>0$, where $n(\msf v)$ is the valence of $\fk\Gamma$ at $\msf v$ including both edges and legs,
\item the degree condition holds:
\[
\sum_{\msf v\in\tV}\beta(\msf v)=\beta.
\]
\end{enumerate}

An automorphism of $\fk\Gamma\in G_{|\Gamma|}(\D)$ consists of automorphisms of the sets $\tV$ and $\tH$ which leaves invariant the structures $\tL$, $\msf g$, $\msf v$, $\iota$, and $\beta$. Let $\aut(\fk\Gamma)$ denote the automorphism group of $\fk\Gamma$.

To present a strata of $\M_\Gamma(\D)$ via a graph $\fk \Gamma$, we also need a decoration which decorates each half edge of $\fk \Gamma$ by a twisted sector of $\D$, i.e. a map $\chi\co \tH(\fk\Gamma)\rto \T(\D)$, and require it matches the orbifold structure $\vec g$ in $\Gamma$.
\begin{defn}\label{D chi}
We call a map
\[
\chi\co\tH(\fk\Gamma) \rto \T(\D)
\]
an orbifold decoration for $\fk \Gamma\in G_{|\Gamma|}(\D)$, if \begin{itemize}
\item $\chi$ maps the $i$-th leg $\msf h_i$ to $[g_i]$, for $1\leq i\leq n$;
\item for a vertex $\msf v\in\tV(\fk\Gamma)$, {the moduli space $\M_{\msf g(\msf v),\beta(\msf v),\{\chi(\msf h)\}_{\msf h\in \tH(\msf v)}}(\D)$ is nonempty, i.e.} there exists a degree $\beta(\msf v)$ representable map from a genus $\msf g(\msf v)$ orbifold curve whose marked points are mapped into the twisted sectors of $\D$ specified by $\{\chi(\msf h)\}_{\msf h\in \tH(\msf v)}$, where $\tH(\msf v)$ is the set of half edges with vertex $\msf v$;
\item for an edge $\msf e=(\msf h,\msf h')\in \tE(\fk\Gamma)$, we have $\chi(\msf h)=\chi(\msf h')\inv$, where $\chi(\msf h')\inv$ means the twisted sector $I(\D_{\chi(\msf h')})$, and $I\co\I\D\rto\I\D$ is the canonical involution map for twisted sectors.
\end{itemize}
We define ${\fk D}_{\fk\Gamma,\Gamma}$ to be the set of all such orbifold decorations associated to $\fk\Gamma\in G_{|\Gamma|}(\D)$.
\end{defn}

For each graph $\fk \Gamma\in G_{|\Gamma|}(\D)$, and each orbifold decoration $\chi\in{\fk D}_{\fk \Gamma,\Gamma}$, there is a component $\M_{\fk \Gamma,\chi}$ parameterizing maps with nodal domains of topological types given by $\fk\Gamma$ and orbifold structures given by $\chi$. Let
\[
\zeta_{\fk\Gamma,\chi}\co \M_{\fk\Gamma,\chi}\hrto\M_\Gamma(\D)
\]
be the inclusion of the strata.

Now we consider $\M_{\Gamma_{A,r,\rho}}((\sqrt[r]\D)_\rho)$. We could apply the above description to $\M_{\Gamma_{A,r,\rho}}((\sqrt[r]\D)_\rho)$ to describe the strata of $\M_{\Gamma_{A,r,\rho}}((\sqrt[r]\D)_\rho)$ by replacing $\D$ above by $(\sqrt[r]\D)_\rho$ and $\Gamma$ by $\Gamma_{A,r,\rho}$ formally. So we have the set of graphs $G_{|\Gamma_{A,r,\rho}|}((\sqrt[r]\D)_\rho)= G_{|\Gamma|}((\sqrt[r]\D)_\rho)= G_{|\Gamma|}(\D)=G_{\msf g,\beta,n}(\D)$ and the set of decorations $\fk D_{\tilde{\fk\Gamma},\Gamma_{A,r,\rho}}$ for each graph $\tilde{\fk\Gamma}=\fk\Gamma\in G_{|\Gamma_{A,r,\rho}|}((\sqrt[r]\D)_\rho)= G_{|\Gamma|}(\D)$. We next describe the relations between strata of $\M_{\Gamma_{A,r,\rho}}((\sqrt[r]\D)_\rho)$ and $\M_{\Gamma}(\D)$.

Given a graph $\tilde{\fk\Gamma}=\fk\Gamma\in G_{|\Gamma_{A,r,\rho}|}(\D)= G_{|\Gamma|}(\D)$, and an orbifold decoration $\tilde\chi\in{\fk D}_{\tilde{\fk\Gamma},\Gamma_{A,r,\rho}}$. We get an orbifold decoration $\chi\in\fk D_{\fk\Gamma,\Gamma}$ as the composite map
\[
\chi\co \tH(\fk\Gamma)\xrightarrow{\tilde\chi} \T((\sqrt[r]\D)_\rho)\xrightarrow{\pi} \T(\D).
\]
By the relation between $ \T((\sqrt[r]\D)_\rho)$ and $\T(\D)$ in Lemma \ref{L lift-[g]}, we see that $\tilde\chi$ is uniquely characterized by $\chi$ and an associated map
\[
w\co\tH(\fk\Gamma)\rto\{0,\ldots,r-1\}
\]
via
\begin{align}\label{E tildechi=chi-w}
\tilde\chi(\msf h)=(\chi(\msf h),e^{2\pi\sqrt{-1}\frac{\age_{\chi(\msf h)}(\L)+w(\msf h)}{r}})=\Upsilon_{r,\rho}\left(\chi(\msf h),\age_{\chi(\msf h)}(\L)+w(\msf h)\right),
\end{align}
i.e. for every $\msf h\in \tH(\fk\Gamma)$, $\tilde\chi(\msf h)$ is a lifting of $\chi(\msf h)$. As a consequence we have
\begin{align}\label{E age-tildechi(h)}
\age_{\tilde\chi(\msf h)}(\sqrt[r]\L)=\frac{w(\msf h)+\age_{\chi(\msf h)}(\L)}{r}.
\end{align}
We call $\tilde\chi$ a {\em lifting} of $\chi$ by $w$.

The map $w$ is called a {\em weight function (associated to $\chi$)} (cf. \cite{Janda-Pandharipande-Pixton-Zvonkine2018,Tseng-You2016b}). According to the requirements in Definition \ref{D chi} for orbifold decorations, such an associated weight function $w$ of $\chi$ must satisfy the following properties.
\begin{lemma}\label{L weight}
\begin{enumerate}
\item For each leg $\msf h_i\in \tL(\fk \Gamma), 1\leq i\leq n$, $w(\msf h_i) =[r\cdot\age_{([g_i],\xi_i)}(\sqrt[r]\L)]_{\integer} \equiv [a_i]_{\integer}\pmod r$, where $[\cdot]_\integer$ denotes the integer part of a real number.

\item For $\msf e=(\msf h_+,\msf h_-)\in\tE(\fk\Gamma)$, if $\rho(\chi(\msf h_+))=1$, then $w(\msf h_+)+w(\msf h_-)\equiv 0 \pmod r$. If $\rho(\chi(\msf h_+))\neq 1$, then $w(\msf h_+)+w(\msf h_-)\equiv r-1 \pmod r$.
\item For $\msf v\in\tV(\fk\Gamma)$, $\sum_{\msf h\in\tH(\msf v)} w(\msf h)\equiv A(\msf v,\chi) \pmod r$, where $A(\msf v,\chi):=\int^\orb_{\beta(\msf v)}c_1(\L)-\sum_{\msf h\in \tH(\msf v)} \age_{\chi(\msf h)}(\L)\in\integer$.
\end{enumerate}
\end{lemma}

\begin{remark}
For a $\msf v\in \tV(\fk \Gamma)$, that $A(\msf v,\chi)\in\integer$ follows from the orbifold Riemann--Roch formula for the orbifold line bundle $\f^*\L\rto \C$, where $\f\co\msf C\rto \D$ is a stable map in $\M_{\msf g(\msf v),\beta(\msf v),\{\chi(\msf h)\}_{\msf h\in \tH(\msf v)}}(\D)$. Similarly, that $\sum_{\msf h\in\mathrm{ H}(\msf v)} w(\msf h)\equiv A(\msf v,\chi) \pmod r$ follows from the orbifold Riemann--Roch formula for the orbifold line bundle $\f^*\sqrt[r]\L\rto \C$, where $\f\co\msf C\rto (\sqrt[r]\D)_\rho$ is a stable map in $\M_{\msf g(\msf v),\beta(\msf v),\{\tilde\chi(\msf h)\}_{\msf h\in \tH(\msf v)}}((\sqrt[r]\D)_\rho)$.
\end{remark}

For a fixed $\chi\in{\fk D}_{\fk\Gamma,\Gamma}$ denote by $W^{\L,\rho}_{\fk\Gamma,\chi,r}$ by the set of all weight functions satisfying the three conditions in Lemma \ref{L weight}. Then we have
\begin{lemma}
For a fixed $\chi\in{\fk D}_{\fk\Gamma,\Gamma}$ the set of liftings $\tilde\chi\in {\fk D}_{\fk\Gamma,\Gamma_{A,r,\rho}}$ of $\chi$ is 1-to-1 corresponding to the set $W^{\L,\rho}_{\fk\Gamma,\chi,r}$.
\end{lemma}
\begin{proof}To see this 1-to-1 correspondence one only need to notice that the third condition above is the sufficient and necessary condition to lift a map in $\M_{\msf g(\msf v),\beta(\msf v),\{\chi(\msf h)\}_{\msf h\in \tH(\msf v)}}(\D)$ to a map in $\M_{\msf g(\msf v),\beta(\msf v),\{\tilde\chi(\msf h)\}_{\msf h\in \tH(\msf v)}}((\sqrt[r]\D)_\rho)$.
\end{proof}

In the following, for a given $\chi\in{\fk D}_{\fk\Gamma,\Gamma}$ and a given $w\in W^{\L,\rho}_{\fk\Gamma,\chi,r}$, we denote by $\tilde\chi_{(\chi;w)}\in {\fk D}_{\fk\Gamma,\Gamma_{A,r,\rho}}$ the corresponding orbifold decoration that lifts $\chi$ by $w$.

As for $\M_\Gamma(\D)$, for a graph $\fk\Gamma\in G_{|\Gamma|}(\D)$, an orbifold decoration $\chi \in{\fk D}_{\fk\Gamma,\Gamma}$ and a weight $w\in W_{\fk\Gamma,\chi,r}^{\L,\rho}$, there is a component $\M_{\fk\Gamma,(\chi;w)}=\M_{\fk\Gamma,\tilde\chi_{(\chi;w)}}$ parameterizing maps with nodal domains of topological types given by $\fk\Gamma$ and orbifold structures given by the orbifold decoration $\tilde\chi_{(\chi;w)}$, hence $\chi$ and $w$. Let
\[
\zeta_{\fk\Gamma,(\chi;w)}\co\M_{\fk\Gamma,(\chi;w)}\rto \M_{\Gamma_{A,r,\rho}}((\sqrt[r]\D)_\rho)
\]
be the inclusion of this strata, which is the restriction of $\fk i\co  B_\node\rto \M_{\Gamma_{A,r,\rho}}((\sqrt[r]\D)_\rho)$ to this component.

At last we consider the strata of $\tilde\M_{\Gamma_{A,r,\rho}}((\sqrt[r]\D)_\rho)$. Since it is obtained by performing weight-$r_\node$ blowup of $\M_{\Gamma_{A,r,\rho}}((\sqrt[r]\D)_\rho)$ along the locus of nodal curves $B_\node$, all the strata of $\tilde\M_{\Gamma_{A,r,\rho}}((\sqrt[r]\D)_\rho)$ are 1-to-1 corresponding to all the strata of $\M_{\Gamma_{A,r,\rho}}((\sqrt[r]\D)_\rho)$. For each strata $\M_{\fk\Gamma,(\chi;w)}$ of $\M_{\Gamma_{A,r,\rho}}((\sqrt[r]\D)_\rho)$ we denote its lifting in $\tilde \M_{\Gamma_{A,r,\rho}}((\sqrt[r]\D)_\rho)$ by $\tilde\M_{\fk\Gamma,(\chi;w)}$, which is a $\prod_{\msf e\in \tE(\fk \Gamma)}\integer_{r(\msf e)}$-gerbe over $\M_{\fk\Gamma,(\chi;w)}$ with $r(\msf e)$ being the order of the orbifold structure of the node corresponding to the edge $\msf e$. Meanwhile, the inclusion $\zeta_{\fk\Gamma,(\chi;w)}$ also lifts to an inclusion
\[
\tilde \zeta_{\fk \Gamma,(\chi;w)}\co \tilde \M_{\fk\Gamma,(\chi;w)}\rto\tilde \M_{\Gamma_{A,r,\rho}}((\sqrt[r]\D)_\rho).
\]

The natural maps between $\M_{\Gamma}(\D)$, $\M_{\Gamma_{A,r,\rho}}((\sqrt[r]\D)_\rho)$ and $\tilde\M_{\Gamma_{A,r,\rho}}((\sqrt[r]\D)_\rho)$ can restrict to each strata. We have the following commutative diagram of these natural maps and inclusions of strata
\begin{align}\label{E strataembeding}
\begin{split}
\xymatrix{
\tilde \M_{\Gamma_{A,r,\rho}}((\sqrt[r]\D)_\rho) \ar[r]^-\pi & \M_{\Gamma_{A,r,\rho}}((\sqrt[r]\D)_\rho) \ar[r]^-\epsilon & \M_\Gamma(\D)\\
\tilde \M_{\fk\Gamma,(\chi;w)}\ar[r]^-\pi \ar[u]^-{\tilde\zeta_{\fk\Gamma,(\chi;w)}} & \M_{\fk\Gamma,(\chi;w)}\ar[r]^\epsilon \ar[u]^-{\zeta_{\fk\Gamma,(\chi;w)}} &
\M_{\fk\Gamma,\chi} \ar[u]^-{\zeta_{\fk\Gamma,\chi}}.}
\end{split}
\end{align}

\subsubsection{The formula for $ch((\sqrt[r]\L)_{\Gamma_{A,r,\rho}})$}\label{SSS ch}
Now we write down the formula for $ch((\sqrt[r]\L)_{\Gamma_{A,r,\rho}})$. To use T\"one's Grothendieck--Riemann--Roch formula (\cite{Tone1999}) a technique assumption on $\D$ (so automatically on $(\sqrt[r]\D)_\rho$ for every $r\in\integer_{\geq 1}$) that $\D$ is a quotient orbifold of smooth quasi-projective scheme by a linear algebraic group needs to be added. (This is also where we need the assumption for Theorem \ref{T polynomiality}). Under this assumption, the Chern character $ch((\sqrt[r]\L)_{\Gamma_{A,r,\rho}})$ was computed by Tseng (\cite{Tseng2010}) using T\"one's Grothendieck--Riemann--Roch formula.

By results in \cite{Tseng2010} (see also \cite{Tonita2014}) we have the general formula
\begin{align}\label{E Chern-chara}
ch((\sqrt[r]\L)_{\Gamma_{A,r,\rho}})&=\pi_* \left(ch(\f^*\sqrt[r]\L) Td^\vee({\bar L_{n+1}})\right)
-\sum_{i=1}^n \sum_{k\geq 1} \frac{\msf{ev}_i^*A_k}{k!} \bar\psi_i^{k-1} 
\\
&+\frac{1}{2}(\pi\circ \tilde i)_* \left(\sum_{k\geq 2} \frac{{r_\node^2}}{k!} \cdot\msf{ev}_\node^*A_k\cdot \frac{\bar \psi^{k-1}_+ +(-1)^k\bar\psi^{k-1}_-}{\bar\psi_+ +\bar\psi_-}\right).\nonumber
\end{align}
We next compute each entry for the current situation.
\begin{itemize}
\item Using $\pi\co(\sqrt[r]\D)_\rho\rto\D$ to pull back $\L$ to $(\sqrt[r]\D)_\rho$, we get
    \[
    (\sqrt[r]\L)^{\otimes r}=\pi^*\L.
    \]
    So $c_1(\sqrt[r]\L)=\frac{1}{r}\pi^*c_1(\L)$. In the following we abbreviate $\pi^*c_1(\L)$ as $c_1(\L)$. Therefore
    \[
    ch(\f^*\sqrt[r]\L)=e^{c_1(\f^*\sqrt[r]\L)}=\sum_{k\geq 0}\frac{1}{k!} (\f^*c_1(\sqrt[r]\L))^k=\sum_{k\geq 0}\frac{1}{k!} \left(\frac{\f^*c_1(\L)}{r}\right)^k.
    \]

\item By the construction from Remark \ref{R universl-curve}, $\bar L_{n+1}$ is the cotangent line bundle associated to the $(n+1)$-th marked point for $\tilde\M_{\Gamma_{A,r,\rho}'}((\sqrt[r]\D)_\rho)$; in particular, as this $(n+1)$-th marked point is smooth, i.e., with trivial orbifold structure, and it is the pull back of the cotangent line bundle $\bar L_{n+1}$ over $\M_{\Gamma_{A,r,\rho}'}((\sqrt[r]\D)_\rho)$. By definition we have
    \[
    Td^\vee(\bar L_{n+1})=\frac{\bar\psi_{n+1}}{e^{\bar\psi_{n+1}}-1} =\sum_{k\geq 0}\frac{B_k}{k!}\bar\psi^k_{n+1},
    \]
    where $B_k$ are Bernoulli numbers. Therefore
    \[
    \pi_*\left(ch(\f^*\sqrt[r]\L)Td^\vee(\bar L_{n+1})\right)=
    \sum_{d\geq 0}\sum_{k+l=d} \frac{B_k}{k!l!}
    \pi_*\left(\left(\frac{1}{r}\right)^l\cdot\left(\f^* c_1(\L)\right)^l\cdot\bar\psi^k_{n+1}\right).
    \]
\item $A_k$ is defined in \cite[Definition 4.1.2]{Tseng2010}. We have $A_k\in H^*_{\mathrm{CR}}((\sqrt[r]\D)_\rho)=H^*(\msf I(\sqrt[r]\D)_\rho)$. For each twisted sector $((\sqrt[r]\D)_\rho)_{([g],\xi)}$ of $(\sqrt[r]\D)_\rho$ indexed by $([g],\xi)$, the component of $A_k$ in $H^*(((\sqrt[r]\D)_\rho)_{([g],\xi)})$ is
\[
\sum_{\theta\in S^1} ch\left(\left(\sqrt[r]{\L}\right)^{(\theta)}_{([g],\xi)}\right) B_k(\frac{\log\theta}{2\pi \sqrt{-1}}),
\]
where
\begin{itemize}
\item $\left(\sqrt[r]{\L}\right)^{(\theta)}_{([g],\xi)}$ is the eigen-bundle of eigenvalue $\theta$ of the pullback of $\L$ on $((\sqrt[r]\D)_\rho)_{([g],\xi)}\subseteq\msf I(\sqrt[r]\D)_\rho$,
\item $B_k(x)$ are the Bernoulli polynomials, defined by
\[
\frac{te^{tx}}{e^t-1}=\sum_{k\geq 0}\frac{B_k(x)}{k!} t^k.
\]
\item $\frac{\log\theta}{2\pi\sqrt{-1}}\in[0,1)$.
\end{itemize}
Since $([g],\xi)$ acts on $\sqrt[r]{\L}$ by multiplying $\xi$, the component of $A_k$ in $H^*(((\sqrt[r]\D)_\rho)_{([g],\xi)})$ is
\[
ch\left(\left(\sqrt[r]{\L}\right)^{(\xi)}_{([g],\xi)}\right) B_k\left(\frac{\log\xi}{2\pi\sqrt{-1}}\right),\qq\mathrm{with}\qq
\left(\sqrt[r]{\L}\right)^{(\xi)}_{([g],\xi)} =e_{([g],\xi)}^*\sqrt[r]\L
\]
where $e_{([g],\xi)}\co((\sqrt[r]\D)_\rho)_{([g],\xi)}\rto(\sqrt[r]\D)_\rho$ is the natural evaluation map, and
\[
\frac{\log\xi}{2\pi\sqrt{-1}}=\age_{([g],\xi)}(\sqrt[r]\L).
\]
So the component of $A_k$ in $H^*(((\sqrt[r]\D)_\rho)_{([g],\xi)})$ is
\begin{align*}
e_{([g],\xi)}^*ch(\sqrt[r]\L)\cdot B_k\left({\age_{([g],\xi)}(\sqrt[r]\L)}\right)
&=\left(\sum_{l\geq 0}\frac{(e_{([g],\xi)}^*c_1(\sqrt[r]\L))^l}{l!}\right)\cdot B_k\left({\age_{([g],\xi)}(\sqrt[r]\L)}\right)\\
&=\left(\sum_{l\geq 0} \left(\frac{1}{r}\right)^l\cdot \frac{(e_{([g],\xi)}^*c_1(\L))^l}{l!}\right)\cdot B_k\left({\age_{([g],\xi)}(\sqrt[r]\L)}\right).
\end{align*}

\item In the last term, $r_\node$ is the order of the orbifold structure at the node, $\msf{ev}_\node$ is the evaluation map at the node, $\bar\psi_+$ and $\bar\psi_-$ are the $\bar\psi$-classes associated to the branches of the node. Explicitly, associated to each node, there are two bundle $\bar L_+$ and $\bar L_-$ whose fibers are the cotangent lines of the coarse spaces of the two branches of the node.

\item Finally we could rewrite the last term as follows. First we have the inclusion of locus of the nodes $i\co \scr Z_\node\hrto\scr C$. After the blowup of $\scr C$ along $\scr Z_\node$ with weight $r_\node$, $\scr Z_\node$ becomes a $ \integer_{r_\node}\times \integer_{r_\node}$-gerbe over $\scr Z_\node$, which we denoted by $\tilde{\scr Z}_\node$. The inclusion $i$ lifts to $\tilde i\co \tilde{\scr Z}_\node\hrto\tilde{\scr C}$, which is the $\tilde i$ in the last term. Meanwhile the locus of nodal curves $B_\node$ in $\M_{\Gamma_{A,r,\rho}}((\sqrt[r]\D)_\rho)$ becomes a $\integer_{r_\node}$-gerbe over $B_\node$, which we denoted by $\tilde B_\node$. Then we could rewrite the last term of the right hand side of \eqref{E Chern-chara} by using the inclusion
    \[
    \tilde{\fk i}\co \tilde B_\node\rto\tilde{\M}_{\Gamma_{A,r,\rho}}((\sqrt[r]\D)_\rho)
    \]
    which is the lifting of $\fk i\co B_\node\rto\M_{\Gamma_{A,r,\rho}}((\sqrt[r]\D)_\rho)$. As $\tilde{\scr Z}_\node$ is a $\integer_{r_\node}$-gerbe over $\tilde B_\node$, the last term becomes
    \begin{align*}
    &\frac{1}{2}\,\tilde{\fk i}_* \left(\sum_{k\geq 2} \frac{r_\node}{k!} \cdot\msf{ev}_\node^*A_k\cdot \frac{\bar \psi^{k-1}_+ +(-1)^k\bar\psi^{k-1}_-}{\bar\psi_+ +\bar\psi_-}\right)\\
    =\,&
    \frac{1}{2}\,\tilde{\fk i}_* \left(\sum_{k\geq 1} \frac{r_\node}{(k+1)!} \cdot\msf{ev}_\node^*A_{k+1}\cdot \frac{\bar \psi^k_+ -(-1)^k\bar\psi^k_-}{\bar\psi_+ +\bar\psi_-}\right).
    \end{align*}
\end{itemize}

To write down a formula for $c(-(\sqrt[r]\L)_{\Gamma_{A,r,\rho}})$ via $ch((\sqrt[r]\L)_{\Gamma_{A,r,\rho}}$, we need write down all homogenous components of $ch((\sqrt[r]\L)_{\Gamma_{A,r,\rho}})$. Precisely, the degree-$2d$ component of $ch((\sqrt[r]\L)_{\Gamma_{A,r,\rho}})$ for each $d\geq 1$ is
\begin{align}\label{E ch-d-L-Gamma-r}
ch_d((\sqrt[r]\L)_{\Gamma_{A,r,\rho}}) 
=&\sum_{k+l=d+1} \frac{B_k}{k!l!}\left(\frac{1}{r}\right)^l \pi_*\left(\left(\f^*  c_1(\L)\right)^l\cdot\bar\psi^k_{n+1}\right)\\
-&\sum_{i=1}^n \sum_{k+l=d}\left(\frac{1}{r}\right)^l\cdot \frac{(\msf{ev}_i^* e_{([g_i],\xi_i)}^*c_1(\L))^l\cdot\bar\psi_i^k}{(k+1)!l!}\cdot B_{k+1}({\age_{([g_i],\xi_i)}(\sqrt[r]\L)})\nonumber\\
+&\frac{1}{2}\sum_{([g],\xi)\in\T((\sqrt[r]\D)_\rho)} \tilde\zeta_{([g],\xi),*}
\Bigg( \sum_{\substack{k+l=d\\k\geq 1}}\left(\frac{1}{r}\right)^l\cdot
\frac{ \fk o([g],\xi)\cdot
B_{k+1}({\age_{([g],\xi)}(\sqrt[r]\L)})}{(k+1)!l!} \nonumber\\
&\qq\qq\qq \qq\qq\qq \qq\qq\qq \cdot\msf{ev}_\node^*e_{([g],\xi)}^*c_1(\L)^l\cdot
\frac{\bar \psi^k_+ -(-1)^k\bar\psi^k_-}{\bar\psi_+ +\bar\psi_-}
\Bigg),\nonumber
\end{align}
where
\[
\tilde\zeta_{([g],\xi)}\co \tilde B_{\node,([g],\xi)}\rto \tilde\M_{\Gamma_{A,r,\rho}}((\sqrt[r]\D)_\rho)
\]
is the lifting of the inclusion
\[
\zeta_{([g],\xi)}\co B_{\node,([g],\xi)}\rto \M_{\Gamma_{A,r,\rho}}((\sqrt[r]\D)_\rho),
\]
of the locus of nodal curves with a node whose orbifold structure is given by $([g],\xi)$, and $\fk o([g],\xi)$ is the order of a representative $(g,\xi)$ of $([g],\xi)$.

\subsubsection{Chern class of $-(\sqrt[r]\L)_{\Gamma_{A,r,\rho}}$}\label{SSS chernclass}
Now we could use the formula \eqref{E c-from-ch} to write down a formula of $c(-(\sqrt[r]\L)_{\Gamma_{A,r,\rho}})$ in terms of $ch_d((\sqrt[r]\L)_{\Gamma_{A,r,\rho}}), d\geq 1$. As the last term of $ch_d((\sqrt[r]\L)_{\Gamma_{A,r,\rho}})$ in \eqref{E ch-d-L-Gamma-r} supports over the locus of nodal curves, using \eqref{E c-from-ch}, the formula for $c(-(\sqrt[r]\L)_{\Gamma_{A,r,\rho}})$ has the following expression, where the last term in Chern characters are written as summation over all strata of $\tilde{\M}_{\Gamma_{A,r,\rho}}((\sqrt[r]\D)_\rho)$.
\begin{align}\label{E c-L-r-Gamma-r}
\sum_{\fk \Gamma\in G_{|\Gamma|}(\D)}\sum_{\chi\in{\fk D}_{\fk\Gamma,\Gamma}} \sum_{w\in W_{\fk\Gamma,\chi,r}^{\L,\rho}} \frac{1}{|\aut(\fk\Gamma)|}\\
\tilde \zeta_{\fk\Gamma,(\chi;w),*}\Bigg[
\prod_{\msf v\in\tV(\fk\Gamma)}\exp\left(
\sum_{d\geq 1}\sum_{k+l=d+1} \frac{(-1)^d(d-1)!B_k}{k!l!}\cdot\pi_*
\left(\left(\frac{\f^*c_1(\L)}{r}\right)^l \cdot\bar\psi^k_{n+1}\right)
\right)\nonumber\\
\times \prod_{i=1}^n \exp\left(
\sum_{d\geq 1}\sum_{k+l=d} \frac{(-1)^{d-1}(d-1)!}{(k+1)!l!}\cdot  B_{k+1}\left({\frac{[a_i]_r}{r}}\right)\cdot
\left(\frac{\msf{ev}_i^*e_{([g_i],\xi_i)}^* c_1(\L)}{r}\right)^l\cdot\bar\psi_i^k \right)\nonumber\\
\times \prod_{\substack{\msf e\in\tE(\fk\Gamma)\\
\msf e=(\msf h_+,\msf h_-)}}
\frac{r(\msf e)}{\bar\psi_{\msf h_+}+\bar\psi_{\msf h_-}}\Bigg\{1-\exp\Bigg(
\sum_{d\geq 1}\sum_{k+l=d,k\geq1}\frac{(-1)^{d-1}(d-1)!}{(k+1)!l!}\nonumber\\
\cdot B_{k+1}\left(\frac{w(\msf h_+)+\age_{\chi(\msf h_+)}(\L)}{r}\right)\cdot \left(\frac{\msf{ev}_\node^*c_1(\L)}{r}\right)^l\cdot (\bar\psi_{\msf h_+}^k-(-\bar\psi_{\msf h-}^k))
\Bigg)
\Bigg\}
\Bigg].\nonumber
\end{align}
Where we have used the following facts:
\begin{itemize}
\item for a given $\chi$ and $w$ and the corresponding $\tilde\chi_{(\chi;w)}$, for the $n$ legs $\msf h_i, 1\leq i\leq n$ corresponding to the $n$ marked points, by \eqref{E age-gxi-i}
\[
\age_{\tilde\chi_{(\chi;w)}(\msf h_i)}(\sqrt[r]\L)= \age_{([g_i],\xi_i)}(\sqrt[r]\L)= \frac{[a_i]_r}{r},
\]
\item for other half edges corresponding to nodal points,  by \eqref{E age-tildechi(h)}
\[
\age_{\tilde\chi_{(\chi;w)}(\msf h)}(\sqrt[r]\L) =\frac{w(\msf h)+\age_{\chi(\msf h)}(\L)}{r}
\]
\end{itemize}

Now we cap {$c(-(\sqrt[r]\L)_{\Gamma_{A,r,\rho}})$} with $[\tilde\M_{\Gamma_{A,r,\rho}}((\sqrt[r]\D)_\rho)]^\vir$ and push it forward to $\M_{\Gamma_{A,r,\rho}}((\sqrt[r]\D)_\rho)$ by the blowdown map $\pi\co\tilde \M_{\Gamma_{A,r,\rho}}((\sqrt[r]\D)_\rho)\rto \M_{\Gamma_{A,r,\rho}}((\sqrt[r]\D)_\rho)$. A factor
\[
\frac{1}{\prod_{\msf e\in \tE(\fk\Gamma)} r(\msf e)}
\]
for each summand in \eqref{E c-L-r-Gamma-r}, will show up
from the pushing forward because each strata $\tilde \M_{\fk\Gamma,\chi,w}$ is a $\prod_{\msf e\in \tE(\fk \Gamma)}\integer_{r(\msf e)}$-gerbe over the strata $\M_{\fk\Gamma,\chi,w}$. This kills the corresponding factor in the fourth line in \eqref{E c-L-r-Gamma-r}. Hence the push forward of $c(-(\sqrt[r]\L)_{\Gamma_{A,r,\rho}})\cap [\tilde\M_{\Gamma_{A,r,\rho}}((\sqrt[r]\D)_\rho)]^\vir$ by $\pi_*$ is the cap product of $[\M_{\Gamma_{A,r,\rho}}((\sqrt[r]\D)_\rho)]^\vir$ with
\begin{align*}
\sum_{\fk \Gamma\in G_{|\Gamma|}(\D)}\sum_{\chi\in{\fk D}_{\fk\Gamma,\Gamma}} \sum_{w\in W_{\fk\Gamma,\chi,r}^{\L,\rho}} \frac{1}{|\aut(\fk\Gamma)|}\\
 \zeta_{\fk\Gamma,(\chi;w),*}\Bigg[
\prod_{\msf v\in\tV(\fk\Gamma)}\exp\left(
\sum_{d\geq 1}\sum_{k+l=d+1} \frac{(-1)^d(d-1)!B_k}{k!l!}\cdot\pi_*
\left(\left(\frac{\f^*c_1(\L)}{r}\right)^l \cdot\bar\psi^k_{n+1}\right)
\right)\nonumber\\
\times \prod_{i=1}^n \exp\left(
\sum_{d\geq 1}\sum_{k+l=d} \frac{(-1)^{d-1}(d-1)!}{(k+1)!l!}\cdot B_{k+1}\left({\frac{[a_i]_r}{r}}\right)\cdot
\left(\frac{\msf{ev}_i^*e_{([g_i],\xi_i)}^* c_1(\L)}{r}\right)^l\cdot\bar\psi_i^k \right)\nonumber\\
\times \prod_{\substack{\msf e\in\tE(\fk\Gamma)\\
\msf e=(\msf h_+,\msf h_-)}}
\frac{1}{\bar\psi_{\msf h_+}+\bar\psi_{\msf h_-}}\Bigg\{1-\exp\Bigg(
\sum_{d\geq 1}\sum_{k+l=d,k\geq1}\frac{(-1)^{d-1}(d-1)!}{(k+1)!l!}\nonumber\\
\cdot B_{k+1}\left(\frac{w(\msf h_+)+\age_{\chi(\msf h_+)}(\L)}{r}\right)\cdot \left(\frac{\msf{ev}_\node^*c_1(\L)}{r}\right)^l\cdot (\bar\psi_{\msf h_+}^k-(-\bar\psi_{\msf h-}^k))
\Bigg)
\Bigg\}
\Bigg],\nonumber
\end{align*}
Here in the second line, we abuse notation and use $\zeta_{\fk\Gamma,(\chi;w)}$ to denote the composition of $\tilde\zeta_{\fk\Gamma,(\chi;w)}$ with the blowdown map.

Next we further push it forward to $\M_\Gamma(\D)$ by
the natural projection $\epsilon\co \M_{\Gamma_{A,r,\rho}}((\sqrt[r]\D)_\rho)\rto\M_\Gamma(\D)$ and finally obtain $\tau_*(c(-(\sqrt[r]\L)_{\Gamma_{A,r,\rho}})\cap [\tilde\M_{\Gamma_{A,r,\rho}}((\sqrt[r]\D)_\rho)]^\vir)$, which is the cap product of $[\M_\Gamma(\D)]^\vir$ with
\begin{align}\label{E epsilon-c-Lr}
\sum_{\fk \Gamma\in G_{|\Gamma|}(\D)}\sum_{\chi\in{\fk D}_{\fk\Gamma,\Gamma}} \sum_{w\in W_{\fk\Gamma,\chi,r}^{\L,\rho}} \frac{r^{2\msf g-1-h^1(\fk\Gamma)}}{|\aut(\fk\Gamma)|}\\
\zeta_{\fk\Gamma,(\chi;w),*}\Bigg[
\prod_{\msf v\in\tV(\fk\Gamma)}\exp\left(
\sum_{d\geq 1}\sum_{k+l=d+1} \frac{(-1)^d(d-1)!B_k}{k!l!}\cdot\pi_*
\left(\left(\frac{\f^*c_1(\L)}{r}\right)^l \cdot\bar\psi^k_{n+1}\right)
\right)\nonumber\\
\times \prod_{i=1}^n \exp\left(
\sum_{d\geq 1}\sum_{k+l=d} \frac{(-1)^{d-1}(d-1)!}{(k+1)!l!}\cdot B_{k+1}\left({\frac{[a_i]_r}{r}}\right)\cdot
\left(\frac{\msf{ev}_i^*e_{([g_i],\xi_i)}^*c_1(\L)}{r}\right)^l \cdot\bar\psi_i^k \right)\nonumber\\
\times \prod_{\substack{\msf e\in\tE(\fk\Gamma)\\
\msf e=(\msf h_+,\msf h_-)}}
\frac{1}{\bar\psi_{\msf h_+}+\bar\psi_{\msf h_-}}\Bigg\{1-\exp\Bigg(
\sum_{d\geq 1}\sum_{k+l=d,k\geq1}\frac{(-1)^{d-1}(d-1)!}{(k+1)!l!}\nonumber\\
\cdot B_{k+1}\left(\frac{w(\msf h_+)+\age_{\chi(\msf h_+)}(\L)}{r}\right)\cdot \left(\frac{\msf{ev}_\node^*c_1(\L)}{r}\right)^l\cdot (\bar\psi_{\msf h_+}^k-(-\bar\psi_{\msf h-}^k))
\Bigg)
\Bigg\}
\Bigg].\nonumber
\end{align}
Here again we abuse notation and use $\zeta_{\fk\Gamma,(\chi;w)}$ to denote the map further composite with the natural projection $\epsilon$ (see \eqref{E strataembeding}).

We now explain the factor $r^{2\msf g-1+h^1(\fk\Gamma)}$ which shows up in the first line of \eqref{E epsilon-c-Lr}. The strata $\M_{\fk \Gamma,(\chi;w)}$ can be written as fiber products, along each edge in $\tE(\fk\Gamma)$, of moduli spaces $\M_{\msf g(\msf v),\beta(\msf v),\{\tilde\chi_{(\chi;w)}(\msf h)\}_{\msf h\in\tH(\msf v)}}((\sqrt[r]\D)_\rho)$ of stable maps to $(\sqrt[r]\D)_\rho$ associated to the vertices $\msf v\in \tV(\fk\Gamma)$. For a fixed $\sf v$, it was proved in \cite{Tang-Tseng2016} that the map
\[
\epsilon\co \M_{\msf g(\msf v),\beta(\msf v),\{\tilde\chi_{(\chi;w)}(\msf h)\}_{\msf h\in\tH(\msf v)}}((\sqrt[r]\D)_\rho)\rto \M_{\msf g(\msf v),\beta(\msf v),\{\chi(\msf h)\}_{\msf h\in\tH(\msf v)}}(\D)
\]
is of degree $r^{2\msf g(\msf v)-1}$, and in fact,
\[
\M_{\msf g(\msf v),\beta(\msf v),\{\tilde\chi_{(\chi;w)}(\msf h)\}_{\msf h\in\tH(\msf v)}}((\sqrt[r]\D)_\rho) \cong \left(\prod_{r^{2\msf g(\msf v)}}\M_{\msf g(\msf v),\beta(\msf v),\{\chi(\msf h)\}_{\msf h\in\tH(\msf v)}}(\D)\right)\rtimes\integer_r.
\]
Further we obtain for the strata $\M_{\fk\Gamma,(\chi;w)}$ that
\[
{\M_{\fk\Gamma,(\chi;w)}}
\cong \dot\prod_{\msf v\in\tV(\fk\Gamma),\msf I(\sqrt[r]\D)_\rho}\left[\left(\prod_{r^{2\msf g(\msf v)}}\M_{\msf g(\msf v),\beta(\msf v),\{\chi(\msf h)\}_{\msf h\in\tH(\msf v)}}(\D)\right)\rtimes\integer_r\right]
\]
where $\dot\prod$ means fiber product along edges in $\tE(\fk\Gamma)$ and with $\msf I(\sqrt[r]\D)_\rho$ as the target space. On the other hand,
\[
\M_{\fk\Gamma,\chi}
= \dot\prod_{\msf v\in\tV(\fk\Gamma),\msf I\D} \M_{\msf g(\msf v),\beta(\msf v),\{\chi(\msf h)\}_{\msf h\in\tH(\msf v)}}(\D).
\]
So the degree for the strata projection $\epsilon\co \M_{\fk\Gamma,(\chi;w)}\rto \M_{\fk\Gamma,\chi}$ is
\[
r^{\sum_{\msf v\in \tV(\fk \Gamma)}2\msf g(\msf v)-|\tV(\fk\Gamma)|+|\tE(\fk \Gamma)|},
\]
where $-|\tV(\fk\Gamma)|$ comes from the $\integer_r$-gerbe structure of each $\M_{\msf g(\msf v),\beta(\msf v),\{\tilde\chi_{(\chi;w)}(\msf h)\}_{\msf h\in\tH(\msf v)}}((\sqrt[r]\D)_\rho)$, and $|\tE(\fk\Gamma)|$ comes the $\integer_r$-gerbe $\msf I(\sqrt[r]\D)_\rho\rto\msf I\D$ for each $\msf e\in\tE(\fk\Gamma)$.
We have
\begin{align*}
&\sum_{\msf v\in \tV(\fk \Gamma)}2\msf g(\msf v)-|\tV(\fk\Gamma)|+|\tE(\fk \Gamma)|\\
&=2\left(\sum_{\msf v\in \tV(\fk\Gamma)}\msf g(\msf v)+h^1(\fk \Gamma)\right)-2 h^1(\fk\Gamma)-|\tV(\fk\Gamma)| +|\tE(\fk\Gamma)|\\
&=2\msf g-h^1(\fk \Gamma)-(|\tE(\fk\Gamma)|-|\tV(\fk\Gamma)|+1)-|\tV(\fk\Gamma)| +|\tE(\fk\Gamma)| & (\text{by }\eqref{E genus}\text{ and }\eqref{E h1(Gamma)})\\
&=2\msf g-1-h^1(\fk\Gamma),
\end{align*}
this explains the factor $r^{2\msf g-1-h^1(\fk\Gamma)}$.

Finally we could prove the polynomiality in Theorem \ref{T polynomiality}. Note that
\[
\frac{w(\msf h_+)+\age_{\chi(\msf h_+)}(\L)}{r}=
\left\{\begin{array}{ll}
1-\frac{w(\msf h_-)}r &\textrm{if $\age_{\chi(\msf h_+)}(\L)=0$,}\\
1-\frac{w(\msf h_-)+\age_{\chi({\sf h}_-)}(\L)}{r} &\textrm{if $\age_{\chi(\msf h_+)}(\L)\neq 0$.}
\end{array}\right.
\]
The Bernoulli polynomials satisfy the following property
\[
B_m(x+y)=\sum_{k=0}^m {m \choose k}
B_k(x)y^{m-k}.
\]
This implies that terms of $\tau_*(c(-(\sqrt[r]\L)_{\Gamma_{A,r,\rho}}) \cap[\tilde\M_{\Gamma_{A,r,\rho}}((\sqrt[r]\D)_\rho)]^\vir)$ depend polynomially on $\{w(\msf h) \mid \msf h\in\tH(\fk\Gamma)\}$, the value of the weight function $w\in W_{\fk \Gamma,\chi,r}^{\L,\rho}$ on half-edges. So the proof of \cite[Proposition 3'']{Janda-Pandharipande-Pixton-Zvonkine2017} can be applied to the current case. In particular, when $r\gg1$ the minimal degree in $r$ for $\tau_*(c_d(-(\sqrt[r]\L)_{\Gamma_{A,r,\rho}}) \cap[\tilde\M_{\Gamma_{A,r,\rho}}((\sqrt[r]\D)_\rho)]^\vir)$ in the current case is
\[
(2\msf g-1-h^1(\fk\Gamma))-2d+h^1(\fk\Gamma)=2\msf g-2d-1.
\]
Therefore the cycle
\[
r^{2d-2\msf g+1}\tau_*(c_d(-(\sqrt[r]\L)_{\Gamma_{A,r,\rho}}) \cap[\tilde\M_{\Gamma_{A,r,\rho}}((\sqrt[r]\D)_\rho)]^\vir)
\]
is a polynomial in $r$ when $r\gg 1$. This finishes the proof of Theorem \ref{T polynomiality}.

Furthermore, the same as in \cite{Janda-Pandharipande-Pixton-Zvonkine2017} we could write down the constant term, i.e. lowest order term, of $r^{2d-2\msf g+1}\tau_*(c_d(-(\sqrt[r]\L)_{\Gamma_{A,r,\rho}}) \cap[\tilde\M_{\Gamma_{A,r,\rho}}((\sqrt[r]\D)_\rho)]^\vir)$ in $r$. To this end, we only need to take the lowest degree terms in $r$ in each exponent of the formula \eqref{E epsilon-c-Lr}. Note that
\[
B_1(x)=x-\frac{1}{2},\qq B_2(x)=x^2-x+\frac{1}6,
\]
and that when $r\gg 1$
\[
\frac{[a_i]_r}{r}=\left\{\begin{array}{ll}
\frac{a_i}{r}  &\textrm{$a_i\geq 0$,}\\
1+\frac{a_i}{r} &\textrm{$a_i<0$,}
\end{array}\right.\qq
\mathrm{for}\qq 1\leq i\leq n.
\]
We obtain the following proposition.
\begin{prop}\label{P leading-term}
Suppose that $\D$ is a quotient orbifold of a smooth quasi-projective scheme by a linear algebraic group.  When $r\gg 1$, the constant term of
\[
r^{2d-2\msf g+1}\tau_*(c_d(-(\sqrt[r]{\L})_{\Gamma_{A,r,\rho}}) \cap[\tilde\M_{\Gamma_{A,r,\rho}}((\sqrt[r]\D)_\rho)]^\vir)
\]
is the cap product of $[\M_\Gamma(\D)]^\vir$ with the degree $2d$ part of the following formula
\begin{align}\label{E leading-term}
\sum_{\fk \Gamma\in G_{|\Gamma|}(\D)}\sum_{\chi\in{\fk D}_{\fk\Gamma,\Gamma}} \sum_{w\in W_{\fk\Gamma,\chi,r}^{\L,\rho}} \frac{r^{-h^1(\fk\Gamma)}}{|\aut(\fk\Gamma)|}\\
\zeta_{\fk\Gamma,(\chi;w),*}\Bigg[
\prod_{\msf v\in\tV(\fk\Gamma)}\exp\left(-\frac{1}{2}\pi_*
\left(\left(\f^*c_1(\L)\right)^2\right)
\right)
\times \prod_{i=1}^n \exp\left(
\frac{a_i^2}{2}\bar\psi_i+a_i \msf{ev}_i^*e_{([g_i],\xi_i)}^*c_1(\L)
\right)\nonumber\\
\times \prod_{\substack{\msf e\in\tE(\fk\Gamma)\\
\msf e=(\msf h_+,\msf h_-)}}
\frac{1-\exp\left(
\frac{-(w(\msf h_+)+\age_{\chi(\msf h_+)}(\L))\cdot(w(\msf h_-)+\age_{\chi(\msf h_-)}(\L))\cdot (\bar\psi_{\msf h_+}+\bar\psi_{\msf h-})}{2}\right)
}{\bar\psi_{\msf h_+}+\bar\psi_{\msf h_-}}
\Bigg],\nonumber
\end{align}
where $([g_i],\xi_i)=\Upsilon_{r,\rho}([g_i],a_i)$ for $1\leq i\leq n$.

\end{prop}

We next apply this polynomiality to compute the double ramification cycles with orbifold targets.

\section{Double ramification cycles with orbifold targets}\label{S DR-cycle}

In this section we apply the polynomiality in Theorem \ref{T polynomiality} to calculate the double ramification cycles with orbifold targets. The arrangement of this section is as follows. In \S \ref{SS D-DR-cycle} we first introduce the definition of double ramification cycles with orbifold targets (Definition \ref{D def-DR}), and then we state the key method of computation in Theorem \ref{T DR} based on localization (with proof given in \S \ref{SS localization}) and polynomiality. From Theorem \ref{T DR} and Proposition \ref{P leading-term}, we derive an explicit formula for double ramification cycles with orbifold targets which is stated in Theorem \ref{T formula-of-DR}. This formula also recovers the double ramification cycles
with smooth targets given by \cite{Janda-Pandharipande-Pixton-Zvonkine2018}. The whole subsection \S \ref{SS localization} is devoted to the proof of Theorem \ref{T DR}.

\subsection{A formula for double ramification cycles with orbifold targets}\label{SS D-DR-cycle}

\subsubsection{Definition of double ramification cycles with orbifold targets}\label{SSS D-DR-cycle}

Consider the projectification $\Y=\P(\L\oplus\mc O_\D)$ of $\L$ in Definition \ref{D P(L+O)}. Its 0-section and $\0$-section are $\D_0=\P(0\oplus\mc O_\D)$ and $\D_\0=\P(\L\oplus 0)$, both are isomorphic to $\D$. The normal bundle of $\D_0$ and $\D_\0$ in $\Y$ are $\L$ and $\L^*$ respectively.

We denote by
\[
\Gamma:=(\msf g,\beta, \vec g,\mu_0,\mu_\0)
\]
a topological data for $(\D_0|\Y|\D_\0)$, where
\begin{itemize}
\item[(1)] $\msf g\geq 0$ is the genus of connected (nodal) curves,
\item[(2)] $\beta\in H_2(|\D|;\integer)$ is the homology class that a stable curve represents,
\item[(3)] $\vec g=([g_1],\ldots,[g_m])\in(\T(\Y))^m$  encodes the orbifold information of absolute marked points,
\item[(4)] $\mu_0=\left(([g_{0,1}],\mu_{0,1}),\ldots, ([g_{0,n_0}]),\mu_{0,n_0})\right)\in (\scr T(\D_0)\times\ration_{>0})^{n_0}$ encodes the orbifold information and contact orders of relative marked points mapped to $\D_0$, and
\item[(5)] $\mu_\0=\left(([g_{\0,1}],\mu_{\0,1}),\ldots, ([g_{\0,n_\0}],\mu_{\0,n_\0})\right)\in(\scr T(\D_\0)\times\ration_{>0})^{n_\0}$ encodes the orbifold information and contact orders of relative marked points mapped to $\D_\0$,
\end{itemize}
and require it satisfy the following conditions
\begin{itemize}
\item for $1\leq i\leq m$, there is $\rho({[g_i]})=1$ and then the twisted sector $\Y_{[g_i]}=\P(\L_{[g_i]}\oplus\mc O_{\D_{[g_i]}})$ of $\Y$ is an orbifold $\P^1$-bundle over $\D_{[g_i]}$;

\item $\int_\beta^{\mathrm{orb}}c_1(\L)=|\mu_0|-|\mu_\0|$, where $|\mu_0|=\sum\limits_{i=1}^{n_0}\mu_{0,i}$, $|\mu_\0|=\sum\limits_{i=1}^{n_\0}\mu_{\0,i}$;

\item $\rho({[g_{0,i}]})=e^{2\pi\sqrt{-1}\mu_{0,i}}$, for $1\leq i\leq n_0$;

\item $\bar\rho({[g_{\0,i}]})=e^{2\pi\sqrt{-1}\mu_{\0,i}}$, for $1\leq i\leq n_\0$. 
\end{itemize}

Denote by $\M_\Gamma(\D_0|\Y|\D_\0)^\sim$ the moduli space of stable maps of topological type $\Gamma$ to rubber targets associated to $(\D_0|\Y|\D_\0)$, and denote by $\M_\Gamma(\D_0|\Y|\D_\0)$ the moduli space of stable maps of topological type $\Gamma$ to $(\D_0|\Y|\D_\0)$. Let $\vec\mu_0=\left([g_{0,1}],\ldots, [g_{0,n_0}]\right)$ and $\vec\mu_\0=\left([g_{\0,1}],\ldots, [g_{\0,n_\0}]\right)$, by forgetting the contact orders in $\mu_0$ and $\mu_\0$. The (complex) virtual dimension of $\M_\Gamma(\D_0|\Y|\D_\0)^\sim$ is
\[
\vdimc\,\M_\Gamma(\D_0|\Y|\D_\0)^\sim =\int_\beta^\orb c_1(\D)+(\dim \D-2)(1-\msf g)+n-\iota^\D(\vec g)-\iota^\D(\vec\mu_0)- \iota^\D(\vec\mu_\0)-1.
\]
Here $n=m+n_0+n_\0$ is the number of marked points and $\iota^\D(\cdot)$ denotes the degree shifting (or age) in $\D$.

From the $\Gamma=(\msf g,\beta,\vec g,\mu_0,\mu_\0)$ we get a topological data of $\D$, which we still denote by
\begin{align}\label{E top-data-D}
\Gamma=(\msf g,\beta,\vec g\sqcup \vec \mu_0\sqcup \vec\mu_\0).
\end{align}
We denote by $\M_\Gamma(\D)$ the moduli space of stable maps to $\D$ of type $\Gamma$. The (complex) virtual dimension of $\M_\Gamma(\D)$ is
\[
\vdimc\,\M_\Gamma(\D)= \int_\beta^\orb c_1(\D)+(\dim \D-3)(1-\msf g)+n-\iota^\D(\vec g)-\iota^\D(\vec\mu_0)-\iota^\D(\vec\mu_\0).
\]

There is a natural forgetful map
\[
\epsilon_\D:\M_\Gamma(\D_0|\Y|\D_\0)^\sim\rto \M_\Gamma(\D).
\]

\begin{defn}\label{D def-DR}
The {\bf double ramification cycle (DR-cycle in short)} of type $\Gamma$ for $\L\rto \D$ is defined to be
\[
\msf{DR}_\Gamma(\D,\L):= \epsilon_{\D,*}[\M_\Gamma(\D_0|\Y|\D_\0)^\sim]^\vir\in H_{2(\vdimc\,\M_\Gamma(\D)-g)}(\M_\Gamma(\D)).
\]
\end{defn}

To compute $\msf{DR}_\Gamma(\D,\L)$, we need two moduli spaces: one is the moduli space of relative stable maps to $(\Y_{\D_\0,r}|\D_0)$, the other is the moduli space of (absolute) stable maps to $(\sqrt[r]{\D_\0})_{\bar\rho}=(\sqrt[r]\D)_{\bar\rho}$, both with topological data induced from the topological data $\Gamma=(\msf g,\beta,\vec g,\mu_0,\mu_\0)$ via the lifting $\Gamma_{r,\bar\rho}$ and $\bar\rho$-compatible vectors for $\vec g$, $\vec\mu_0$ and $\vec\mu_\0$. We first describe these $\bar\rho$-compatible vectors.

From constraints on $\Gamma=(\msf g,\beta,\vec g,\mu_0,\mu_\0)$, we see that the following three vectors
\begin{align*}
A_{\vec g,\bar\rho}&:=(0,\ldots,0)\in\ration^m,\\
A_{\vec\mu_0,\bar\rho}&:=(-\mu_{0,1},\ldots,-\mu_{0,n_0})\in \ration^{n_0},\\
A_{\vec\mu_\0,\bar\rho}&:=(\mu_{\0,1},\ldots,\mu_{\0,n_\0})\in \ration^{n_\0}
\end{align*}
are $\bar\rho$-admissible for $\vec g$, $\vec\mu_0$ and $\vec\mu_\0$ respectively. Hence
\begin{align}\label{E A-bar-rho}
A_{\bar\rho}:=(A_{\vec g,\bar\rho}, A_{\vec\mu_0,\bar\rho},A_{\vec\mu_\0,\bar\rho}) =(0,\ldots,0,-\mu_{0,1},\ldots,-\mu_{0,n_0}, \mu_{\0,1},\ldots,\mu_{\0,n_\0})\in\ration^n
\end{align}
is $\bar\rho$-admissible for $\vec g\sqcup\vec\mu_0\sqcup\vec\mu_\0$, and consequently $\bar\rho$-admissible for $\Gamma=(\msf g,\beta,\vec g\sqcup \vec \mu_0\sqcup \vec\mu_\0)$.

The topological data of the moduli space of relative stable maps to $(\Y_{\D_\0,r}|\D_0)$ is constructed from $\Gamma=(\msf g,\beta,\vec g,\mu_0,\mu_\0)$ and $A_{\vec\mu_\0,\bar\rho}$, and the topological data of the moduli space of (absolute) stable maps to $(\sqrt[r]{\D_\0})_{\bar\rho}$ is constructed from $\Gamma=(\msf g,\beta,\vec g\sqcup \vec \mu_0\sqcup \vec\mu_\0)$ and $A_{\bar\rho}$, via the lifting $\Upsilon_{r,\bar\rho}$. We next describe them explicitly.

\subsubsection{A relative moduli space of $(\Y_{\D_\0,r}|\D_0)$ from $\Gamma=(\msf g,\beta,\vec g,\mu_0,\mu_\0)$}\label{SSS Gamma(r)}
Consider the $r$-th root construction of $\Y$ along $\D_\0$, i.e. $\Y_{\D_\0,r}=\P_{r,1}(\L\oplus\mc O_\D)$. We have the natural projection $\pi\co\Y_{\D_\0,r}\rto\D$.
\begin{itemize}
\item The normal bundle of the zero section $\D_0=\P_{r,1}(0\oplus\mc O_\D)\cong \D$ in $\Y_{\D_\0,r}$ is $\L$.
\item The normal bundle of the infinite section $(\sqrt[r]{\D_\0})_{\bar\rho}=\P_{r,1}(\L\oplus 0)$ in $\Y_{\D_\0,r}$
    is the $r$-th root of $\L^*$, i.e. $\sqrt[r]{\L^*}$. So
    $
    (\sqrt[r]{\L^*}
    )^{\otimes r}\cong \pi^*\L^*,
    $
    for $\pi\co (\sqrt[r]{\D_\0})_{\bar\rho}\rto\D$.
\end{itemize}
Next according to $\Gamma=(\msf g,\beta,\vec g,\mu_0,\mu_\0)$, we assign a new topological data
\[
\Gamma(r)=(\msf g,\beta+dF,\vec g\sqcup\vec\mu_\0^r, \mu_0)
\]
to $(\Y_{\D_\0,r}|\D_0)$ as follow:
\begin{enumerate}
\item The genus of domain orbifold curve is $\msf g$.
\item The homology class is $\beta+dF\in H_2(|\Y_{\D_\0,r}|,\integer)=H_2(|\Y|;\integer)$, where $d$ is determined by $|\mu_0|=(\beta+dF)\cdot [\D_0]$ and $\beta$ is viewed as homology class of $|\Y|$ via the inclusion $|\D|=|\D_0|\hookrightarrow |\Y|=|\Y_{\D_\0,r}|$.
\item The absolute marked points are decorated by $\vec g\sqcup\vec\mu_\0^r$, where $\vec \mu_\0^r$ is constructed from the relative data $\mu_\0$ by
    \[
    \vec\mu_\0^r:= \Upsilon_{r,\bar\rho}(\vec\mu_\0,A_{\vec\mu_\0,\bar\rho}) =(\Upsilon_{r,\bar\rho}([g_{\0,i}],\mu_{\0,i}))_{i=1}^{n_\0}.
    \]
    Hence by \eqref{E D-Psi}, we have
    \[
    \Upsilon_{r,\bar\rho}([g_{\0,i}],\mu_{\0,i})= ([g_{\0,i}],e^{2\pi \sqrt{-1}\frac{\mu_\0,i}{r}})
    \]
    for $1\leq i\leq n_\0$.
\item The relative data to $\D_0$ is $\mu_0$.
\end{enumerate}

From $\Gamma(r)$ we get a relative moduli space $\M_{\Gamma(r)}(\Y_{\D_\0,r}|\D_0)$. When $r$ is sufficiently large, e.g. when $r\gg |\mu_\0|$, the (complex) virtual dimension of $\M_{\Gamma(r)}(\Y_{\D_\0,r}|\D_0)$ is
\[
\vdimc\,\M_{\Gamma(r)}(\Y_{\D_\0,r}|\D_0) =\vdimc\,\M_{\Gamma}(\D_0|\Y|\D_\0)^\sim+1.
\]

\subsubsection{A moduli space of $(\sqrt[r]{\D_\0})_{\bar\rho}$ from $\Gamma=(\msf g,\beta,\vec g\sqcup \vec \mu_0\sqcup \vec\mu_\0)$}\label{SSS Gamma-00(r)}
Similarly, we have a topological data for $(\sqrt[r]{\D_\0})_{\bar\rho}$
\begin{align}\label{E Gamma-00-r}
\Gamma_\0(r):=\Gamma_{A_{\bar\rho},r,\bar\rho}= \Upsilon_{r,\bar\rho}(\Gamma,A_{\bar\rho}).
\end{align}
Therefore $\Gamma_\0(r)$ consists of the following data:
\begin{enumerate}
\item the genus $\msf g$;
\item the homology class $\beta\in H_2(|\D|,\integer)$;
\item the orbifold information of absolute marked points is given by $\Upsilon_{r\bar\rho}(\vec g\sqcup \vec \mu_0\sqcup \vec\mu_\0,A_{\bar\rho})=(\Upsilon_{r,\bar\rho}(\vec g,A_{\vec g,\bar\rho}), \Upsilon_{r,\bar\rho}(\vec\mu_0,A_{\vec\mu_0,\bar\rho}) , \Upsilon_{r,\bar\rho}(\vec\mu_\0,A_{\vec\mu_\0,\bar\rho}))$ with
\begin{itemize}
\item $\Upsilon_{r,\bar\rho}(\vec g,A_{\vec g,\bar\rho})=(\ldots, ([g_i],1),\ldots)$, corresponding to the original absolute marked points;
\item $\Upsilon_{r,\bar\rho}(\vec\mu_0,A_{\vec\mu_0,\bar\rho}) =(\ldots, ([g_{0,i}],e^{2\pi\sqrt{-1}\frac{-\mu_{0,i}}{r}}),\ldots)$,
    corresponding to the relative marked points over $\D_0$;
\item $\Upsilon_{r,\bar\rho}(\vec\mu_\0,A_{\vec\mu_\0,\bar\rho}) =(\ldots, ([g_{\0,i}],e^{2\pi\sqrt{-1}\frac{\mu_{\0,i}}{r}}),\ldots)$,
    corresponding to the relative marked points over $\D_\0$.
\end{itemize}
\end{enumerate}

Denote the corresponding moduli space of $(\sqrt[r]{\D_\0})_{\bar\rho}$ by $\M_{\Gamma_\0(r)}((\sqrt[r]{\D_\0})_{\bar\rho})$.

\subsubsection{A formula for the DR-cycle}
Our main theorem in this section is
\begin{theorem}\label{T DR}
Suppose $\D$ is a quotient orbifold of a smooth quasi-projective scheme by a
linear algebraic group. The DR-cycle $\msf{DR}_\Gamma(\D,\L)$ can be computed by
\begin{align}\label{E DR}
\msf{DR}_\Gamma(\D,\L)&=\left[{\tau}_*( -r\cdot c_{\msf g}(-(\sqrt[r]{\L^*})_{\Gamma_\0(r)}) \cap [\tilde\M_{\Gamma_\0(r)} ((\sqrt[r]{\D_\0})_{\bar\rho})]^\vir)\right]_{r^0}.
\end{align}
The right hand side is a polynomial in $r$ when $r\gg 1$, and $[\cdot]_{r^0}$ means its constant term.
\end{theorem}

When $\D$ is smooth, the formula \eqref{E DR} specifies to the one of Janda--Pandharipande--Pixton--Zvonkine (cf.  \cite[(25)]{Janda-Pandharipande-Pixton-Zvonkine2018}).

This theorem is proved by virtual localization on the relative moduli space $\M_{\Gamma(r)}(\Y_{\D_\0,r}|\D_0)$
(for sufficiently large $r$) and we postpone the proof to \S \ref{SS localization}.

Now by Theorem \ref{T DR}, we can apply the calculation result \eqref{E leading-term} in Proposition \ref{P leading-term} under the following situation:
\begin{enumerate}
\item the formula \eqref{E leading-term} is applied to the line bundle $\L^*\rto\D_\0=\D$ with representation $\bar\rho$;
\item the topological data $\Gamma_\0(r)$ in \eqref{E Gamma-00-r} for $(\sqrt[r]{\D_\0})_{\bar\rho}=(\sqrt[r]\D)_{\bar\rho}$  is obtained from $\Gamma=(\msf g,\beta,\vec g\sqcup\vec\mu_0\sqcup\vec\mu_\0)$ and $A_{\bar\rho}$  via Definition \ref{D lift-Gamma}, i.e. $\Gamma_\0(r)=\Upsilon_{r,\bar\rho}(\Gamma,A_{\bar\rho})$ with $A_{\bar\rho}$ given in \eqref{E A-bar-rho}.
\end{enumerate}

This gives us the following explicit formula for $\msf{DR}_\Gamma(\D,\L)$.

\begin{theorem}\label{T formula-of-DR}
Suppose $\D$ is a quotient orbifold of a smooth quasi-projective scheme by a linear algebraic group. The DR-cycle $\msf{DR}_\Gamma(\D,\L)$ is the cap product of $[\M_\Gamma(\D)]^\vir$ with the degree $2\msf g$ part of
\begin{align}\label{E DR-formula}
-\sum_{\fk \Gamma\in G_{|\Gamma|}(\D)}\sum_{\chi\in{\fk D}_{\fk\Gamma,\Gamma}} \sum_{w\in W_{\fk\Gamma,(\chi;w),r}^{\L^*,\bar\rho}} \frac{r^{-h^1(\fk\Gamma)}}{|\aut(\fk\Gamma)|}
\\
\zeta_{\fk\Gamma,(\chi;w),*}\Bigg[\prod_{\msf v\in\tV(\fk\Gamma)}\exp\left(-\frac{1}2\pi_* \left( \left(\f^*c_1(\L^*)\right)^2\right)\right)\times \prod_{i=1}^n \exp\left(
\frac{\bar a_i^2}{2}\bar\psi_i+\bar a_i\msf{ev}_i^* e_{([g_i],\bar\xi_i)}^*c_1(\L^*)\right)
\nonumber\\
\times \prod_{\substack{\msf e\in\tE(\fk\Gamma) \nonumber\\
\msf e=(\msf h_+,\msf h_-)}}
\frac{1-\exp \left( \frac{-(w(\msf h_+)+\age_{\chi(\msf h_+)}(\L^*))\cdot(w(\msf h_-)+\age_{\chi(\msf h_-)}(\L^*))\cdot (\bar\psi_{\msf h_+}+\bar\psi_{\msf h-})}{2}\right) }{\bar\psi_{\msf h_+}+\bar\psi_{\msf h_-}}  \Bigg]
\end{align}
for $r\gg 1$, where $(\bar a_1,\ldots,\bar a_n)=A_{\bar\rho}=(0,\ldots,0,-\mu_{0,1}, \ldots,-\mu_{0,n_0},\mu_{\0,1},\ldots,\mu_{\0,n_\0})$ and
$(\ldots,([g_i],\bar\xi_i),\ldots)=\Upsilon_{r,\bar\rho}(\vec g\sqcup\vec\mu_0\sqcup\vec\mu_\0,A_{\bar\rho})$.
\end{theorem}

Now we proceed to prove Theorem \ref{T DR}.

\subsection{Localization}\label{SS localization}
In this subsection, we carry out virtual localization to the moduli space
$\M_{\Gamma(r)}(\Y_{\D_\0,r}|\D_0)$ and compute the virtual fundamental cycle $[\M_{\Gamma(r)}(\Y_{\D_\0,r}|\D_0)]^\vir$.
The formula derived from localization reduces the calculation of $\msf{DR}_\Gamma(\D,\L)$ to Chern classes of
certain bundle over $\tilde\M_{\Gamma_\0(r)}((\sqrt[r]{\D_\0})_{\bar\rho})$ which was carried out in \S \ref{SS TGW-root-bdls}. Apply it, we complete the proof of Theorem \ref{T DR}.

To do the virtual localization, we consider the natural $\cplane^*$-action on $\Y_{\D_\0,r}$ via the dilation
over $\L$, which induces $\cplane^*$-action on $\M_{\Gamma(r)}(\Y_{\D_\0,r}|\D_0)$. The fixed loci of $\cplane^*$-action on $\Y_{\D_\0,r}$ consists of $(\sqrt[r]{\D_\0})_{\bar\rho}$ and $\D_0$, the fixed lines are the fibers of $\pi\co\Y_{\D_\0,r}\rto \D$.

\subsubsection{Graphs}\label{SSS localize-graphs}
As usual, the fixed loci of the induced $\cplane^*$-action on $\M_{\Gamma(r)}(\Y_{\D_\0,r}|\D_0)$ can be described by certain decorated graphs as those used in strata description in \S \ref{SSS strata}. Here we first describe these graphs.

A decorated graph\footnote{To distinguish the graphs appearing in localization calculation in this subsection and graphs appearing in strata description in \S \ref{SSS strata}, for vertices, edges, etc., in $\Phi$, we use the font $v, e$ instead of $\msf v, \msf e$, etc.} (in localization calculation) is
\[
\Phi=(\tV,\tE,\tL,\tH,\msf g\co\tV\rto\integer_{\geq 0}, \beta\co\tV\rto H_2(|\D|;\integer), \fk l\co\tV\rto\{0,1\}, v\co\tH\rto \tV,\iota\co\tH\rto \tH)
\]
which satisfies the following properties:
\begin{enumerate}
\item $\tV$ is a vertex set with a genus function $\msf g$, a homology class function $\beta$, and a label function $\fk l$. For each $v\in\tV$, the homology class $\beta(v)$ must be an effective curve class of $\D$. We also require the genus and degree conditions to hold:
    \[
    \msf g=\sum_{v\in \tV}\msf g(v)+h^1(\Phi),\qq\mathrm{and} \qq\beta=\sum_{v\in \tV}\beta(v),
    \]
    where $h^1(\Phi)$ is the rank of the degree one homology group of $\Phi$.

\item $\tE$ is the edge set. Every edge corresponds to an orbifold map of the following form\footnote{See \cite[Lemma 5.6]{Chen-Du-Hu2019} for the form of orbifold maps between cyclic group quotients of weighted projective spaces.}, a Galois cover,
    \[\begin{split}
    &\P_{\fkp_e,1}/\integer_{\fk o(g_e)}\rto \P_{r,1}/\langle g_e\rangle, \\ &[x,y]\mapsto [z=x^{d_e},w=y^{\fkq_e}],\qq e^{2\pi\sqrt{-1}\frac{1}{\fk o(g_e)}} \mapsto g_e,
    \end{split}\]
    where
\begin{itemize}
\item $\fkp_e,\fkq_e, d_e\in\integer_{\geq 1}$ satisfy $\frac{\fkq_e}{\fkp_e}=\frac{d_e}{r}$ and $\gcd(\fkp_e,\fkq_e)=1$,
\item $g_e$ is an element of the local group of a point $p\in |\D|=|\D_0|$, $\fk o(g_e)=\ord(g_e)$ is the order of $g_e$ and $\langle g_e\rangle$ is a cyclic subgroup of the local group generated by $g_e$,

\item the $\integer_{\fk o(g_e)}$ acts on $\P_{\fkp_e,1}$ by acting on $x$ with weight $1$,

\item the $\langle g_e\rangle$-action on $\P_{r,1}$ is obtained by identify $\P_{r,1}$ with the fiber of $\Y_{\D_\0,r}$ over $p$, so the $\langle g_e\rangle$ acts on $\P_{r,1}$ by acting on $z$ via multiplying $\rho(g_e)$,
\item the edge degree $d_e$ satisfies $d_e\equiv \fk b(g_e)\mod{\fk o(g_e)}$, where $\fk b(g_e)\in[0,\fk o(g_e)-1]\cap\integer$ is the action weight of $g_e$ on $\L$ and is determined by
    \[
    e^{2\pi\sqrt{-1}\frac{\fk b(g_e)}{\fk o(g_e)}}=\rho(g_e).
    \]
\end{itemize}
Therefore the images of the two orbifold points in $\P_{\fkp_e,1}/\integer_{\fk o(g_e)}$ are mapped as follows
    \begin{align*}
    &\langle e^{2\pi\sqrt{-1}\frac{1}{\fk o(g_e)}} \rangle\ltimes[0,1]\mapsto \langle g_e\rangle \ltimes[0,1]\in(\D_0)_{[g_e]}=\D_{[g_e]},\\
    &\langle e^{2\pi\sqrt{-1}\frac{1}{\fkp\fk o(g_e)}}\rangle\ltimes[1,0]\mapsto \langle g_e\inv, e^{2\pi\sqrt{-1}\frac{\fkq_e}{\fkp_e\fk o(g_e)}}\rangle \ltimes[1,0]\in((\sqrt[r]{\D_\0})_{\bar\rho})_{([g_e\inv], e^{2\pi\sqrt{-1}\frac{\fkq_e}{\fkp_e\fk o(g_e)}})}.
    \end{align*}
    Set
    \begin{align}\label{E D-kappa-e}
    \kappa_e:=\frac{d_e}{\fk o(g_e)}.
    \end{align}
    Note that
    \begin{align}\label{E q/pg=r/mg}
    \frac{\fkq_e}{\fkp_e\fk o(g_e)}=\frac{d_e}{r\fk o(g_e)}=\frac{\kappa_e}{r}.
    \end{align}
    The edge is labeled with $([g_e],\kappa_e)$. It corresponds to a $(1+0+1)$-point, genus zero, fiber class relative moduli space of $((\sqrt[r]{\D_\0})_{\bar\rho}|\Y_{\D_\0,r}|\D_0)$, which we denoted by $\scr F_e$, or more precisely $\scr F([g_e],\kappa_e)$, the contact orders for the relative marked point corresponding to $[1,0]$ is $\kappa_e$ and the contact order for the relative marked point corresponding to $[0,1]$ is $\frac{\fkp_e}{\fkq_e\fk o(g_e)}=\frac{\kappa_e}{r}$.

\item $\tH$ is the set of half edges with a map $v\co\tH\rto\tV$ that assigning half edges to vertices, and an involution $\iota\co\tH\rto\tH$, whose fixed loci is the set of Legs $\tL$,

\item $\tL$, the set of legs, is placed in bijective correspondence with the absolute and relative markings:
    \begin{itemize}
     \item leg $j$ is labeled with $([g_{0,j}],\mu_{0,j})\in \mu_0$ if it is incident to a vertex labeled $0$,
     \item leg $j$ is labeled with $([g_{\0,j}],e^{2\pi \sqrt{-1}\frac{\mu_{\0,i}}{r}})\in\vec\mu_\0^r$ if it is incident to a vertex labeled $\0$,
     \item there are exactly $m$ legs labeled with $[g_i], 1\leq i\leq m$, which are incident to vertices labeled with either $0$ or $\0$.
     \end{itemize}

\item $\Phi$ is a connected graph, and $\Phi$ is bipartite with respect to labeling $\fk l$: every edge is incident to a $0$-labeled vertex and an $\0$-labeled vertex.

\item If $\fk l(v)=0$, denote by $\mu_0(v)$ the list of labels formed by
    \begin{itemize}
    \item the label $([g_{0,j}],\mu_{0,j})$ of the leg $j$ incident to $v$,
    \item the label $([g_e],\kappa_e)$ for the edge $e$ incident to $v$, and
    \item the label $[g_j]$ of the leg $j$ incident to $v$.
    \end{itemize}
    For every such vertex $v$, we impose the condition
    \[
    |\mu_0(v)|=\int_{\beta(v)}^{\mathrm{orb}}c_1(\L),
    \]
    where $|\mu_0(v)|=\sum \mu_{0,j}-\sum \kappa_e$.

\item If $\fk l(v)=\0$, denote by $\mu_\0(v)$ the list of labels formed by
    \begin{itemize}
    \item the label $([g_{\0,j}],e^{2\pi\sqrt{-1}\frac{\mu_{\0,j}}{r}})$ of the leg $j$ incident to $v$,
    \item the label $([g_e],\kappa_e)$ for the edges $e$ incident to $v$, and
    \item the label $[g_j]$ of the leg $j$ incident to $v$.
    \end{itemize}
    For every such vertex $v$, we impose the condition
    \[
    |\mu_\0(v)|=\int_{\beta(v)}^{\mathrm{orb}}c_1(\L) \pmod r,
    \]
    where $|\mu_\0(v)|=\sum \kappa_e-\sum \mu_{\0,j}$, which is a consequence of orbifold Riemann--Roch formula
    \[
    \frac{|\mu_\0(v)|}r-\int_{\beta(v)}^{\mathrm{orb}} c_1(\L^{\frac{1}r})=\frac{|\mu_\0(v)|}r- \frac{\int_{\beta(v)}^{\mathrm{orb}}c_1(\L)}r\in\integer.
    \]
    and \eqref{E q/pg=r/mg}.

\end{enumerate}

\begin{remark}\label{R target-expan?}
\begin{enumerate}
\item If the target space for the moduli space associated to $\Phi$ is expanded, $\tV^0(\Phi)$ contains only one vertex denoted by $v_o$, and for this case we denote by $\M_{v_o}^\sim$ the moduli space of (possible disconnected) stable maps to rubber.
\item If the target space for the moduli space associated to $\Phi$ is not expanded, $\tV^0(\Phi)$ is 1-to-1 correspondence to relative markings ($\mu_0$ over $\D_0)$.
\end{enumerate}
\end{remark}
\begin{remark}\label{R inv-F}
The moduli space $\scr F_e=\scr F([g_e],\kappa_e)$ is a fibration over $\D_{[g_e]}'$ with fiber being the $(1+0+1)$-pointed, genus zero relative moduli space of $\P_{r,1}(\cplane\oplus\cplane)\rtimes\langle g_e\rangle$, whose topological type is determined by the topological type of $\scr F_e$ (cf. \cite{Chen-Du-WangR2019a}). The $\D_{[g_e]}'$ is obtained from $\D_{[g_e]}$ by modulo out the cyclic kernel $\langle g_e\rangle$ of the arrow space of $\D_{[g_e]}$. When $r$ is sufficiently large, e.g. $r>\fk o(g_e)\cdot d_e$, $\scr F_e=\scr F([g_e],\kappa_e)$ is a $\integer_{d_e}$-gerbe over $\D_{[g_e]}'$, and then for $\alpha_0\in H^*(\D_{[g_e]})$ and $\alpha_\0\in H^*(\D_{[g_e\inv]})$, we have
\begin{align}\label{E inv-F}
\int_{\scr F_e} ev_0^*\alpha_0\cup ev_\0^*\alpha_\0=\frac{1}{d_e}\int_{\D_{[g_e]}'}\alpha_0 \cup\alpha_\0 =\frac{\fk o(g_e)}{d_e}\int_{\D_{[g_e]}}\alpha_0\cup\alpha_\0= \frac{1}{\kappa_e}\int_{\D_{[g_e]}}\alpha_0\cup\alpha_\0.
\end{align}
\end{remark}

To every $\0$-labeled vertex $v$ of $\Phi$, we assign the moduli space
\[
\M_{\msf g(v),\beta(v),\mu_\0(v)}((\sqrt[r]{\D_\0})_{\bar\rho}).
\]
We have a natural map
\[
\epsilon\co \M_{\msf g(v),\beta(v),\mu_\0(v)}((\sqrt[r]{\D_\0})_{\bar\rho}) \rto \M_{\msf g(v),\beta(v),\pi(\mu_\0(v))}(\D),
\]
obtained from $\pi\co (\sqrt[r]{\D_\0})_{\bar\rho}\rto\D$. We will use the notation
\[
\M_v((\sqrt[r]{\D_\0})_{\bar\rho}) =\M_{\msf g(v),\beta(v),\mu_\0(v)}((\sqrt[r]{\D_\0})_{\bar\rho}).
\]

\subsubsection{Unstable vertices}\label{SSS unstable-vertex}

A vertex $v\in \tV(\Phi)$ is unstable if $\beta(v)=0$ and $2\msf g(v)-2+n(v) \leq 0$. As we are dealing with relative moduli space, there are four types of unstable vertices:
\begin{itemize}
\item[(i)] $\fk l(v)=\0, \msf g(v)=0$, $v$ carries no markings from $\vec g$ and $\vec\mu_\0^r$ and one incident edge,

\item[(ii)] $\fk l(v)=\0, \msf g(v)=0$, $v$ carries no markings from $\vec g$ and $\vec\mu_\0^r$ and two incident edges,

\item[(iii)] $\fk l(v)=\0, \msf g(v)=0$, $v$ carries one absolute marking from $\vec g$ and one incident edge,

\item[(iv)] $\fk l(v)=\0, \msf g(v)=0$, $v$ carries one absolute marking from $\vec\mu_\0^r$ and one incident edge,

\item[(v)] $\fk l(v)=0, \msf g(v)=0$, $v$ carries one relative marking from $\mu_0$ and one incident edge.
\end{itemize}

\begin{remark}
Here we view all Galois covers as edges, hence there may be unstable vertex of type (iii) and (iv). However, for these two types of vertices, there is no nodal points over $(\sqrt[r]{\D_\0})_{\bar\rho}$.
\end{remark}

As the smooth case, by similar proof of \cite[Lemma 12]{Janda-Pandharipande-Pixton-Zvonkine2018}, we have
\begin{lemma}\label{L unstable-vertex}
When $r\gg 1$, the unstable vertices of type (i), (ii) and (iii) can not occur.
\end{lemma}

In the following, we always assume that $r$ is sufficiently large such that type (i), (ii) and (iii) unstable vertices do not occur, and the relative moduli space $\scr F_e=\scr F([g_e],\kappa_e)$ is a $\integer_{d_e}$-gerbe over $\D_{[g_e]}'$ for every edge $e\in \tE(\Phi)$ and every possible graph $\Phi$.

\subsubsection{Fixed loci}
Now we describe the fixed loci corresponding to the graphs described above. A stable map in the $\cplane^*$-fixed locus corresponding to a graph $\Phi$ is obtained by gluing together maps associated to the vertices $v\in\tV(\Phi)$ with Galois covers associated to the edges $e\in\tE(\Phi)$. Denote by $\tV^\0_{\mathrm{st}}(\Phi)$ the set of $\0$-labeled stable vertices of $\Phi$ and $\tV^\0_{\mathrm{ust}}(\Phi)$ the set of $\0$-labeled unstable vertices of $\Phi$.

\begin{remark}\label{R 3-Phi's}
We divide all graphs $\Phi$ into three types.
\begin{enumerate}
\item A graph $\Phi$ of the first type has $\tV^\0_{\mathrm{st}}(\Phi)=\varnothing$, then the target is expanded and $\Phi$ contains only one $0$-labeled vertex $v_o$ (cf. Remark \ref{R target-expan?}) with edges. Among its edges there are exactly $n_\0$ edges corresponding to $\vec\mu_\0^r$, hence $\mu_\0$; each one of these $n_\0$ edges is labeled by $([g_{\0,i}\inv],\mu_{\0,i})$. Other possible edges correspond to certain $[g_i]$'s in $\vec g$.

    \begin{lemma}\label{L No-edge-for-gi}
    There is no edges corresponding to any $[g_i]$ in $\vec g$.
    \end{lemma}
    \begin{proof}
    Otherwise suppose the contrast. So at least one such absolute marking is distribute to $(\sqrt[r]{\D_\0})_{\bar\rho}$, which will contribute a relative marked point over $\D_\0$ to the rubber component with contact order at least $1$, as $\rho(g_i)=1$. On the other hand, all absolute marking corresponding to $\vec\mu^r_\0$ are distributed to $(\sqrt[r]{\D_\0})_{\bar\rho}$ as they support over $(\sqrt[r]{\D_\0})_{\bar\rho}$; each one of them contribute an edge labeled by $([g_{\0,i}\inv],\mu_{\0,i})$, hence a relative marked point over $\D_\0$ to the rubber component with contact order $\mu_{\0,j}$.
    Therefore the sum of contact orders of the rubber component at $\D_\0$, denoted by $|\mu_\0'|$, is
\[
|\mu_\0'|\geq \sum_{1\leq i\leq n_\0}\mu_{\0,j}+1=|\mu_\0|+1.
\]
On the other hand for such a $\Phi$ the stable maps to the rubber component represent the class $\beta$ when project to $\D$. So we also have
\[
|\mu_0|-|\mu_\0'|=\int_{\beta}c_1(\L).
\]
Finally from the constraint on $\Gamma=(\msf g,\beta,\vec g,\mu_0,\mu_\0)$, we have
\[
|\mu_0|-|\mu_\0|=\int_{\beta}c_1(\L).
\]
Consequently $|\mu_\0'|=|\mu_\0|$, a contradiction.
\end{proof}
Hence for these $\Phi$, all absolute markings corresponding to $\vec g$ are distributed to the rubber component, and there are only $n_\0$ edges:  $e_{\0,i},1\leq i\leq n_\0$, labeled by $([g_{\0,i}\inv],\mu_{\0,i})$ (corresponding to $\mu_\0$). So there is only one graph of this form. We denote this graph by $\Phi_0$. Moreover, for this $\Phi_0$, the topological data for the rubber component $\M_{v_o}^\sim$ is the original $\Gamma$, so we denote this moduli space by $\M_\Gamma^\sim$.

\item The second type consists of only one graph, for which the target space is not expanded. It contains only one stable $\0$-vertex $v_\0$ with $n_0$ edges: $e_i, 1\leq i\leq n_0$, labeled by $([g_{0,i}],\mu_{0,i})$ in $\mu_0$. We denote this graph by $\Phi_\0$. The topological data for the moduli space corresponding to the vertex $v_\0$ is $\Gamma_\0(r)$, and we denote the associated moduli space $\M_{v_\0}((\sqrt[r]{\D_\0})_{\bar\rho})$ by $\M_{\Gamma_\0(r)}((\sqrt[r]{\D_\0})_{\bar\rho})$.
\item Others. $\Phi$ contains one $0$-vertex $v_o$ and nonempty $\tV^\0_{\mathrm{st}}(\Phi)$. We denote a vertex in $\tV^\0_{\mathrm{st}}(\Phi)$ by $v$ and
    \[
    \tE(v)=\{e_{v,1},\ldots,\}.
    \]
    Each edge $e_{v,j}$ is decorated by $([g_{e_{v,j}}], \kappa_{e_{v,j}})$. We denote an unstable vertex in $\tV^\0_{\mathrm{ust}}(\Phi)$ by $w$. As we assume that $r$ is sufficient large, for a $w\in \tV^\0_{\mathrm{ust}}(\Phi)$ we have $|\tE(w)|$=1, we denote the unique edge incident to $w$ by $e_w$.
\end{enumerate}
\end{remark}

Denote the fixed locus corresponding to $\Phi$ by $\M_\Phi$. For the two special graphs $\Phi_0$ and $\Phi_\0$ we have the following two lemmas.
\begin{lemma}
\[
\M_{\Phi_0}=\M_{\Gamma}^\sim\times_{(\msf{ID}_0)^{n_\0}} \prod_{1\leq i\leq n_\0} \scr F_i
\]
with $\scr F_i:=\scr F([g_{\0,i}\inv],\mu_{\0,i}), 1\leq i\leq n_\0$.
\end{lemma}

\begin{lemma}
\[
\M_{\Phi_\0}=\M_{\Gamma_\0(r)}((\sqrt[r]{\D_\0})_{\bar\rho}) \times_{(\msf I(\sqrt[r]{\D_\0})_{\bar\rho})^{n_0}} \prod_i \scr F_i
\]
with $\scr F_i:=\scr F([g_{0,i}],\mu_{0,i}), 1\leq i\leq n_0$.
\end{lemma}

Now consider a $\Phi$ other than $\Phi_0,\Phi_\0$. For each $v\in\tV_{\mathrm{st}}^\0(\Phi)$, define
\[
\scr N_v((\sqrt[r]{\D_\0})_{\bar\rho}):= \M_v((\sqrt[r]{\D_\0})_{\bar\rho}) \times_{(\msf I(\sqrt[r]{\D_\0})_{\bar\rho})^{|\tE(v)|}} \prod_{e\in\tE(v)}\scr F_e
\]
with $\scr F_e=\scr F([g_e],\kappa_e)$.

\begin{lemma}
\[\M_\Phi=\left(\prod_{v\in\tV_{\mathrm{st}}^\0(\Phi)}\scr N_v((\sqrt[r]{\D_\0})_{\bar\rho})\times\prod_{w\in \tV^\0_{\mathrm{ust}}(\Phi)}\scr F_{e_w}\right)\times_{(\msf{ID}_0)^{|\tE(\Phi)|}}\M_{v_o}^\sim,
\]
where we have used the fact that for an unstable vertex $w\in \tV^\0_{\mathrm{ust}}(\Phi)$, $|\tE(w)|=1$.
\end{lemma}

For a $\Phi$ set $\kappa=\kappa_\Phi:=\prod_{e\in \tE(\Phi)}\kappa_e$. Then as now $\scr F_e=\scr F([g_e],\kappa_e)$ is a $\integer_{d_e}$-gerbe over $\D_{[g_e]}'$, we have
\begin{lemma}\label{L 3.13}
\begin{eqnarray}
\M_{\Phi_0} &\cong& \kappa\inv\M_\Gamma^\sim,\\
\M_{\Phi_\0} &\cong& r^{n_0}\kappa\inv\M_{\Gamma_\0(r)} ((\sqrt[r]{\D_\0})_{\bar\rho}),\\
\M_\Phi &\cong& r^{|\Phi|}\kappa\inv\M_{\Phi}',
\end{eqnarray}
with $|\Phi|=\sum_{v\in\tV_{\mathrm{st}}^\0(\Phi)}|\tE(v)|$ and $\M_\Phi'=(\prod_{v\in\tV_{\mathrm{st}}^\0(\Phi)} \M_v((\sqrt[r]{\D_\0})_{\bar\rho})) \times_{(\msf{ID}_0)^{|\Phi|}} \M_{v_o}^\sim$.
\end{lemma}

\subsubsection{Localization formula}

By the virtual localization we have
\begin{align}\label{E loc-formula}
[\M_{\Gamma(r)}(\Y_{\D_\0,r}|\D_0)]^\vir= \sum_{\Phi} \frac{1}{|\aut(\Phi)|}\cdot i_*\left( \frac{[\M_\Phi]^\vir}{e(\cN_\Phi)}\right),
\end{align}
where $i:\M_\Phi\rto \M_{\Gamma(r)}(\Y_{\D_\0,r}|\D_0)$ is the inclusion of fixed loci, and $\cN_\Phi$ is the virtual normal bundle of $\M_\Phi$ in $\M_{\Gamma(r)}(\Y_{\D_\0,r}|\D_0)$.

Let
\[
\msf T\rto \Y_{\D_\0,r}
\]
be the tangent line bundle of the fiber of $\Y_{\D_\0,r}\rto\D$. We also denote by $\msf T$ the pull-back of $\msf T$ from $\Y_{\D_\0,r}$ to the expansion of $\Y_{\D_\0,r}$ along $\D_0$.

Let $[\f:\C\rto \Y_{\D_\0,r}]\in\M_\Phi$. The $\cplane^*$-equivariant Euler class of the virtual normal bundle of $\M_\Phi$ in $\M_{\Gamma(r)}(\Y_{\D_\0,r}|\D_0)$ can be described as
\begin{align}\label{E normal-bdl-A.2}
\frac{1}{e(\cN_\Phi)}= \frac{e(H^1(\C,\f^*\msf T(-\D_0)))}{e(H^0(\C,\f^*\msf T(-\D_0)))}\cdot \frac{1}{\prod_i e(\N_i)}\cdot\frac{1}{e(\N_0)},
\end{align}
where $\N_i$ is a node corresponds to half edge of $\Phi$ adjacent to a $\0$-labeled vertex, and $\N_0$ corresponds to the expansion of the target $\Y_{\D_\0,r}$ over $\D_0$.

We first compute the leading term
\[
\frac{e(H^1(\C,\f^*\msf T(-\D_0)))}{e(H^0(\C,\f^*\msf T(-\D_0)))}.
\]
We use the normalization exact sequence for the domain tensored with the line bundle $\f^*\msf T(-\D_0)$. The associated long exact sequence in cohomology decomposes the leading term into a product of vertex, edge, node contributions:
\begin{itemize}
\item Let $v\in \tV_{\mathrm{st}}^\0(\Phi)$ be a stable vertex over $(\sqrt[r]{\D_\0})_{\bar\rho}\subseteq \Y_{\D_\0,r}$ corresponding to moduli space
    \[
    \M_v((\sqrt[r]{\D_\0})_{\bar\rho}) :=\M_{\msf g(v),\beta(v),\mu_\0(v)}((\sqrt[r]{\D_\0})_{\bar\rho}).
    \]
    As in \S \ref{SS TGW-root-bdls} we have a universal curve
    \[
    \pi\co \tilde{\scr C}\rto\tilde{\M}_v((\sqrt[r]{\D_\0})_{\bar\rho})
    \]
    and a pull back bundle $\f^*\sqrt[r]{\L^*}$ (the pull back of the normal bundle of $(\sqrt[r]{\D_\0})_{\bar\rho}$). Consider the class
\[
e(-(\sqrt[r]{\L^*})_v\otimes\mc O({-\frac{1}r}))
=c_{\mathrm{rk}(v)} (-(\sqrt[r]{\L^*})_v\otimes\mc O({-\frac{1}r}))
\]
in $H^*(\tilde\M_v)\otimes \ration[t,\frac{1}{t}]$, where
\[
(\sqrt[r]{\L^*})_v=\mc R\pi_*\f^*\sqrt[r]{\L^*}
\]
and $\mc O({-\frac{1}r})$ is a trivial line bundle with a $\cplane^*$-action of weight ${-\frac{1}r}$. Since we assume that $r$ is sufficient large (cf. Lemma \ref{L unstable-vertex}), the rank of $-(\sqrt[r]{\L^*})_v$ over $\tilde\M_v((\sqrt[r]{\D_\0})_{\bar\rho})$ is
\[
\mathrm{rk}(v)=\msf g(v)-1+|\tE(v)|.
\]
So
\[
c_{\mathrm{rk}(v)}(-(\sqrt[r]{\L^*})_v\otimes\mc O(-\frac{1}r))=\sum_{0\leq d\leq \mathrm{rk}(v)}c_d(-(\sqrt[r]{\L^*})_v) \left({-\frac{t}r}\right)^{\msf g(v)-1+|\tE(v)|-d}.
\]
The contribution $\frac{e(H^1(\C,\f^*\msf T(-\D_0)))}{e(H^0(\C,\f^*\msf T(-\D_0)))}$ yields the class $\pi_*(c_{\mathrm{rk}(v)}(-(\sqrt[r]{\L^*})_v\otimes\mc O(-\frac{1}r)))$, denoted by
\[
\tilde c_{\mathrm{rk}(v)}(-(\sqrt[r]{\L^*})_v\otimes\mc O(-\frac{1}r))=\sum_{0\leq d\leq \mathrm{rk}(v)}\tilde c_d(-(\sqrt[r]{\L^*})_v) \left({-\frac{t}r}\right)^{\msf g(v)-1+|\tE(v)|-d},
\]
where $\pi\co \tilde \M_v((\sqrt[r]{\D_\0})_{\bar\rho})\rto \M_v((\sqrt[r]{\D_\0})_{\bar\rho})$ is the natural blowdown map.

\item The two possible unstable vertices contribute 1.
\item The edge contribution is trivial since the degree
$\frac{\kappa_e}{r}$ of $\f^*\msf T(-\D_0)$ is less than $1$ for sufficient large $r$.
\item The contribution of a node $\N$ over $\D_\0$ is trivial. Suppose that the edge corresponding this node $\N$ is labeled with $([g_e],\kappa_e)$, then the isotropy of the image of $\N$ is $(g_e\inv,e^{2\pi\sqrt{-1}\frac{\fkq_e}{\fkp_e\fk o(g_e)}})=(g_e\inv,e^{2\pi\sqrt{-1}\frac{\kappa_e}{r}})$. Since we must have $\kappa_e>0$, so $\N$ must be an orbifold node. Then the space of sections $H^0(\N,\f^* \msf T(-\D_0))$ vanishes, and $H^1(\N,\f^*\msf T(-\D_0))$ is trivial for dimension reasons.
Nodes over $\D_0$ contribute $1$.
\end{itemize}

Consider next the last two factors of \eqref{E normal-bdl-A.2},
\[
\frac{1}{\prod\limits_i e(\msf N_i)}\frac{1}{e(\msf N_\0)}.
\]
\begin{itemize}
\item The product $\prod\limits_i e(\msf N_i)\inv$ is over the nodes that correspond to half-edges of the graph $\Phi$ adjacent to a stable $\0$-labeled vertex. If $\msf N$ is such a node corresponding to an edge $e\in \tE(\Phi)$ and the associated vertex $v$ is stable, then
    \[
    \frac{1}{e(\msf N)} =\frac{1}{-\frac{1}{\kappa_e}(t+ev^*_e(c_1(\L))) -\bar\psi_e}=- \frac{\kappa_e}{t+ev^*_e(c_1(\L)) +\kappa_e\bar\psi_e}.
    \]
    This factor corresponds to the smoothing of the node $\msf N$ of the domain curve: $e(\msf N)$ is the first Chern class of the normal line bundle of the divisor of nodal domain curves. There are two parts denoted by $\cplane_{e,+}\otimes\cplane_{e,-}$ corresponding to the two branches mapped into $(\sqrt[r]{\D_\0})_{\bar\rho}$ and fibers respectively. Then
\[
(\tilde\cplane_{e,-})^{\fkq_e}=\L^{*,\frac{1}r},\qq (\tilde\cplane_{e,-})^{\fkp_e\fk o(g_e)}=\cplane_{e,-}.
\]
where $|[\tilde\cplane_{e,-}/\integer_{r\fk o(g_e)}]|=\cplane_{e,-}$ is the coarse space. So since by \eqref{E q/pg=r/mg}, $\frac{\fkq_e}{\fkp_e\fk o(g_e)}=\frac{\kappa_e}{r}$, we have
\[
\cplane_{e,-}=\L^{*,\frac{\fkp_e\fk o(g_e)}{r\fkq_e}}= \L^{*,\frac{1}{\kappa_e}}.
\]
Therefore
\begin{align*}
&e(\cplane_{e,-})=\frac{1}{\kappa_e}(-t-ev_e^*(c_1(\L))), \qq\mathrm{and}\qq e(\cplane_{e,+}\otimes\cplane_{e,-})=\frac{1}{\kappa_e} (-t-ev_e^*(c_1(\L)))-\bar\psi_e,
\end{align*}
with $-\bar\psi_e$ corresponding to $\cplane_{e,+}$.

In the case of an unstable vertex of type (iv), the associated edge does not produce a node of the domain. The type (iv) edge incidences do not appear in
$\prod_ie(\msf N_i)\inv$.

\item $\msf N_0$ corresponds to the expansion of the target $\Y_{\D_\0,r}$ over $\D_0$. The factor $e(\msf N_0)$ is $1$ if the target $(\Y_{\D_\0,r}|\D_0)$ does not expand and
\[
\frac{1}{e(\msf N_0)}=\frac{\kappa} {t+\Psi_0}
\]
if the target expands. 
\end{itemize}

Finally, for each $v\in \tV^\0_{\mathrm{st}}(\Phi)$ we define
\begin{align}\label{E cont-v-00}
\cont_v:=r^{|\tE(v)|} \tilde c_{\mathrm{rk}(v)}(-(\sqrt[r]{\L^*})_v\otimes\mc O(-\frac{1}{r}))\cdot \prod_{e\in\tE(v)}\frac{-\kappa_e}{t+ev_e^*(c_1(\L))+\kappa_e \bar\psi_e}.
\end{align}

Then the contributions of all decorated graphs to the virtual localization formula for the virtual class of $\M_{\Gamma_r}(\Y_r|\D_0)$ are as follows.
\begin{prop}
We have
\begin{eqnarray*}
\cont_{\Phi_0}&=&\frac{1}{t+\Psi_0}\cap [\M_\Gamma^\sim]^\vir,\\
\cont_{\Phi_\0}&=&\kappa\inv\cont_{v_\0}\cap [\M_{\Gamma_\0(r)}((\sqrt[r]{\D_\0})_{\bar\rho}) ]^\vir,\\
\cont_\Phi &=&\frac{1} {|\aut(\Phi)|} \cdot \frac{1}{t+\Psi_0}\cdot \prod_{v\in\tV^\0_{\mathrm{st}}}\cont_v\cap [\M_\Phi']^\vir.
\end{eqnarray*}
\end{prop}

\subsubsection{Proof of Theorem \ref{T DR}}
For each $\Phi$ and a stable vertex $v\in \tV^\0_{\mathrm{st}}(\Phi)$ over $(\sqrt[r]{\D_\0})_{\bar\rho}$ we have
\begin{align*}
\cont_v&=r^{|\tE(v)|} \tilde c_{\mathrm{rk}(v)}(-(\sqrt[r]{\L^*})_v\otimes\mc O(-\frac{1}{r}))\cdot \prod_{e\in\tE(v)}\frac{-\kappa_e}{t+ev_e^*(c_1(\L))+\kappa_e \bar\psi_e}\\
&=\tilde c_{\mathrm{rk}(v)}(-(\sqrt[r]{\L^*})_v\otimes\mc O(-\frac{1}{r}))\cdot \prod_{e\in\tE(v)}\frac{-r\cdot\kappa_e}{t+ev_e^*(c_1(\L))+\kappa_e \bar\psi_e}\\
&=\sum_{0\leq d\leq \mathrm{rk}(v)} \tilde c_d(-(\sqrt[r]{\L^*})_v) \left(-\frac{t}{r}\right)^{\msf g(v)-1+|\tE(v)|-d}\cdot
\left(-\frac{t}{r}\right)^{-|\tE(v)|}
\cdot
\prod_{e\in\tE(v)}\frac{\kappa_e}{1+\frac{ev_e^*(c_1(\L))+\kappa_e \bar\psi_e}{t}}\\
&=\kappa_v\sum_{0\leq d\leq \mathrm{rk}(v)} \tilde c_d(-(\sqrt[r]{\L^*})_v) \left(-\frac{t}{r}\right)^{\msf g(v)-1-d}\cdot
\prod_{e\in\tE(v)}\frac{1}{1+\frac{ev_e^*(c_1(\L))+\kappa_e \bar\psi_e}{t}},
\end{align*}
where $\kappa_v:=\prod_{e\in\tE(v)}\kappa_e$.

Set
\[
\hat c_d:=(-r)^{2d-2\msf g(v)+1}\tilde c_d(-(\sqrt[r]{\L^*})_v).
\]
Then
\begin{align*}
\mathrm{Cont}_v=\kappa_v t\inv \sum_{0\leq d\leq \mathrm{rk}(v)}(-tr)^{\msf g(v)-d}\hat c_d\cdot
\prod_{e\in\tE(v)}\frac{1}{1+\frac{ev_e^*(c_1(\L))+\kappa_e \bar\psi_e}{t}}.
\end{align*}
Set $\widetilde{\cont}_v:=t\cdot\cont_v$. Change the variable $s:=tr,r=r$, we get
\begin{align*}
\widetilde{\mathrm{Cont}}_v=\kappa_v\sum_{0\leq d\leq \mathrm{rk}(v)}(-s)^{\msf g(v)-d}\hat c_d\cdot
\prod_{e\in\tE(v)}\frac{1}{1+\frac{r(ev_e^*(c_1(\L))+\kappa_e \bar\psi_e)}{s}}.
\end{align*}

The virtual class of the moduli space of rubber maps has non-equivariant limit, and $\cplane^*$ acts trivially on $\M_{\Gamma}(\D)$. Therefore the $\cplane^*$-equivariant push-forward
$\epsilon_{\D,*}([\M_{\Gamma(r)}(\Y_{\D_\0,r}|\D_0)]^\vir)$ via the natural map
\[
\epsilon_\D:\M_{\Gamma(r)}(\Y_{\D_\0,r}|\D_0)\rto \M_\Gamma(\D)
\]
is a polynomial in $t$. Hence its coefficient of $t\inv$ is equal to $0$. That is
\begin{align}\label{E coeff-t0=0}
\mathrm{Coeff}_{t^0}\left[\epsilon_{\D,*}(t\cdot [\M_{\Gamma(r)}(\Y_{\D_\0,r}|\D_0)]^\vir)\right]=0.
\end{align}
We have a the commutative diagram
\[
\xymatrix{
\M_v((\sqrt[r]{\D_\0})_{\bar\rho}) \ar[r]^-i \ar[dr]_-{\epsilon} & \M_{\Gamma(r)}(\Y_{\D_\0,r}|\D_0) \ar[d]^-{\epsilon_\D} \\
& \M_v(\D).}
\]
The topological data for $\M_v(\D)$ is determined by the topological data of $\M_v((\sqrt[r]{\D_\0})_{\bar\rho})$ via projecting orbifold information from $\T((\sqrt[r]{\D_\0})_{\bar\rho})$ of marked points to $\T(\D_\0)=\T(\D)$. The genuses and homology classes for both of them are the same.

Note that
\[
\epsilon_*(\tilde c_d((-\sqrt[r]{\L^*})_v)\cap [\M_v((\sqrt[r]{\D_\0})_{\bar\rho})]^\vir)= \tau_*(c_d((-\sqrt[r]{\L^*})_v)\cap [\tilde\M_v((\sqrt[r]{\D_\0})_{\bar\rho})]^\vir).
\]
As $(\sqrt[r]{\D_\0})_{\bar\rho}=(\sqrt[r]\D)_{\bar\rho}$,
so by applying Theorem \ref{T polynomiality} to $\L^*\rto\D$ we see that $\epsilon_*(\hat c_d\cap [\M_v((\sqrt[r]{\D_\0})_{\bar\rho})]^\vir)$ hence $\epsilon_*(\widetilde{\cont}_v\cap [\M_v((\sqrt[r]{\D_\0})_{\bar\rho})]^\vir)$ is a polynomial in $r$ for sufficient large $r$ and rational in $s$.

\begin{coro}\label{C Phi-poly}
$\epsilon_{\D,*}(t\cdot\cont_\Phi)$ is a polynomial in $r$ and rational in $s$ for sufficient large $r$.
\end{coro}
\begin{proof}
For $\Phi_0$ we have
\[
\epsilon_{\D,*}(t\cdot\cont_{\Phi_0})=\epsilon_{\D,*}( \frac{t}{t+\Psi_0}\cap [\M_\Gamma^\sim]^\vir)=\epsilon_{\D,*}(\frac{1}
{1+\frac{r\Psi_0}{s}}\cap[\M_\Gamma^\sim]^\vir),
\]
which is a polynomial in $r$ and rational in $s$.

For $\Phi_\0$ we have
\[
\epsilon_{\D,*}(t\cdot\cont_{\Phi_\0})=\epsilon_{\D,*}( \kappa\inv\widetilde{\cont}_{v_\0}\cap [\M_{\Gamma_\0(r)}((\sqrt[r]{\D_\0})_{\bar\rho}) ]^\vir),
\]
which is a polynomial in $r$ and rational in $s$.

For general $\Phi$, we have
\begin{align*}
t\cdot\cont_{\Phi}&=\frac{1}{|\aut(\Phi)|}\cdot \frac{t}{t+\Psi_0}\cdot \left(\frac{1}{t}\right)^{|\tV^\0_{\mathrm{st}}(\Phi)|}\cdot \prod_{v\in\tV^\0_{\mathrm{st}}(\Phi)}\widetilde{\mathrm{Cont}}_v \cap [\M_\Phi']^\vir
\\
&=\frac{1}{|\aut(\Phi)|}\cdot \frac{1}{1+\frac{r\Psi_0}{s}}\cdot \left(\frac{r}{s}\right)^{|\tV^\0_{\mathrm{st}}(\Phi)|}\cdot \prod_{v\in\tV^\0_{\mathrm{st}}(\Phi)}\widetilde{\cont}_v \cap [\M_\Phi']^\vir.
\end{align*}
So $\epsilon_{\D,*}(t\cdot\cont_{\Phi})$ is a polynomial in $r$ of lowest degree $|\tV^\0_{\mathrm{st}}(\Phi)|$ and rational in $s$.
\end{proof}
Now we extract the coefficient of $t^0$ and then coefficient of $r^0$ in $\epsilon_{\D,*}(t\cdot [\M_{\Gamma(r)}(\Y_{\D_\0,r}|\D_0)]^\vir)$. This is equivalent to extract the coefficient of $s^0r^0$.

By the proof of Corollary \ref{C Phi-poly} we see that only $\Phi_0$ and $\Phi_\0$ contribute the coefficients of $r^0$. Therefore the $r^0$ coefficient is
\begin{align*}
&\mathrm{Coeff}_{r^0}\left[\epsilon_{\D,*}(t\cdot [\M_{\Gamma(r)}(\Y_{\D_\0,r}|\D_0)]^{\mathrm{vir}})\right]\\
&=-\mathrm{Coeff}_{r^0}\left[ \sum_{0\leq d\leq \msf g-1+n_0}\epsilon_*(\hat c_d \cap [\M_{\Gamma_\0(r)}((\sqrt[r]{\D_\0})_{\bar\rho})]^{\mathrm{vir}}) (-s)^{\msf g-d}\right]+ \msf{DR}_{\Gamma}(\D,\L).
\end{align*}
Finally, we take $d=\msf g$ and get
\begin{align*}
&\mathrm{Coeff}_{s^0r^0}\left[\epsilon_{\D,*}(t\cdot [\M_{\Gamma_\0(r)}(\Y_r|\D_0)]^\vir)\right] \\ &=-\mathrm{Coeff}_{r^0}\left[ \epsilon_*(\hat c_{\msf g} \cap [\M_{\Gamma_\0(r)}((\sqrt[r]{\D_\0})_{\bar\rho})]^\vir) \right]+ \msf{DR}_{\Gamma}(\D,\L).
\end{align*}
Then by \eqref{E coeff-t0=0}, $\mathrm{Coeff}_{s^0r^0}\left[\epsilon_{\D,*}(t\cdot [\M_{\Gamma_r}]^\vir)\right]$ vanishes. So we have
\begin{align}\label{E DR-in-proof}
\msf{DR}_{\Gamma}(\D,\L)&=\mathrm{Coeff}_{r^0} \left[\epsilon_*(\hat c_{\msf g} \cap [\M_{\Gamma_\0(r)}((\sqrt[r]{\D_\0})_{\bar\rho})]^\vir) \right]\\
&=\left[\tau_*(-r\cdot c_{\msf g}(-(\sqrt[r]{\L^*})_{\Gamma_\0(r)}) \cap [\tilde\M_{\Gamma_\0(r)} ((\sqrt[r]{\D_\0})_{\bar\rho})]^\vir)\right]_{r^0}.\nonumber
\end{align}
This finishes the proof of Theorem \ref{T DR}.

\begin{remark}\label{R Gamma-0-DR}
In the computation of $\msf{DR}_\Gamma(\D,\L)$ above we take $r$-th root construction on $\Y$ along $\D_\0$. One can also take $r$-th root construction on $\Y$ along $\D_0$ and repeat the computation above. First of all we have $\rho$-compatible vectors for $\vec g,\vec\mu_0$ and $\vec\mu_\0$ given by
\[
A_{\vec g,\rho}=A_{\vec g,\bar\rho}, \qq A_{\vec\mu_0,\rho}=-A_{\vec\mu_0,\bar\rho},\qq A_{\vec\mu_\0,\rho}=-A_{\vec\mu_\0,\rho},
\]
and hence a $\rho$-compatible vector
\[
A_{\rho}=(A_{\vec g,\rho},A_{\vec\mu_0,\rho},A_{\vec\mu_\0,\rho})=-A_{\bar\rho}
\]
for $\Gamma=(\msf g,\beta,\vec g\sqcup\vec\mu_0\sqcup\vec\mu_\0)$. So we now have another moduli space for $(\sqrt[r]{\D_0})_\rho$ with topological type
\[
\Gamma_0(r):=\Gamma_{A_\rho,r,\rho}= \Upsilon_{r,\rho}(\Gamma,A_\rho).
\]
The liftings of marked points are in the following way
\begin{eqnarray*}
\vec g&\rto& \Upsilon_{r,\rho}(\vec g,A_{\vec g,\rho}) =(\ldots,([g_i],1),\ldots),\\
\vec\mu_0&\rto& \Upsilon_{r,\rho}(\vec\mu_0,-A_{\vec\mu_0,\rho})= (\ldots,([g_{0,i}],e^{2\pi\sqrt{-1}\frac{\mu_{0,i}}{r}}),\ldots), \\
\vec\mu_\0&\rto& \Upsilon_{r,\rho}(\vec\mu_\0,-A_{\vec\mu_\0,\rho})= (\ldots,([g_{0,i}],e^{-2\pi\sqrt{-1}\frac{\mu_{\0,i}}{r}}),\ldots).
\end{eqnarray*}
Note that $\Upsilon_{r,\rho}(\vec g,A_{\vec g,\rho})=\Upsilon_{r,\bar\rho}(\vec g,A_{\vec g,\bar\rho})$ as for these $\vec g$, $\rho(\vec g)=\bar\rho(\vec g)=(1,\ldots,1)$.

Then by similar localization calculation as above we could show that
\[
\msf{DR}_\Gamma(\D,\L)=\left[r\cdot c_{\msf g}(-(\sqrt[r]\L)_{\Gamma_0(r)}) \cap[\tilde\M_{\Gamma_0(r)} ((\sqrt[r]{\D_0})_\rho)]^\vir\right]_{r^0}
\]
with $(\sqrt[r]\L)_{\Gamma_0(r)}=\mc R\pi_*\f^*(\sqrt[r]\L)$ and $\f$ being the universal map for the universal curve over $\tilde\M_{\Gamma_0(r)} ((\sqrt[r]{\D_0})_\rho)$. Hence the formula for $\msf {DR}_\Gamma(\D,\L)$ becomes the cap product of $[\M_\Gamma(\D)]^\vir$ with the degree $2\msf g$ part of
\begin{align}\label{E DR-formula-by-D0}
\sum_{\fk \Gamma\in G_{|\Gamma|}(\D)}\sum_{\chi\in{\fk D}_{\fk\Gamma,\Gamma}} \sum_{w\in W_{\fk\Gamma,\chi,r}^{\L,\rho}} \frac{r^{-h^1(\fk\Gamma)}}{|\aut(\fk\Gamma)|}
\\
\zeta_{\fk\Gamma,(\chi;w),*}\Bigg[
\prod_{\msf v\in\tV(\fk\Gamma)}\exp\left(-\frac{1}2\pi_*
\left(\left(\f^*c_1(\L)\right)^2\right)
\right)\times \prod_{i=1}^n \exp\left(\frac{a_i^2}{2}\bar\psi_i+a_i\msf{ev}_i^* e_{([g_i],\xi_i)}^*c_1(\L)\right)
\nonumber\\
\times \prod_{\substack{\msf e\in\tE(\fk\Gamma)\nonumber\\
\msf e=(\msf h_+,\msf h_-)}}
\frac{1-
\exp
\left(
\frac{-(w(\msf h_+)+\age_{\chi(\msf h_+)}(\L))\cdot(w(\msf h_-)+\age_{\chi(\msf h_-)}(\L))\cdot (\bar\psi_{\msf h_+}+\bar\psi_{\msf h-})}{2}\right)}{\bar\psi_{\msf h_+}+\bar\psi_{\msf h_-}}
\Bigg]
\end{align}
for $r\gg 1$, where $(a_1,\ldots,a_n)=A_\rho=(0,\ldots,0,\mu_{0,1},\ldots,\mu_{0,n_0}, -\mu_{\0,1},\ldots,-\mu_{\0,n_\0})$ and $(\ldots,([g_i],\xi_i),\ldots)=\Upsilon_{r,\rho}(\vec g\sqcup \vec\mu_0\sqcup\vec\mu_\0,A_{\rho})$. Hence we have an equality between \eqref{E DR-formula} and \eqref{E DR-formula-by-D0}. Comparing \eqref{E DR-formula} with \eqref{E DR-formula-by-D0} we see
\[
a_i=-\bar a_i, \qq \xi_i=\bar\xi_i\inv.
\]

When $(\D,\L)=(X,L)$ is smooth, the formula \eqref{E DR-formula-by-D0} coincides with the one obtained by Janda--Pandharipande--Pixton--Zvonkine in \cite{Janda-Pandharipande-Pixton-Zvonkine2018}.
\end{remark}

\subsection{A cycle version of Leray--Hirsch result}
As an application of the computation of DR-cycles for $\L\rto\D$ we could prove a cycle version of Leray--Hirsch result for orbifold Gromov--Witten theory obtained in \cite{Chen-Du-WangR2019a} under the assumption that $\D$ is a quotient orbifold of a smooth quasi-projective scheme by a linear algebraic group.

\begin{theorem}
When $\D$ is a quotient orbifold of a smooth quasi-projective scheme by a linear algebraic group, the formula for DR-cycles with target $\D$ calculates the push-forward to the moduli space of orbifold maps to $\D$ of the virtual fundamental classes of the moduli spaces of orbifold stable maps to
\[
(\Y|\D_0),\qq (\Y|\D_\0),\qq (\D_0|\Y|\D_\0)
\]
in terms of tautological classes and $c_1(\L)$.
\end{theorem}

\begin{proof}
We apply virtual localization w.r.t. the $\cplane^*$-action on $\Y=\P(\L\oplus\mc O_\D)$ to calculate the virtual fundamental classes of moduli spaces of orbifold stable maps to these three targets. We take the first one as an example.

As in \S \ref{SS localization}, the localization formula express the virtual cycle of $\M_\Gamma(\Y|\D_0)$ into contributions from simple fixed loci (for which the targets are not expanded) and composite fixed loci (for which the targets are expanded). After pushing forward to $\M_\Gamma(\D)$, the contribution from the simple fixed loci is already of the desired form. So we only have to consider the composite fixed loci. The composite fixed loci is a composition of simple fixed loci and moduli space of stable orbifold maps to the rubber targets $(\D_0|\Y|\D_\0)$ and powers of $\Psi_0$ class. We only have to consider the contribution from the rubber components and $\Psi_0$. By the rubber calculus described in \cite[Section 4.2]{Chen-Du-WangR2019a}, we can remove those $\Psi_0$ classes. Then \eqref{E DR-formula} (equivalently \eqref{E DR-formula-by-D0}) proves the theorem.
\end{proof}

\section{Relative v.s. absolute orbifold Gromov--Witten invariants}\label{S rel-abs-GWI}

In this section we apply the polynomiality in Theorem \ref{T polynomiality} and the localization analysis in \S \ref{SS localization} to present a relation between the relative orbifold Gromov--Witten (GW for short) invariants and the absolute GW-invariants of root constructions. 
The main result of this section also appears in \cite[Theorem 8]{Tseng-You2020a} which is stated in the cycle level.

\subsection{Notations and the result}
Let $(\X|\D)$ be an orbifold relative pair such that $\D$ is a divisor of $\X$. We first collect come notations and state the main result in this section.
\subsubsection{Inertia spaces}
Denote the index sets of inertia orbifolds of $\X$ and $\D$ by $\T(\X)$ and $\T(\D)$ respectively. As $\D$ is a sub-orbifold of $\X$, the inertia space $\msf I\D$ is a sub-orbifold of the inertia space $\msf I\X$. Hence
\[
\T(\D)\subseteq\T(\X).
\]
Denote by $\L$ the normal line bundle of $\D$ in $\X$. So we have a representation $\rho\co D^1\rto U(1)$ associated to $\L$. Then for every $[g]\in\scr T(\D)$ when $\rho([g])\neq 1$, we have
\[
\X_{[g]}=\D_{[g]}.
\]
We set
\[
\scr T(\D)_+:=\{[g]\in\scr T(\D)\mid \rho([g])\neq 1\},
\]
and
\[
\scr T(\X)_0:=\scr T(\X)\setminus\scr T(\D)_+.
\]
We next consider relative orbifold GW-invariants of $(\X|\D)$.

\subsubsection{Relative orbifold GW-invariants of $(\X|\D)$}

As in \S \ref{SS D-DR-cycle} let $\Gamma=(\msf g,\beta,\vec h,\mu)$ be a relative topological data of $(\X|\D)$ with
\begin{enumerate}
\item $\msf g$ the genus, $\beta\in H_2(|\X|;\integer)$ the homology class,

\item $\vec h=([h_1],\ldots,[h_m])\in (\T(\X)_0)^m$ the orbifold information of absolute marked points,
\item $\mu=\left(([g_1],\mu_1),\ldots,([g_n],\mu_n)\right)\in (\T(\D)\times \ration_{>0})^n$ the orbifold information and contact orders of relative marked points, satisfying
    \[
    |\mu|:=\sum_{j=1}^n \mu_j=\int_\beta^{\mathrm{orb}}[\D],
    \]
    and
    \[
    \rho(g_i)=e^{2\pi\sqrt{-1} \mu_i},\qq\mathrm{for}\qq 1\leq i\leq n.
    \]
\end{enumerate}
Let $\vec \mu=([g_1],\ldots,[g_n])$ by forgetting the contact orders in $\mu$. So
\begin{align}\label{E A-mu-rho}
A_{\vec\mu,\rho}=(\mu_1,\ldots,\mu_n)\in\ration^n
\end{align}
is $\rho$-admissible for $\vec\mu$.

Denote the moduli space of relative orbifold stable maps to $(\X|\D)$ of topological type $\Gamma$ by $\M_\Gamma(\X|\D)$. From $\Gamma=(\msf g,\beta,\vec h,\mu)$ we also get a topological type $(\msf g,\beta,\vec h\sqcup\vec\mu)$, still denoted by $\Gamma$, of $\X$ by viewing $[g_i]\in \T(\D)\subseteq \T(\X)$. We denote the moduli space of absolute orbifold stable maps to $\X$ of topological type $\Gamma$ by $\M_\Gamma(\X)$. We have a natural map
\[
\epsilon\co \M_\Gamma(\X|\D)\rto\M_\Gamma(\X)
\]
by first projecting those components mapped to the rubber target $\P(\L\oplus\mc O_\D)$ of $(\D,\L)$ to $\D$ and then stabilizing the domain curves. Over $\M_\Gamma(\X)$ we have the psi-classes $\bar\psi_i$ corresponding to the $(m+n)$ absolute marked points. It is the Chern class of the line bundle over $\M_\Gamma(\X)$, whose fiber over a stable map $\f\co \msf C\rto \X$ is the cotangent line of the coarse space of the domain curve at the $i$-th marking. We set
\[
\bar\psi_i:=\epsilon^*\bar\psi_i
\]
over $\M_\Gamma(\X|\D)$.

A relative orbifold GW-invariant is of the form
\begin{align}\label{E ROGW-XD}
\Big\langle \underline\alpha
\,\Big| \, \underline\mu\Big\rangle^{\X|\D}_\Gamma:=\int_{[\M_\Gamma(\X|\D)]^\vir} \prod_{i=1}^m \msf{ev}_i^*(\alpha_i) \bar\psi_i^{a_i}\wedge \prod_{j=1}^n\msf{rev}_j^*(\theta_j)\bar\psi_{m+j}^{b_j}
\end{align}
where
\begin{itemize}
\item $\underline\alpha=(\bar\psi^{a_1}\alpha_1,\ldots, \bar\psi^{a_m}\alpha_m)\in(\cplane[\bar\psi]\otimes H^*_{\mathrm{CR}}(\X))^m,\, \underline\mu=(\bar\psi^{b_1}\theta_1,\ldots, \bar\psi^{b_n}\theta_n)\in (\cplane[\bar\psi]\otimes H^*_{\mathrm{CR}}(\D))^n$, with $\alpha_i\in H^*(\X_{[h_i]})$ and $\theta_j\in H^*(\D_{[g_j]})$.

\item $\msf{ev}_i$ and $\msf {rev}_j$ are evaluation maps at absolute and relative marked points respectively.

\item $\bar\psi_i$ and $\bar\psi_{m+j}$ are the psi-classes of the corresponding absolute and relative marked points respectively.
\end{itemize}

\subsubsection{Absolute orbifold GW-invariants of root construction}\label{SSS aogw-Xr}
Let $\X_{\D,r}$ be the $r$-th root construction of $\X$ along $\D$, with exceptional divisor $(\sqrt[r]\D)_\rho$, a $\integer_r$-gerbe over $\D$. Denote by $\pi\co\X_{\D,r}\rto\X$ the natural projection, which induces a morphism over inertia spaces $\msf I\pi\co \msf I\X_{\D,r}\rto\msf I\X$.

As in \S \ref{SSS Gamma(r)} from $\Gamma$ of $(\X|\D)$ we get a topological data of stable maps to $\X_{\D,r}$
\begin{align}\label{E gamma-r-in-Sec3}
\Gamma(r)=(\msf g,\beta,\vec h\sqcup\vec\mu^r),
\end{align}
by the following convention:
\begin{itemize}
\item For $[h_i]\in \vec h$,
\begin{itemize}
\item when $[h_i]\notin \T(\D)$, $\X_{[h_i]}$ lifts to a twisted sector of $\X_{\D,r}$, we leave it unchanged;
\item when $[h_i]\in \T(\D)\setminus\T(\D)_+$, it lifts to $\Upsilon_{r,\bar\rho}([h_i],0) =([h_i],1)\in\T((\sqrt[r]\D)_\rho) \subseteq\T(\X_{\D,r})$.
\end{itemize}
\item The $\vec\mu^r$ is
\begin{align}\label{E vec-mu-r}
\vec\mu^r=\Upsilon_{r,\rho}(\vec\mu,A_{\vec\mu,\rho})=(\ldots, ([g_j],e^{2\pi i\frac{\mu_j}{r}}),\ldots)
\end{align}
with $A_{\vec\mu,\rho}$ given by \eqref{E A-mu-rho}. We denote $([g_j],e^{-2\pi i\frac{\mu_j}{r}})$ by $([g_j],\xi_j), 1\leq j\leq n$ for simplicity.
\end{itemize}
Then we lift $\underline\alpha$ and $\underline\mu$ as follow
\begin{itemize}
\item for $\alpha_i$, we take the component of $\msf I\pi^*\alpha_i$ over $(\X_{\D,r})_{[h_i]}$ or $(\X_{\D,r})_{([h_i],1)}$,
\item for $\theta_j$, we take the component of $\msf I\pi^*\theta_j$ over $((\sqrt[r]\D)_\rho)_{([g_j],\xi_j)}$.
\end{itemize}
To save notations we still denote these liftings by $\underline\alpha$ and $\underline\mu$ respectively. Then we get an absolute orbifold GW-invariants of $\X_r$:
\begin{align}\label{E AOGW-Xr}
\Big\langle \underline\alpha,\underline\mu\Big\rangle^{\X_{\D,r}}_{\Gamma(r)} :=\int_{[\M_{\Gamma(r)}(\X_{\D,r})]^\vir} \prod_{i=1}^m\msf{ev}_i^*\alpha_i\bar\psi_i^{a_i} \wedge \prod_{j=1}^n\msf{ev}^*_{m+j}\theta_j\bar\psi_{m+j}^{b_j}.
\end{align}

\subsubsection{Main result in this section}
Now we can state our main result of this section.
\begin{theorem}\label{Thm abs-rel}
Suppose $\D$ is a quotient orbifold of a smooth quasi-projective scheme by a linear algebraic group. When $r\gg 1$, $\bl \underline\alpha,\underline\mu\br^{\X_{\D,r}}_{\Gamma(r)}$ is a polynomial in $r$, and the constant term satisfies
\[
\left[\Big\langle \underline\alpha,\underline\mu\Big\rangle^{\X_{\D,r}}_{\Gamma(r)}
\right]_{r^0}=
\Big\langle \underline\alpha\,\Big|\, \underline\mu\Big\rangle^{\X|\D}_{\Gamma}
\]
where $[\cdot]_{r^0}$ means the constant term of a polynomial in $r$.
\end{theorem}
We will prove this theorem in the rest of this section.

\subsection{Reducing to local model by degeneration}
In this subsection we use degeneration formula (cf. \cite{Chen-Li-Sun-Zhao2011,Abramovich-Fantechi2016}) to reduce the proof of Theorem \ref{Thm abs-rel} to local model.

Let $\Y=\P(\L\oplus\mc O_\D)$ be the projectification of the normal line bundle $\L$ of $\D$ in $\X$. Let $\Y_{\D_0,r}=\P_{r,1}(\L^*\oplus\mc O_D)$ be the $r$-th root construction of $\Y$ along its $0$-section $\D_0$. Then the $0$-section of $\Y_{\D_0,r}$ is $(\sqrt[r]{\D_0})_\rho$ with normal line bundle $\sqrt[r]\L$. The $\0$-section of $\Y_{\D_0,r}$ is still $\D_\0\cong \D$.

\subsubsection{Degeneration of $\X_{\D,r}$}\label{SSS deg-XDr}
We first consider the following degeneration of $\X_{\D,r}$ along $(\sqrt[r]\D)_\rho$
\[
\X_{\D,r} \xrightarrow{\mathrm{degenerate}} (\X|\D)\wedge_\D(\Y_{\D_0,r}|\D_\0).
\]
The gluing is along $\D_\0$. Then the degeneration formula gives rise to
\begin{align}\label{E dege-Xr-inv}
\Big\langle\underline\alpha,\underline\mu \Big\rangle^{\X_{\D,r}}_{\Gamma(r)}= \sum_{(\Gamma(r)^\pm)}c(\Gamma(r)^\pm) \cdot \Big\langle\underline\alpha^-\,\Big|\, \underline{\check\eta}\Big\rangle^{\X|\D}_{\Gamma(r)^-}\cdot \Big\langle\underline\alpha^+,\underline\mu\,\Big|\, \underline{\eta}\Big\rangle^{\Y_{\D_0,r}|\D_\0}_{\Gamma(r)^+},
\end{align}
where
\begin{itemize}
\item the summation is taken over all possible splitting $\Gamma(r)^\pm=(g^\pm,\beta^\pm,\vec h^\pm\sqcup(\vec\mu^r)^\pm,\eta^\pm)$ of $\Gamma$, with
    \begin{itemize}
    \item $\eta^+=\eta=(([k_1],\eta_1),\ldots([k_\tau],\eta_\tau)) \in (\T(\D_\0)\times\ration)^\tau= (\T(\D)\times\ration)^\tau$ being a partition of $\beta_+\cdot[\D_\0]$, and
    \item $\eta^-=\check\eta=(([k_1\inv],\eta_1), \ldots,([k_\tau\inv],\eta_\tau))$,
    \item $(\vec\mu^r)^-=\varnothing$ since they all support over $\D_r$;
    \end{itemize}
\item $\underline\eta=(\delta_1,\ldots,\delta_{\tau})$ with $\delta_i\in H^*((\D_\0)_{[k_i]})=H^*(\D_{[k_i]})$ is a cohomological weighted partition corresponds to $\eta$ and $\underline{\check\eta}$ is the dual cohomological weighted partition of $\underline\eta$;
\item the constant $c(\Gamma_r^\pm)=\frac{\prod_j\eta_j}{|\aut(\eta)|}$.
\end{itemize}

\subsubsection{Degeneration of $(\X|\D)$.}

For $(\X|\D)$ we can also degenerate $\X$ along $\D$ to get
\[
(\X|\D)\xrightarrow{\mathrm{degenerate}} (\X|\D)\wedge_\D(\D_\0|\Y|\D_0)
\]
where the gluing is along $\D_\0$. Then
\begin{align}\label{E dege-XD-inv}
\Big\langle\underline\alpha\,\Big|\, \underline\mu\Big\rangle^{\X|\D}_{\Gamma}= \sum_{(\Gamma^\pm)}c(\Gamma^\pm)\cdot \Big\langle\underline\alpha^-\,\Big|\,\underline{\check\eta} \Big\rangle^{\X|\D}_{\Gamma^-}\cdot \Big\langle\underline\eta\,\Big|\,\underline\alpha^+\,\Big|\, \underline\mu\Big\rangle^{\D_\0|\Y|\D_0}_{\Gamma^+},
\end{align}
where as in \S \ref{SSS deg-XDr}
\begin{itemize}
\item the summation is taken over all possible splitting $\Gamma^+=(g^+,\beta^+,\vec h^+,\eta^+,\mu)$ and $\Gamma^-=(g^-,\beta^-,\vec h^-,\eta^-)$ of $\Gamma$, with
    \begin{itemize}
    \item $\eta^+=\eta=(([k_1],\eta_1),\ldots([k_\tau],\eta_\tau)) \in (\T(\D_\0)\times\ration)^\tau= (\T(\D)\times\ration)^\tau$ being a partition of $\beta_+\cdot[\D_\0]$, and
    \item $\eta^-=\check\eta=(([k_1\inv],\eta_1), \ldots,([k_\tau\inv],\eta_\tau))$;
    \end{itemize}
\item $\underline\eta=(\delta_1,\ldots,\delta_{\tau})$ with $\delta_i\in H^*((\D_\0)_{[k_i]})=H^*(\D_{[k_i]})$ is a cohomological weighted partition corresponds to $\eta$ and $\underline{\check\eta}$ is the dual cohomological weighted partition of $\underline\eta$;
\item the constant $c(\Gamma^\pm)=\frac{\prod_j\eta_j}{|\aut(\eta)|}$.
\end{itemize}

\subsubsection{Comparison between local models}

Recall that $\Y_{\D_0,r}$ is the $r$-th root construction of $\Y$ along $\D_0$. Hence along the way that we match invariants \eqref{E ROGW-XD} and \eqref{E AOGW-Xr} we could match invariants of $(\D_0|\Y|\D_\0)$ with invariants of $(\Y_{\D_0,r}|\D_\0)$. Then by comparing the summands in \eqref{E dege-Xr-inv} and \eqref{E dege-XD-inv} we have the following lemma.
\begin{lemma}
There is a 1-to-1 correspondence between the summands in \eqref{E dege-Xr-inv} and \eqref{E dege-XD-inv}, under which \begin{itemize}
\item the datum on the ``$-$'' side, i.e. $(\X|\D)$ side, are the same, and
\item the datum on the ``$+$'' side are matched via the way that we match \eqref{E ROGW-XD} and \eqref{E AOGW-Xr}.
\end{itemize}
Hence for every matched pair of summands in \eqref{E dege-Xr-inv} and \eqref{E dege-XD-inv}, we have $\Gamma(r)^-=\Gamma^-$ and $\Gamma(r)^+$ is obtained from $\Gamma^+$ via the convention in \S \ref{SSS aogw-Xr}, see \eqref{E gamma-r-in-Sec3}.

\end{lemma}

So to prove Theorem \ref{Thm abs-rel} we only have to match the invariants of the ``$+$'' side of the degenerations. Explicitly, we reduce Theorem \ref{Thm abs-rel} to
\begin{lemma}\label{L equal-local-model}
Suppose $\D$ is a quotient orbifold of a smooth quasi-projective scheme by a linear algebraic group. When $r\gg 1$, $\Big\langle\underline\alpha,\underline\mu\,\Big|\, \underline{\eta}\Big\rangle^{\Y_{\D_0,r}|\D_\0}_{\Gamma(r)}$ is a polynomial in $r$ and
\[
\left[\Big\langle\underline\alpha,\underline\mu\,\Big|\, \underline{\eta} \Big\rangle^{\Y_{\D_0,r}|\D_\0}_{\Gamma(r)}\right]_{r^0}= \Big\langle\underline\mu\,\Big|\, \underline\alpha\,\Big|\, \underline\eta \Big\rangle^{\D_0|\Y|\D_\0}_{\Gamma},
\]
where for the topological data $\Gamma$ for $(\D_0|\Y|\D_\0)$ and the insertions $\underline\alpha,\underline\mu$ and $\underline\eta$, see the beginning of \S \ref{SSS 4.3.1}.
\end{lemma}

\subsection{Local model}\label{SS local-model}
In this subsection we prove Lemma \ref{L equal-local-model}.

\subsubsection{Setup and some preliminary results}\label{SSS 4.3.1}

First we relate both the relative invariants of $(\D_0|\Y|\D_\0)$ and relative invariants of $(\Y_{\D_0,r}|\D_\0)$ to rubber invariants of $(\D_0|\Y|\D_\0)$.

We now use
\[
\Gamma=(\msf g,\beta,\vec h,\mu,\eta),
\]
to denote topological data of relative stable maps to $(\D_0|\Y|\D_\0)$, where
\begin{itemize}
\item $\msf g$ is the genus, $\beta\in H_2(|\Y|;\integer)$ is the homology class,
\item $\vec h=([h_1],\ldots,[h_m])\in \T(\D)^m\subseteq \T(\Y)^m$ with $\rho([h_i])=0$,
\item $\mu=(([g_1],\mu_1),\ldots, ([g_n],\mu_n))$ encodes  the orbifold information and contact orders of relative marked points mapped to $\D_0$ and
\item $\eta=(([k_1],\eta_1),\ldots, ([k_\tau],\eta_\tau))$ encodes the orbifold information and contact orders of relative marked points mapped to $\D_\0$.
\end{itemize}
This is similar to the $\Gamma$ in \S \ref{SS D-DR-cycle}. Then as in \S \ref{SSS aogw-Xr}, \eqref{E gamma-r-in-Sec3}, we get a topological type $\Gamma(r)$ of relative stable maps to $(\Y_{\D_0,r}|\D_\0)$ by changing $\mu$ into $\vec\mu^r=\Upsilon_{r,\rho}(\vec\mu,A_{\vec\mu,\rho})$ as \eqref{E vec-mu-r} with $\vec\mu=([g_1],\ldots,[g_n])$ and $A_{\vec\mu,\rho}=(\mu_1,\ldots,\mu_n)$. These two topological types are the $\Gamma$ and $\Gamma(r)$ in Lemma \ref{L equal-local-model}. And the insertions in Lemma \ref{L equal-local-model} are $\underline\alpha=(\bar\psi^{a_1}\alpha_1,\ldots, \bar\psi^{a_m}\alpha_m)$ for $\vec h$, $\underline\mu=(\bar\psi^{b_1}\theta_1,\ldots, \bar\psi^{b_n}\theta_n)$ for $\mu$ and $\underline\eta=(\bar\psi^{c_1}\delta_1,\ldots,\bar\psi^{c_\tau} \delta_\tau)$ for $\eta$.

We have proved in \cite{Chen-Du-WangR2019a} the following result.
\begin{lemma}\label{L Y-to-rubber}
Suppose that the first absolute marked point in $\Gamma$ is mapped to the untwisted sector $\Y_{[h_1]}=\Y$, i.e. $[h_1]=[1]$. Then
\begin{align*}
[\M_{\Gamma}(\D_0|\Y|\D_\0)^\sim]^\vir
&=\epsilon_{\Gamma,*}\left( \msf{ev}_1^*([\D_0]\cap [\M_\Gamma(\D_0|\Y|\D_\0)]^\vir\right)\\
&=\epsilon_{\Gamma,*}\left( \msf{ev}_1^*([\D_\0]\cap [\M_\Gamma(\D_0|\Y|\D_\0)]^\vir\right)
\end{align*}
where $\epsilon_\Gamma\co \M_\Gamma(\D_0|\Y|\D_\0)\rto \M_\Gamma(\D_0|\Y|\D_\0)^\sim$ is the natural forgetful map.

Under the current circumstance, suppose $\underline\alpha=(\bar\psi^{a_1}([\D_\0]\cup\alpha),\ldots)$ with $\alpha\in H^*(\D_\0)=H^*(\D)$, i.e. $\alpha_1=[\D_\0]\cup\alpha$. Write $\tilde{\underline\alpha}=(\bar\psi^{a_1}\alpha,\ldots)$. Then the invariants
\[
\Big\langle\underline\mu \,\Big|\,  \underline\alpha\,\Big|\,\underline\eta \Big\rangle^{\D_0|\Y|\D_\0}_\Gamma= \Big\langle\underline\mu \,\Big|\, \tilde{\underline\alpha}\,\Big|\,\underline\eta \Big\rangle^{\D_0|\Y|\D_\0,\sim}_\Gamma.
\]
\end{lemma}

We will find out that Lemma \ref{L Y-to-rubber} is not enough to prove Lemma \ref{L equal-local-model}; we need to consider the more general case that $[h_1]$ satisfies $\rho([h_1])=1$. We will generalize Lemma \ref{L Y-to-rubber} to this more general case in Lemma \ref{L Y-to-rubber-general}. Here we first consider  $(\Y_{\D_0,r}|\D_\0)$ in Lemma \ref{L Yr-to-rubber} as the proof of Lemma \ref{L Y-to-rubber-general} is similar to the proof of Lemma \ref{L Yr-to-rubber}.

\begin{lemma}\label{L Yr-to-rubber}
Consider the two maps
\[
\epsilon_{\Gamma(r)}\co \M_{\Gamma(r)}(\Y_{\D_0,r}|\D_\0)\rto \M_\Gamma(\Y),\qq\mathrm{and}\qq \epsilon_\Gamma^\sim\co \M_\Gamma(\D_0|\Y|\D_\0)^\sim \rto\M_\Gamma(\Y).
\]
Suppose that $[h_1]$ in $\Gamma$ satisfies $\rho([h_1])=1$. Then
\[
\epsilon_{\Gamma(r),*}\left(\msf{ev}_1^*([(\D_\0)_{[h_1]}])\cap [\M_{\Gamma(r)}(\Y_{\D_0,r}|\D_\0)]^\vir\right)
\]
is a polynomial in $r$ when $r\gg 1$, and
\[
\left[\epsilon_{\Gamma(r),*}\left(\msf{ev}_1^*([(\D_\0)_{[ h_1]}])\cap [\M_{\Gamma(r)}(\Y_{\D_0,r}|\D_\0)]^\vir\right)\right]_{r^0}= \epsilon_{\Gamma,*}^\sim\left( [\M_{\Gamma}(\D_0|\Y|\D_\0)^\sim]^\vir\right).
\]
\end{lemma}
\begin{proof}
As in \S \ref{SS localization}, by virtual localization we have
\begin{align*}
&\msf{ev}_1^*([(\D_\0)_{[h_1]}])\cap [\M_{\Gamma(r)}(\Y_{\D_0,r}|\D_\0)]^\vir
=\sum_{\Phi}\frac{1}{|\aut(\Phi)| }\cdot
i_*\left((-\msf{ev}_1^*(e_{[h_1]}^*c_1(\L))-t) \cdot\frac{[\M_{\Phi}]^\vir}{e(\mc N_\Phi)} \right)
\end{align*}
where $e_{[h_1]}:(\D_\0)_{[h_1]}\rto \D$ is the natural evaluation map from twisted sector to non-twisted sector, and $\Phi$ are bipartite graphs of the forms in \S\ref{SSS localize-graphs} with {\em the $0$-labeling and $\0$-labeling exchanged}. Moreover, when $r\gg 1$, there are also only two type of unstable vertices, which are type (iv) and type (v) in \S\ref{SSS unstable-vertex} with $0$-labeling and $\0$-labeling exchanged.

As the first absolute marking is mapped to $(\D_\0)_{[h_1]}$, the target must be expanded. So for the localization formula we only need to consider the following two types of graphs with expanded targets:
\begin{itemize}
\item[(1)] {\bf Type I}. Those have no stable vertex labeled by $0$ (i.e. over $(\sqrt[r]{\D_0})_\rho$), but one stable vertex labeled by $\0$ (i.e. over rubber). Such a graph corresponds to the first type graphs in Remark \ref{R 3-Phi's}. So by the proof Lemma \ref{L No-edge-for-gi} we see that there is only one graph of this type, which has exactly $n$ (the number of relative marked points in $\mu$) unstable vertices corresponding to absolute marked points decorated by $\vec\mu^r$ and no other unstable vertices corresponding to absolute marked points decorated by (parts of) $\vec g$. So as in Remark \ref{R 3-Phi's} we denote this graph by $\Phi_\0$ (recall that here the labelling $0$ and $\0$ are exchanged comparing with the labelling of graphs in \S \ref{SS localization}).
    Moreover, the topological data for the rubber component of this graph is the same as $\Gamma$.

\item[(2)] {\bf Type II}. Those have stable vertices labeled by $0$ (i.e. over $(\sqrt[r]{\D_0})_\rho$), and one stable vertex labeled by $\0$ (i.e. over rubber), denoted by $v_\0$. We denote such a graph by $\Phi$. These graphs correspond to the third type graphs in Remark \ref{R 3-Phi's} with $0$-labeling and $\0$-labeling exchanged.
\end{itemize}
Then the contributions of these graphs are as follows.
\begin{itemize}
\item The contribution of the only one type I graph $\Phi_\0$ is
\[
\frac{-\msf{ev}_1^*(e_{[h_1]}^*c_1(\L))-t}{-t+\Psi_\0}\cap [\M_{\Gamma}^\sim]^\vir= \frac{\frac{r}{s}\msf{ev}_1^*( e_{[h_1]}^*c_1(\L))+1}{1-\frac{r\Psi_\0}{s}}\cap [\M_{\Gamma}^\sim]^\vir.
\]
\item The contribution of a type II graph $\Phi$ is
\begin{align*}
&\left(\frac{1}{|\aut(\Phi)|} \frac{-\msf{ev}_1^*(e_{[h_1]}^*c_1(\L))-t}{-t+\Psi_\0} \left(\frac{1}{t}\right)^{|\tV_{\mathrm{st}}^0(\Phi)|} \prod_{v\in \tV^0_{\mathrm{st}}(\Phi)}\widehat{\cont}_v\right)\cap [\M_{\Phi}'']^\vir\\
&=\left(\frac{1}{|\aut(\Phi)|} \frac{\frac{r}{s}\msf{ev}_1^*(e_{[h_1]}^*c_1(\L))+1} {1-\frac{r\Psi_\0}{s}} \left(\frac{r}{s}\right)^{|\tV_{\mathrm{st}}^0(\Phi)|} \prod_{v\in \tV^0_{\mathrm{st}}(\Phi)}\widehat{\cont}_v\right)\cap [\M_{\Phi}'']^\vir,
\end{align*}
where $\M_\Phi''$ is similar to $\M_\Phi'$ and is given by
\[
\M_\Phi''=(\prod_{v\in\tV_{\mathrm{st}}^0(\Phi)} \M_v((\sqrt[r]{\D_0})_\rho))
\times_{(\msf{ID}_\0)^{|\Phi|}} \M_{v_\0}^\sim,
\]
with $|\Phi|=\sum_{v\in\tV_{\mathrm{st}}^0(\Phi)}|\tE(v)|$, and $\widehat{\cont}_v$ is similar to $\widetilde{\cont}_v$ and is given by
\begin{align}\label{E hat-Cont}
&t\cdot r^{|\tE(v)|} \tilde c_{\mathrm{rk}(v)}(-(\sqrt[r]\L)_v\otimes\mc O(\frac{1}{r}))\cdot \prod_{e\in\tE(v)}\frac{\kappa_e}{t+ev_e^*(c_1(\L))-\kappa_e \bar\psi_e}\\
=\,&\kappa_v \sum_{0\leq d\leq \mathrm{rk}(v)}(s)^{\msf g(v)-d}\cdot r^{2d-2\msf g(v)+1} \tilde c_d(-(\sqrt[r]\L)_v)\cdot
\prod_{e\in\tE(v)}\frac{1}{1+\frac{r(ev_e^*(c_1(\L))-\kappa_e \bar\psi_e)}{s}}.\nonumber
\end{align}
\end{itemize}
with $(\sqrt[r]\L)_v=\mc R\pi_*\f^*\sqrt[r]\L$ over $\tilde\M_v((\sqrt[r]\D)_\rho)$.

Now we push forward these contributions to $\M_\Gamma(\Y)$ via $\epsilon_{\Gamma(r)}$, then $\M_v((\sqrt[r]\D)_\rho)$ is pushed forward to $\M_v(\D_0)$. As
\begin{align*}
&\epsilon_*(r^{2d-2\msf g(v)+1} \tilde c_d(-(\sqrt[r]\L)_v)
\cap[\M_{v}((\sqrt[r]\D)_\rho)]^\vir)
=\tau_*(r^{2d-2\msf g(v)+1} c_d(-(\sqrt[r]\L)_v)\cap [\tilde\M_{v}((\sqrt[r]\D)_\rho)]^\vir)
\end{align*}
and by Theorem \ref{T polynomiality} $\tau_*(r^{2d-2\msf g(v)+1} c_d(-(\sqrt[r]\L)_v) \cap[\tilde\M_{v}((\sqrt[r]\D)_\rho)]^\vir)$ is a polynomial in $r$ when $r\gg 1$, therefore the contribution of every type II graph $\Phi$ is a polynomial in $r$ when $r\gg 1$, and its lowest degree of $r$ is the number of stable vertices over $0$. So the contribution of every type II graph is a polynomial in $r$ with lowest degree at least one. On the other hand, the only one type I graph $\Phi_\0$ also contributes a polynomial in $r$ with lowest degree zero. Hence $\epsilon_{\Gamma(r),*}\left(\msf{ev}_1^*([(\D_\0)_{[h_1]}]\cap [\M_{\Gamma(r)}(\Y_{\D_0,r}|\D_\0)]^\vir\right)$ is a polynomial in $r$ when $r\gg 1$. Moreover, its constant term corresponds to the constant term of
\begin{align*}
\epsilon_{\Gamma(r),*}\left(\frac{\msf{ev}_1^*(\frac{r}{s} e_{[h_1]}^*c_1(\L))+1}{1-\frac{r\Psi_\0}{s}}\cap [\M_\Gamma^\sim]^\vir\right),
\end{align*}
which is $ \epsilon^\sim_{\Gamma,*} \left([\M_\Gamma(\D_0|\Y|\D_\0)^\sim]^\vir\right)$.
\end{proof}

Consequently,
\begin{coro}\label{C}
Suppose $\underline\alpha= (\bar\psi^{a_1}([(\D_\0)_{[h_1]}]\cup\alpha),\ldots)$ with $\rho([h_1])=1$ and $\alpha\in H^*((\D_\0)_{[h_1]})=H^*(\D_{[h_1]})$. Write $\tilde{\underline\alpha}=(\bar\psi^{a_1}\alpha,\ldots)$. Then $\Big\langle\underline\alpha,\underline\mu\,\Big|\, \underline{\eta}\Big\rangle^{\Y_{\D_0,r}|\D_\0}_{\Gamma(r)}$ is a polynomial in $r$ when $r\gg 1$, and
\[
\left[\Big\langle\underline\alpha,\underline\mu\,\Big|\, \underline{\eta}\Big\rangle^{\Y_{\D_0,r}|\D_\0}_{\Gamma(r)} \right]_{r^0}=\Big\langle\underline\mu \,\Big|\, \tilde{\underline\alpha}\,\Big|\, \underline\eta \Big\rangle_\Gamma^{\D_0|\Y|\D_\0,\sim}.
\]
\end{coro}
By a similar proof we could generalize Lemma \ref{L Y-to-rubber} to the more general case as in Lemma \ref{L Yr-to-rubber} without using the polynomiality in Theorem \ref{T polynomiality}.

\begin{lemma}\label{L Y-to-rubber-general}
Suppose that $[h_1]$ in $\Gamma$ satisfies $\rho([h_1])=1$. Then
\[
\epsilon_{\Gamma,*}\left(\msf{ev}_1^*([(\D_\0)_{[h_1]}])\cap [\M_\Gamma(\D_\0|\Y|\D_0)]^\vir\right)
=[\M_\Gamma(\D_0|\Y|\D_\0)^\sim]^\vir.
\]
So for the $\underline\alpha$ in last corollary we have
\[
\Big\langle \underline\mu \,\Big|\,  \underline\alpha\,\Big|\, \underline\eta \Big\rangle^{\D_0|\Y|\D_\0}_\Gamma= \Big\langle \underline\mu \,\Big|\, \tilde{\underline\alpha}\,\Big|\,\underline\eta \Big\rangle^{\D_0|\Y|\D_\0,\sim}_\Gamma.
\]
\end{lemma}

\subsubsection{Theorem \ref{Thm abs-rel} in the local model}
Now we prove Lemma \ref{L equal-local-model}, hence finish the proof of Theorem \ref{Thm abs-rel}. We split the proof into several cases.
\begin{itemize}
\item[{\bf Case 1.}] {\bf In $\underline{\alpha}$ this is no term of the form $\bar\psi^a([\D_0]\cdot\alpha)$ or $\bar\psi^a([\D_\0]\cdot\alpha)$.}
    For this case, we have two subcases.
    \begin{itemize}
    \item[{\bf Case 1.1.}] $\beta\cdot[\D_\0]\neq 0$. Then we can use divisor equation to add an insertion of the form $\bar\psi^0[\D_\0]$ to $\underline\alpha$, and then apply the above reductions to rubber invariants as follows. For relative orbifold GW-invariants of $(\D_0|\Y|\D_\0)$ we add a smooth marked point to $\Gamma$ to get $\Gamma_{[1]}$. So we change $\vec h$ into $\vec h_{[1]}=([1],\vec h)$ and enlarge the insertion $\underline\alpha=(\bar\psi^{a_1}\alpha_1,\ldots, \bar\psi^{a_m}\alpha_m)$ into
        \[
        \underline\alpha_{[1]}=\left(\bar\psi^0[\D_\0], \underline\alpha\right)=\left(\bar\psi^0[\D_\0], \bar\psi^{a_1}\alpha_1,\ldots, \bar\psi^{a_m}\alpha_m\right).
        \]
        Set
        \begin{align*}
        {\underline\alpha}_i:=&\,\left(\bar\psi^{a_1}\alpha_1, \ldots, \bar\psi^{a_{i-1}}\alpha_{i-1}, \bar\psi^{a_i-1}[\D_\0]\cup_{\mathrm{CR}}\alpha_i ,\bar\psi^{a_{i+1}}\alpha_{i+1},\ldots, \bar\psi^{a_m}\alpha_m\right)\\
        =&\,\left(\bar\psi^{a_1}\alpha_1, \ldots, \bar\psi^{a_{i-1}}\alpha_{i-1}, \bar\psi^{a_i-1}([(\D_\0)_{[h_i]}]\cup\alpha_i) , \bar\psi^{a_{i+1}}\alpha_{i+1}, ,\ldots \bar\psi^{a_m}\alpha_m\right),
        \end{align*}
        where ``$\cup_{\mathrm{CR}}$'' means the Chen--Ruan product, and the second equality follows from the facts that
        \begin{align}\label{E CR-product}
        \alpha\cup_{\mathrm{CR}}\beta=\alpha\cup\beta|_{\Y_{[h]}}
        \end{align}
        for $\alpha\in H^*(\Y_{[h]})$ and $\beta\in H^*(\Y)$ (see for example \cite{Tseng2010,Du2016}) and $ [\D_\0]|_{\Y_{[h_i]}}=[(\D_\0)_{[h_i]}]$.
        Similarly we get $\Gamma(r)_{[1]}$ by adding a smooth marked point, which corresponds to $\Gamma_{[1]}$ under the convention in \S \ref{SSS aogw-Xr}.
        The divisor equation takes the form
        \begin{align}\label{E Divisor-eqn-Y}
        \Big\langle\underline\mu\,\Big|\, {\underline\alpha}_{[1]}\,\Big|\,\underline\eta \Big\rangle_{\Gamma_{[1]}}^{\D_0|\Y|\D_\0} =\int_\beta[\D_\0]\cdot\Big\langle\underline\mu \,\Big|\, \underline\alpha\,\Big|\, \underline\eta\Big\rangle_\Gamma^{\D_\0|\Y|\D_0}+
        \sum_{j=1}^m\Big\langle\underline\mu \,\Big|\,{\underline\alpha}_i \,\Big|\,\underline\eta\rangle_{\Gamma}^{\D_0|\Y|\D_\0}.
        \end{align}
        For the relative invariants of $(\Y_{\D_0,r}|\D_\0)$ we have
        \begin{align}\label{E Divisor-eqn-Yr}
        \Big\langle {\underline\alpha}_{[1]},\underline\mu\,\Big|\, \underline\eta \Big\rangle_{\Gamma(r)_{[1]}}^{\Y_{\D_0,r}|\D_\0} =\int_\beta[\D_\0]\cdot\Big\langle \underline\alpha, \underline\mu\,\Big|\,\underline\eta \rangle_{\Gamma(r)}^{\Y_{\D_0,r}|\D_\0} +
        \sum_{j=2}^n \Big\langle{\underline\alpha}_i,\underline\mu \,\Big|\,\underline\eta \Big\rangle_{\Gamma(r)}^{\Y_{\D_0,r}|\D_\0} .
        \end{align}
        Note that for \eqref{E Divisor-eqn-Yr} we have used
        \[
        [\D_\0]\cup_{\mathrm{CR}}\theta_i =[\D_\0]|_{\Y_{[g_i]}}\cup\theta_i =[(\D_\0)_{[g_i]}]\cup\theta_i=0
        \]
        for $\theta_i$ in $\underline\mu$ as $\theta_i$ supports over $(\sqrt[r]{\D_0})_\rho$ and $[\D_\0]$ supports over $\D_\0$. The divisor equations for both relative invariants of $(\D_0|\Y|\D_\0)$ and $(\Y_{\D_0,r}|\D_0)$ are of the same form. From Corollary \ref{C} we see that
        $\Big\langle {\underline\alpha}_{[1]},\underline\mu\,\Big|\, \underline\eta \Big\rangle_{\Gamma(r)_{[1]}}^{\Y_{\D_0,r}|\D_\0}$ is a polynomial in $r$ when $r\gg 1$. Moreover by Corollary \ref{C} $\Big\langle{\underline\alpha}_i,\underline\mu \,\Big|\,\underline\eta \Big\rangle_{\Gamma(r)}^{\Y_{\D_0,r}|\D_\0}$ is also a polynomial in $r$ when $r\gg 1$ as for every $i$ we have $\rho([h_i])=1$ and
        \[
        \underline{\alpha}_i=\left(\ldots, \bar\psi^{a_i-1}([(\D_\0)_{[h_i]}]\cup\alpha_i) ,\ldots\right).
        \]
        So by \eqref{E Divisor-eqn-Yr} we see that
        \[
        \Big\langle \underline\alpha, \underline\mu\,\Big|\, \underline\eta \Big\rangle_{\Gamma(r)}^{\Y_{\D_0,r}|\D_\0}
        \]
        is a polynomial in $r$ when $r\gg 1$, and by Corollary \ref{C} and Lemma \ref{L Y-to-rubber-general} we have
        \[
        \left[\Big\langle \underline\alpha, \underline\mu\,\Big|\, \underline\eta \Big\rangle_{\Gamma(r)}^{\Y_{\D_0,r}|\D_\0}\right]_{r^0} =\Big\langle\underline\mu\,\Big|\, \underline\alpha\,\Big|\, \underline\eta\Big\rangle_\Gamma^{\D_0|\Y|\D_\0}.
        \]
        So Lemma \ref{L equal-local-model} holds for this case.

    \item[{\bf Case 1.2.}] $\beta\cdot[\D_\0]=0$. Suppose that there is at least one absolute marked point in $\Gamma$. Then by virtual localization, the virtual dimension of the $\cplane^*$-fixed loci is one less than the virtual dimension of $\M_\Gamma(\D_0|\Y|\D_\0)$. Hence the invariant is zero. For the corresponding relative invariant of $(\Y_{\D_0,r}|\D_\0)$, by using the localization computation in Lemma \ref{L Yr-to-rubber} we see that for every graph $\Phi$ the target is not expanded, since $\beta\cdot[\D_\0]=0$, hence the contributions all come from stable vertex over $(\sqrt[r]{\D_0})_\rho$. Then the result is a polynomial in $r$ of degree at least $1$ when $r\gg 1$. Hence the constant term is $0$. So Lemma \ref{L equal-local-model} holds also for this case when there is at least one absolute marked point. We next consider the case that there is no absolute marked point. We could  always choose an $H\in H^2(\D_\0)$ such that $\beta\cdot H\neq 0$. Then by divisor equation it can be reduced to former case as \eqref{E Divisor-eqn-Y} and \eqref{E Divisor-eqn-Yr}. Then by Corollary \ref{C} and the former case, we see that Lemma \ref{L equal-local-model} holds for this case.
    \end{itemize}
\item[{\bf Case 2.}]  {\bf In $\underline\alpha$ there is an insertion of the form $\bar\psi^a([\D_\0]\cdot\alpha)$.}   Then for $\Gamma(r)$ of $(\Y_{\D_0,r}|\D_\0)$, there is also a smooth absolute marked point with constraint $\bar\psi^a([\D_\0]\cdot\alpha)$. Then by Corollary \ref{C} and Lemma \ref{L Y-to-rubber-general}
\[
\Big\langle \underline\alpha,\underline\mu \,\Big|\,\underline\eta \Big\rangle^{\Y_{\D_0,r}|\D_\0}_{\Gamma(r)}
\]
is a polynomial in $r$ when $r\gg 1$, and
\begin{align*}
\left[\Big\langle\underline\alpha,\underline\mu\,\Big|\, \underline\eta\Big\rangle^{\Y_{\D_0,r}|\D_\0}_{\Gamma(r)} \right]_{r^0}&= \Big\langle\underline\mu \,\Big|\, \tilde{\underline\alpha} \,\Big|\, \underline\eta \Big\rangle^{\D_0|\Y|\D_\0,\sim}_\Gamma
=\Big\langle\underline\mu \,\Big|\, \underline\alpha \,\Big|\, \underline\eta\Big\rangle^{\D_0|\Y|\D_\0}_\Gamma.
\end{align*}

\item[{\bf Case 3.}] {\bf There is an insertion of the form $\bar\psi^a([\D_0]\cdot\alpha)$ in $\underline\alpha$}. Then by using
    \[
    [\D_0]=[\D_\0]+c_1(\L)
    \]
    we reduces this case to {\bf Case 1} and {\bf Case 2}. So Lemma \ref{L equal-local-model} holds for this case too.
\end{itemize}
This finishes the proof of Lemma \ref{L equal-local-model}, hence completes the proof of Theorem \ref{Thm abs-rel}.

\subsection{Genus zero case of Theorem \ref{Thm abs-rel}}
We next take a closer look at the genus zero case of Theorem \ref{Thm abs-rel}. When genus $\msf g=0$, we have an improvement of Lemma \ref{L Yr-to-rubber}.
\begin{lemma}\label{L Yr-to-rubber-g=0}
Consider the two maps
\[
\epsilon_{\Gamma(r)}\co\M_{\Gamma(r)}(\Y_{\D_0,r}|\D_\0)\rto \M_\Gamma(\Y),\qq\mathrm{and}\qq \epsilon_\Gamma^\sim\co \M_\Gamma(\D_0|\Y|\D_\0)^\sim\rto\M_\Gamma(\Y).
\]
Suppose that $[h_1]$ in $\Gamma$ satisfies $\rho([h_1]) =0$, and the genus $\msf g=0$ in $\Gamma$. Then
\[
\epsilon_{\Gamma(r),*}\left(\msf{ev}_1^*([(\D_\0)_{[h_1]}])\cap [\M_{\Gamma(r)}(\Y_{\D_0,r}|\D_\0)]^\vir\right)
\]
is a constant in $r$ when $r\gg 1$, and
\[
\epsilon_{\Gamma(r),*}\left(\msf{ev}_1^*([(\D_\0)_{[h_1]}])\cap [\M_{\Gamma(r)}(\Y_{\D_0,r}|\D_\0)]^\vir\right)= \epsilon^\sim_{\Gamma,*}\left( [\M_{\Gamma}(\D_0|\Y|\D_\0)^\sim]^\vir\right).
\]
Therefore for $\underline\alpha= (\bar\psi^{a_1}([(\D_\0)_{[h_1]}]\cup\alpha),\ldots)$, by setting $\tilde{\underline\alpha}=(\bar\psi^{a_1}\alpha,\ldots)$, we have
\[\Big\langle\underline\alpha,\underline\mu\,\Big|\, \underline{\eta}\rangle^{\Y_{\D_0,r}|\D_\0}_{\Gamma(r)} =\Big\langle\underline\mu \,\Big|\, \tilde{\underline\alpha}\,\Big|\, \underline\eta \Big\rangle_\Gamma^{\D_0|\Y|\D_\0,\sim}
\]
when $r\gg 1$.
\end{lemma}
\begin{proof}
Following the proof of Lemma \ref{L Yr-to-rubber}, we have
\begin{align*}
&\msf{ev}_1^*([(\D_\0)_{[h_1]}])\cap [\M_{\Gamma(r)}(\Y_{\D_0,r}|\D_\0)]^\vir
=\sum_{\Phi}\frac{1}{|\aut(\Phi)|}\cdot
i_*\left((-\msf{ev}_1^*(e_{[h_1]}^*c_1(\L))-t)\cdot \frac{[\M_{\Phi}]^\vir}{e(\mc N_\Phi)} \right).
\end{align*}
As the proof of Lemma \ref{L Yr-to-rubber}, since the first absolute marking has insertion $[(\D_\0)_{[h_1]}]$, the target must expand. So we also only have to consider two possible types of graphs as in the proof or Lemma \ref{L Yr-to-rubber}.
\begin{itemize}
\item For the unique Type I graph $\Phi_\0$, the contribution is
    \[
\frac{-\msf{ev}_1^*(e_{[h_1]}^*c_1(\L))-t}{-t+\Psi_\0}\cap [\M_{\Gamma}^\sim]^\vir.
\]

\item For a type II graph $\Phi$, the contribution is
\begin{align*}
&\frac{1}{|\aut(\Phi)|} \left(\frac{-\msf{ev}_1^*(e_{[h_1]}^*c_1(\L))-t}{-t+\Psi_\0} \left(\frac{1}{t}\right)^{|\tV_{\mathrm{st}}^0(\Phi)|} \prod_{v\in \tV^0_{\mathrm{st}}(\Phi)}\widehat{\cont}_v\right)\cap [\M_{\Phi}']^\vir\\
&=\frac{1}{|\aut(\Phi)|} \Bigg\{ \frac{\msf{ev}_1^*(e_{[h_1]}^*c_1(\L))+t} {t-\Psi_\0} \cdot 
\prod_{v\in \tV^0_{\mathrm{st}}(\Phi)} \Bigg[ \sum_{0\leq d\leq |\tE(v)|-1}\tilde c_d(-(\sqrt[r]\L)_v) \left(\frac{t}{r}\right)^{|\tE(v)|-1-d}\\
&
\qq\qq\qq\qq\qq\qq\qq\qq\qq\qq\qq  \cdot r^{|\tE(v)|}\cdot \prod_{e\in\tE(v)}\frac{\kappa_e}{t+ev_e^*(c_1(\L))-\kappa_e \bar\psi_e}\Bigg]\Bigg\}\cap [\M_{\Phi}'']^\vir\\
&=\frac{1}{|\aut(\Phi)|} \Bigg\{ \frac{\msf{ev}_1^*(e_{[h_1]}^*c_1(\L))+t} {t-\Psi_\0} \cdot \prod_{v\in \tV^0_{\mathrm{st}}(\Phi)} \Bigg[\frac{r}{t} \sum_{0\leq d\leq |\tE(v)|-1}\tilde c_d(-(\sqrt[r]\L)_v)\left(\frac{t}{r}\right)^{-d}
\\&
\qq\qq\qq\qq\qq\qq\qq\qq\qq\qq\qq \cdot \prod_{e\in\tE(v)} \frac{\kappa_e}{1+\frac{ev_e^*(c_1(\L))-\kappa_e \bar\psi_e}{t}}\Bigg]\Bigg\}\cap [\M_{\Phi}'']^\vir.
\end{align*}

\end{itemize}
Therefore we always have a factor $\frac{\msf{ev}_1^*(e_{[h_1]}^*c_1(\L))+t} {t-\Psi_\0}$. On the other hand, for a type II graph $\Phi$, each stable vertex contributes a $t\inv$. Now we push them forward to $\M_\Gamma(\Y)$. We need to extract the coefficient of $t^0$. So we only need to consider the unique type I graph $\Phi_0$.
Then we see that the coefficient of $t^0$ does not depend on $r$ and is exactly the $\epsilon^\sim_{\Gamma,*}\left( [\M_\Gamma(\D_0|\Y|\D_\0)^\sim]^\vir\right)$.
\end{proof}

Then by the proof of Theorem \ref{Thm abs-rel} we have the following theorem.
\begin{theorem}\label{Thm abs-rel-g=0}
Suppose $\D$ is a quotient orbifold of a smooth quasi-projective scheme by a linear algebraic group. For genus $0$ invariants, when $r\gg 1$, we have
\[
\Big\langle \underline\alpha,\underline\mu\Big\rangle^{\X_{\D,r}}_{\Gamma(r)} =\Big\langle \underline\alpha\,\Big|\, \underline\mu\Big\rangle^{\X|\D}_{\Gamma}.
\]
\end{theorem}

\bibliographystyle{abbrv}

\bibliography{chengyong}

\end{document}